\documentclass[10pt]{amsart}
\usepackage{fullpage}
\usepackage{amssymb,amsmath,amsthm,amscd,mathrsfs,graphicx, color}
\usepackage[cmtip,all,matrix,arrow,tips,curve]{xy}
\usepackage[final]{pdfpages}

\numberwithin{equation}{section}
\theoremstyle{plain}
\newtheorem{theorem}[equation]{Theorem}
\newtheorem*{mainthm}{Main Theorem}
\newtheorem*{claim*}{Claim}
\newtheorem{lemma}[equation]{Lemma}
\newtheorem{sublemma}[equation]{Sublemma}
\newtheorem{claim}[equation]{Claim}
\newtheorem{prop}[equation]{Proposition}
\newtheorem{cor}[equation]{Corollary}

\theoremstyle{remark}
\newtheorem{remark}[equation]{Remark}

\theoremstyle{definition}
\newtheorem{definition}[equation]{Definition}
\newtheorem{notation}[equation]{Notation}

\newtheorem{example}[equation]{Example}

\newcommand{\cJ}{{\mathcal J}}
\newcommand{\ocJ}{{\overline{\mathcal J}}}
\newcommand{\cY}{{\mathcal Y}}

\def\Pic{\operatorname{Pic}}
\def\Spec{\operatorname{Spec}}

\def\Hom{\operatorname{Hom}}
\def\Ext{\operatorname{Ext}}

\newcommand{\bP}{\mathbb{P}}

\newcommand{\sF}{\mathscr{F}}

\newcommand{\bC}{\mathbb{C}}

\newcommand{\calA}{\mathcal{A}}
\newcommand{\calC}{\mathcal{C}}
\newcommand{\calF}{\mathcal{F}}
\newcommand{\calM}{\mathcal{M}}
\newcommand{\calL}{\mathcal{L}}
\newcommand{\calP}{\mathcal{P}}

\newcommand{\sD}{\mathscr{D}}

\newcommand{\im}{\operatorname{Im}}

\newcommand{\id}{\mathrm{id}}

\newcommand{\Hyp}{\mathcal Hyp}

\newcommand{\bV}{{\mathbf V}}

\newcommand{\ov}{\overline}
\newcommand{\wt}{\widetilde}
\newcommand{\J}[1]{\ov \Jac{(#1)}}
\newcommand{\ovPrym}[1]{\ov{\operatorname{Prym}}{(#1)}}
\newcommand{\mc}{\mathcal}
\def\Shom{\mc {H}{om}}
\def\Shend{\mc {E}{nd}}
\def\Shext{\mc {E}{xt}}
\newcommand{\wtC}{\wt{\mathcal C}}

\DeclareMathOperator{\Def}{Def}

\DeclareMathOperator{\PGL}{PGL}

\DeclareMathOperator{\Hilb}{Hilb}
\DeclareMathOperator{\Proj}{Proj}
\DeclareMathOperator{\Sing}{Sing}
\DeclareMathOperator{\Jac}{Jac}
\DeclareMathOperator{\Prym}{Prym}
\DeclareMathOperator{\Fix}{Fix}
\DeclareMathOperator{\Nm}{Nm}

\newcommand{\be}{\begin{equation}}
\newcommand{\ee}{\end{equation}}

\begin{document}

 \title[Relative compactified Intermediate Jacobians]{A hyper-K\"ahler compactification of the Intermediate Jacobian fibration associated to a cubic fourfold}

 \author[R. Laza]{Radu Laza}
\address{Stony Brook University,  Stony Brook, NY 11794, USA}
\email{radu.laza@stonybrook.edu}

 \author[G. Sacc\`a]{Giulia Sacc\`a}
\address{Stony Brook University,  Stony Brook, NY 11794, USA}
\email{giulia.sacca@stonybrook.edu}

\author[C. Voisin]{Claire Voisin}
\address{Coll\`ege de France, 3 rue d'Ulm, 75005 Paris, France}
\email{claire.voisin@imj-prg.fr}

\begin{abstract}
Let $X$ be a general cubic fourfold. It was observed by Donagi and Markman that the relative intermediate Jacobian fibration  $\cJ_U/U$ (with $U=(\bP^5)^\vee\setminus X^\vee$) associated to the smooth hyperplane sections of $X$ carries a natural holomorphic symplectic form making the fibration Lagrangian. In this paper, we obtain a smooth projective compactification $\overline \cJ$ of $\cJ_U$  with the property that the holomorphic symplectic form on $\cJ_U$ extends to a holomorphic symplectic form on $\overline \cJ$. In particular, $\overline \cJ$ is a $10$-dimensional compact hyper-K\"ahler manifold, which we show to be deformation equivalent to the exceptional example of O'Grady. This proves a conjecture by Kuznetsov and Markushevich.
\end{abstract}

\maketitle

\bibliographystyle{amsalpha}

\section*{Introduction}
It is a problem of significant interest to construct and classify   compact hyper-K\"ahler (HK) manifolds.
 In dimension $2$, the HK manifolds are $K3$ surfaces. All known higher dimensional
  examples are obtained from $K3$s or abelian varieties, by a moduli construction and
   a deformation. Specifically, Beauville \cite{beauville} has
   given two series of examples: the Hilbert scheme of points on $K3$s,
   and respectively generalized Kummer varieties. Mukai \cite{mukai} has
    given a more general construction, namely he has shown that the moduli
    space of semi-stable sheaves on $K3$s carries a symplectic form, and thus is a
     HK manifold if it is smooth. Unfortunately, these examples are deformation
     equivalent to those of Beauville (\cite{yoshioka}). Starting from a singular
      moduli space of sheaves on $K3$s, O'Grady \cite{og10,og6} has produced two
      genuinely new examples: a $10$-dimensional and a $6$-dimensional one, that we
       call $OG10$ and $OG6$ respectively. It was subsequently
       verified (\cite{kls}) that these are the only two new examples that can be obtained by this method.

It is natural to expect that hyper-K\"ahler manifolds can be constructed from lower dimensional
 objects of similar nature. Specifically, we recall that the moduli space of polarized HK manifolds
  in a fixed deformation class is birational via the period map to a locally symmetric variety $\sD/\Gamma$,
  where $\sD$ is a Type IV domain, and $\Gamma$ an arithmetic group. Such a locally symmetric variety contains
  divisors $\sD'/\Gamma'$ (known as Noether--Lefschetz or Heegner divisors) which are of the same type.
   It is natural to expect that some of these NL divisors are associated (at least motivically) to moduli
   spaces of lower dimensional HK manifolds. For instance, the Beauville--Mukai construction will
   give such an example (i.e. $\Hilb^n(K3)$ form a NL divisor in the corresponding moduli space of
   $2n$-dimensional HK manifolds). Another NL divisor in the moduli space,
that we will call a Lagrangian NL divisor, is obtained by considering the HK manifolds that admit
 a Lagrangian fibration. Again, one can hope that they can be constructed from lower dimensional
 geometric objects.

 As polarized $K3$ surfaces have only $19$ parameters,
 Hodge theory and abstract deformation theory show that many
       hyper-K\"ahler varieties built from polarized $K3$ surfaces
       have projective deformations which are unrelated to $K3$
       surfaces. However, the problem of constructing explicit
       projective models for these deformations is usually hard, one reason being
        the fact that most of these
       deformation spaces are of general type (\cite{ghs}). Cubic fourfolds have $20$ moduli and they are well-known
        to have a Hodge structure with Hodge numbers $h^{3,1}=1,\,h^{2,2}_{prim}=20$. In several
       instances, cubic fourfolds have been used to provide via an auxiliary construction such an
       algebro-geometric
       deformation  of Hilbert schemes of $K3$ surfaces. For instance, one of the key papers in the field,  \cite{bedo} shows that the Fano variety of lines of a
       cubic $X$ is a deformation of a ${\rm Hilb}^2(K3)$. More recently,   \cite{llss} constructs a HK manifold from the variety of cubic rational curves
       in $X$, which  is then shown in \cite{addle} to be
 deformation equivalent to a ${\rm Hilb}^4(K3)$.

 The $10$-dimensional examples by O'Grady have $b_2=24$, which means that
 polarized deformations of them have $21$ moduli, and complete families of such varieties
 with Picard number $2$ have $20$ moduli. The construction by O'Grady in \cite{og10} provides (infinitely many)
 $19$-parameters families of such examples with Picard number $3$, parametrized by any moduli
 space of
 polarized $K3$ surfaces.
 This paper provides an algebro-geometric realization of the moduli space of cubic fourfolds as a $20$-dimensional
 moduli space of deformations of O'Grady's $10$-dimensional examples, and more precisely, as
 a Lagrangian NL divisor in the larger $21$-dimensional deformation space (our varieties are canonically polarized).  Note first that this embedding is a priori given
 by lattice considerations. Indeed,  the moduli space of cubic fourfolds
 is birational to a $20$-dimensional locally symmetric variety $\sD'/\Gamma'$, which is associated to the lattice  $A_2\oplus E_8^2\oplus U^2$ (where $A_2$ and $E_8$ are the standard positive definite root lattices, and $U$ is the hyperbolic plane). On the other hand, by work of Rapagnetta \cite{rapagnetta}, it is known that the second cohomology of $OG10$ equipped with the Beauville--Bogomolov form is isometric to the lattice $A_2(-1)\oplus E_8(-1)^2\oplus U^3$. This shows that from an arithmetic point of view, the situation is similar to that of elliptic $K3$s with a section (lattice $E_8(-1)^2\oplus U^2$), aka unigonal $K3$s, versus general $K3$s (lattice $E_8(-1)^2\oplus U^3$). Furthermore, it is not hard to embed the period space of cubic fourfolds as a NL divisor into a $21$-dimensional period domain of {\it polarized} $OG10$ manifolds (for $K3$s, the unigonal $K3$s form a Heegner divisor in any of the polarized period domains).

 Our contribution in this paper is to realize geometrically this abstract
 embedding by actually constructing a family
 of OG10 polarized manifolds parametrized by the moduli space of cubic $4$-folds. This is done by
 realizing the following  program
 that has been started by Donagi--Markman and developed by Markushevich, with  further evidence
 provided
 by work of Kuznetsov  and O'Grady-Rapagnetta: Starting from a general cubic fourfold $X$, one has the universal family  $\cY/B$ of cubic threefolds over $B:=(\bP^5)^\vee$ obtained as hyperplane sections of $X$, and then the associated relative Intermediate Jacobian fibration $\cJ_U/U$, where $U=B\setminus X^\vee\subset B$ is the locus of smooth hyperplane sections. In 1993, Donagi and Markman \cite{DM1,DM2} had the  insight that this fibration, which they showed to be algebraic, carries a holomorphic symplectic  form. The question naturally was raised, as to whether or not it admitted a holomorphic symplectic compactification.  If such a holomorphic symplectic compactification $\overline \cJ/B$ existed, then it would have Picard number at least $2$ and transcendental second cohomology containing the transcendental cohomology of the cubic $X$. Thus, $b_2(\ocJ)\ge 24$, showing in particular  that $\ocJ$ can not be deformation equivalent to $K3^{[5]}$ type, but potentially equivalent to $OG10$. Markushevich and Kuznetsov \cite{Mark-Kuzn} (with further supporting evidence by O'Grady and others) conjectured that indeed a good compactification $\ocJ$ exists, and that it is of $OG10$ type. In this paper, we verify this conjecture. Specifically, we prove

 \begin{mainthm}\label{theomain20sep} (Cf. Theorem \ref{theoconst} and Corollary \ref{cordefoeq}.)
 Let $X\subset \bP^5$ be a general cubic fourfold. Let $B=(\bP^5)^\vee$, $U=(B\setminus X^\vee)$, and $\cJ_U\rightarrow U$ the associated Intermediate Jacobian fibration. Then there exists a smooth projective compactification $\overline{\cJ}\rightarrow B$, which carries a holomorphic symplectic form, with respect to which the fibration is Lagrangian. Furthermore, $\overline{\cJ}$ is deformation equivalent to $OG10$.
 \end{mainthm}

 The key issue for the theorem, and the main new content of the paper, is the construction of a smooth projective compactification of the Intermediate Jacobian fibration that has a nondegenerate holomorphic $2$-form. In order to do so, it is important to  understand degenerations of intermediate Jacobians as the hyperplane section of $X$ becomes singular. The study of intermediate Jacobians from a different perspective was done in \cite{casaetal,casaetal2}, and the main tool used there is Mumford's construction of the intermediate Jacobian as a Prym variety. This is our approach also here.
A key point that allows us to construct the compactification $\ocJ/B$ is the observation that much of the Prym construction goes through for mildly singular cubic threefolds. Namely, we recall that for a smooth cubic threefold $Y$, the projection from a generic line $l$, realizes $Y$ as a conic bundle over $\bP^2$ with discriminant locus a smooth quintic $C$.  Then, Mumford showed that $J(Y)\cong \Prym(\widetilde C/C)$, where $\widetilde C$ is an \'etale double cover of $C$. In \cite{casaetal}, the authors have noted that for mildly singular cubics $Y$ there still exist {\it good lines} (see Def. \ref{defigood}) $l$ such that the  associated quintic $C$ has a $1$-to-$1$ correspondence with the singularities of $Y$ (including the type), and that the covering $\widetilde C/C$ is still \'etale. This reduces questions about degenerations of cubic threefolds to degenerations of curves. In \cite{casaetal2}, the degeneration of the Prym variety $\Prym(\widetilde C/C)$ is studied from the perspective of stable abelian varieties (in the sense of Alexeev).  Here we study this from the perspective of compactified relative Jacobians (or rather compactified Pryms) instead. For curves, one has a very good understanding of the degeneration of Jacobians. Namely, from the perspective of moduli spaces the correct statement is the theorem of Mumford--Namikawa that there exists an extended period map from $\overline \calM_g$ to $\overline\calA_g^{Vor}$ (and thus to a family of DM stable curves one can associate a family of Alexeev stable abelian varieties). If one insists instead on associating to an arbitrary family of curves $\calC/B$ a family of degenerate abelian varieties, without modifying the base, one gets into the theory of compactified relative Jacobians (to pass from the relative compactified Jacobian to the family of stable abelian varieties, one needs to perform a simultaneous semistable reduction, see \cite{cml2} and the references within). The situation for Prym varieties is more complicated, but still well understood (see \cite{FS}, \cite{ABH}, \cite{casaprym}) when the point of view is that of  Beauville's admissible covers compactification; the relative compactified Prym is less studied but a few cases, namely when the families of curves come from linear systems on surfaces, were studied in  \cite{mark_tik}, \cite{ASF}, and \cite{thesisg}. In our situation, due to the existence of very good lines, we can make an important simplifying assumption: namely, we consider only \'etale double covers of planar curves, and furthermore we can assume that both the cover and the base are irreducible.

Finally, we should remark that this construction also provides the first example of hyper-K\"ahler manifolds admitting a Lagrangian fibration in principally polarized abelian varieties that are not Jacobians of curves. Indeed, the known constructions of compact holomorphic symplectic varieties that are fibered in principally polarized abelian varieties are either the relative compactified Jacobian of a linear system on a K3 surface (Beauville--Mukai system), or the relative compactified Prym associated to a linear system on a K3 with an anti--symplectic involution  (\cite{mark_tik}, \cite{ASF}). In these last examples, either the relative Prym varieties are isomorphic to Jacobian of curves, in which case they are hyper-K\"ahler manifolds deformation to the Hilbert scheme of points on a $K3$ surface, or the total space of the family has singularities that admit no symplectic resolution.

Let us briefly describe  the main steps of the paper.

\subsection*{The hyper-K\"ahler structure} In Section 1, we review the construction of the relative Intermediate Jacobian fibration $\cJ_U\rightarrow U$ and the existence of a holomorphic symplectic $2$-form. We then show, without much difficulty, that this symplectic form extends over the locus of cubics with a single ordinary node $\cJ_{U_1}\rightarrow U_1$, providing a nondegenerate holomorphic $2$-form on $\cJ_{U_1}$.  These results are well known, and are essentially completely contained in \cite{DM1}, \cite[\S8.5.2]{DM2}, but the cycle theoretic method used here has the advantage of providing an extension to a  holomorphic closed $2$-form on any smooth compactification $\overline \cJ$
 of $\cJ_{U_1}$. If furthermore, $\cJ_{U_1}\subset \overline \cJ$ has codimension larger than $2$, then the extended form is everywhere non-degenerate. This allows us to concentrate on the problem of constructing a smooth compactification of $\cJ_{U_1}$ which is flat over $B$. The fact that the variety we construct is irreducible holomorphic symplectic  (or hyper-K\"ahler) makes use of \cite{llss}. Indeed, the intermediate Jacobian fibration contains
 a divisor which is birationally a $\mathbb{P}^1$-bundle on the hyper-K\"ahler $8$-fold recently constructed in \cite{llss}.

\subsection*{Local compactification} In Section 2, we briefly recall the Prym construction, and discuss the existence of a good line for all hyperplane sections of a general cubic fourfold.  Roughly speaking the existence of a good line guarantees that the deformation theory of cubic threefolds and their singularities can be identified locally, up to a smooth factor, with the corresponding deformation theory of quintic curves (see \cite[Sect. 3]{casaetal2}).  Studying the degenerations of intermediate Jacobians thus reduces to studying degenerations of abelian varieties associated to curves (more precisely Prym varieties). Some of the results and ideas in this section occur previously in \cite{casaetal,casaetal2}. Here, we obtain a slight strengthening applicable to our context: for $Y$ any hyperplane section of a general cubic fourfold, there exists a very good line $l$, that is, a line such that the associated cover $\widetilde C/C$ is \'etale, and both curves are irreducible (with singularities in $1$-to-$1$ correspondence with those of the cubic threefold $Y$).
With these assumptions,  there exists a (canonical) relative compactified Prym and
  our main result here is Theorem \ref{smoothness properties}  which says that this compactification has the property that the total space is smooth, provided that the family of (base) curves gives a simultaneous versal deformation of the singularities (an analogue of the corresponding result for Jacobians). Moreover, this relative compactified Prym is equidimensional, in particular
 flat over a smooth base. These results are discussed in Section \ref{sectprym}.

 While Section \ref{sectprym} gives a general construction for a smooth compactification of a family of Pryms (under suitable assumptions), the fact that this is applicable to our situation follows from the versality statements of Section   \ref{sectrans18janvier} (e.g. Corollary \ref{corotransdef}).  It is here, in Section  \ref{sectrans18janvier}, that the   generality assumption on the cubic fourfold $X$ is essential. Moreover, for the purpose of proving the deformation equivalence of our compactified fibration $\ocJ/B$ to OG10, we need to show that the versality statements still hold for general Pfaffian cubic fourfolds (\S\ref{sectranpf}).

\subsection*{Descent}
Let $\sF\rightarrow B$ be the relative Fano surface associated to the family $u:\cY\rightarrow B$ of hyperplane sections of a fixed cubic fourfold, and $\sF^0\rightarrow B$ the open subset of very good lines. The results of the previous three sections give that $\sF^0\rightarrow  B$ is a surjective smooth map (with $2$ dimensional fibers), and the existence of a relative compactified Prym $\overline \calP\to \sF^0$ which is proper over $\sF^0$, with smooth total space $\overline \calP$. Our compactification $\overline \cJ$ is a descent of the family $\overline \calP\rightarrow \sF^0$ to $B$. More precisely, in Section \ref{descent section}, using the relative theta divisor and a relative $\Proj$ construction, we conclude that $\overline \calP$ descends to $B$ giving a smooth compactification $\overline \cJ\rightarrow B$ of $\cJ_U\rightarrow U$ extending $\cJ_{U_1}\rightarrow U_1$, equidimensional over $B$. The arguments of Section \ref{sectintjacbun} now allow us to conclude that $\overline \cJ$ is a $10$-dimensional compact HK manifold, and that $\overline \cJ$ is a Lagrangian fibration (whose general fiber is an intermediate Jacobian). This concludes (see Theorem \ref{theoconst}) the proof of the first half of the Main Theorem.

\subsection*{Relationship with O'Grady's varieties} In the final section, we establish that the constructed object is in fact deformation equivalent to the $OG10$ example \cite{og10}
(see Corollary \ref{cordefoeq}). Assuming the existence of a smooth compactification of $\cJ_U$, partial results in this direction had been established by O'Grady-Rapagnetta. Their idea was to use the
degeneration of the cubic fourfold to the chordal cubic (the secant variety of the Veronese
surface), for which one can show that the compactification of the limiting family of intermediate Jacobians
is birational to the O'Grady moduli space for an adequate $K3$ surface. A similar construction had been
done by Hwang and Nagai \cite{hwangnagai} in the case of a singular cubic fourfold, for which the $K3$ surface
is the surface of lines through the singular point. Unfortunately, in both cases, we could not deduce
from the existence of these birational maps the fact that our compactified Jacobian fibration is deformation equivalent to the O'Grady moduli space, because we have no control of the singularities of the compactified Jacobian fibration at these points.
 We study instead, as suggested also by Markushevich and Kuznetsov \cite{Mark-Kuzn}, the Intermediate Jacobian fibration in the case of a Pfaffian cubic fourfold. By Beauville-Donagi \cite{bedo}, such a cubic $X$ is Hodge theoretic equivalent to a degree $14$ K3 surface $S$. Using \cite{ilievmarku},
\cite{markutikho}, and  \cite{kuznetsov} we show that the fibration $\cJ_U$ (or the compact version)  is  birational to the O'Grady moduli space of sheaves on $S$.
On the other hand, using versality statements established in Section \ref{sectrans18janvier}, we are able to prove that the compactified  Jacobian fibration is smooth also for a general Pfaffian cubic fourfold, so that Huybrechts' fundamental theorem \cite{huybrechts} applies, allowing us to conclude the equivalence to the OG10 example.

\subsection*{Two remaining questions}  We point out that there are two families of intermediate Jacobians associated to the family of hyperplane sections of a cubic fourfold $X$: in addition to the $\cJ_U\rightarrow U$ considered here, there is a twisted family $\cJ'_U\rightarrow U$ parametrizing $1$-cycles of degree $1$ (or $2$) in the fibers of
$u:\mathcal{Y}_U\rightarrow U$. In this paper, we are compactifying the untwisted family $\cJ_U$. A natural question, that will be addressed elsewhere, is the existence of a compactification for the twisted case. Here we only note that once a line has been chosen on a cubic threefold, the intermediate Jacobian and the twisted intermediate Jacobian are naturally identified. Thus, up to the descent argument of Section \ref{descent section} everything goes through unchanged.

On a related note, our construction is somewhat indirect (e.g. it involves the auxiliary choice of a line). It is natural to ask if a more direct construction is possible, in particular one wonders if there is a modular construction for our compactification $\ocJ\rightarrow B$ at least up to natural birational modifications (e.g. resolution of symplectic singularities, or contraction of some divisor on which the symplectic form is degenerate). We remind the reader that O'Grady's original construction for OG10 is indeed modular in this sense (it is the resolution of the moduli of sheaves on $K3$s for a specific choice of Mukai vector). Results of \cite{beauvillevb}, \cite{markutikho} in a relative setting
show that after blowing-up $\mathcal{J}_U$ (or rather its twisted version)  along the universal family of lines,
 one gets  a variety which is  a relative moduli space of coherent sheaves supported on smooth hyperplane sections
 of $X$, and it is maybe possible to recontract it to a symplectic moduli space of coherent sheaves on $X$. The compactification of the present paper is birational to a moduli space of sheaves on the Fano variety of lines on $X$ (supported on the Fano surfaces of lines on the hyperplane sections of $X$). A possible strategy to find a modular compactification is therefore to study the smoothness of this moduli space of sheaves.

\subsection*{Acknowledgement} This work was started while the authors were visiting IAS as part of the special program ``Topology of Algebraic Varieties''. We are grateful to IAS for the wonderful research environment. The IAS stay of the first author was partially supported by NSF grant DMS-1128155. He also acknowledges the support through NSF grants DMS-125481 and DMS-1361143. The second named author  acknowledges the support of the Giorgio and Elena Petronio Fellowship Fund II and of NSF grant DMS-1128155. Claire Voisin was a distinguished visiting Professor at IAS during AY 2014/15 and acknowledges the generous support of the Charles Simonyi Fund and of the Fernholz Foundation.

We learned about  this conjecture from a talk by D. Markushevich at the Hyper--K\"ahler Geometry Workshop at the Simons Center for Geometry and Physics in Fall 2012. It was he who initiated the project and subsequently and independently realized many steps towards its completion. The main result of the paper should thus partly be attributed to him. K. O'Grady and A. Rapagnetta  also significantly contributed to the general project.
 We also acknowledge helpful discussions with E. Arbarello, S. Casalaina-Martin, K. Hulek, A. Kuznetsov, E. Macr\`i, L. Migliorini, K. O'Grady, and A. Rapagnetta. Finally, we are grateful to the referee for careful reading and constructive comments that helped us improve the paper.

 \subsection*{Notation and Conventions} Unless otherwise specified, $X\subset \bP^5$ denotes a (Hodge) general cubic fourfold. We let $B=(\bP^5)^\vee$, $U=B\setminus X^\vee$, and $U_1=B\setminus (X^\vee)^{sing}$ parametrizing hyperplane sections with at most one single ordinary node. Thus, $U\subset U_1\subset B$ and $U_1$ has boundary of codimension $2$ in $B$. Let $u:\cY\rightarrow B$ be the universal family of cubic threefolds obtained as hyperplane sections of $X$, and $u:\cY_U\rightarrow U$, $u:\cY_{U_1}\rightarrow U_1$ its restrictions to $U$ and $U_1$ respectively. Typically, $Y$ or $Y_t$ will denote hyperplane sections of $X$ or fibers of $\cY$.  $\pi_U:\cJ_U\rightarrow U$ will denote the intermediate Jacobian fibration
 associated to the family $\cY_U\rightarrow U$ and similarly for $\pi_{U_1}:\cJ_{U_1}\rightarrow U_1$. $\cJ_U$ and $\cJ_{U_1}$ are smooth quasi-projective varieties together with a holomorphic symplectic form, and the restriction of $\cJ_{U_1}$ to $U$ is $\cJ_U$ (cf. \cite{DM1,DM2} and Section \ref{sectintjacbun}). $\ocJ\rightarrow B$ will denote a proper algebraic extension of $\cJ_U\rightarrow U$ (and of  $\cJ_{U_1}\rightarrow U_1$) over $B$. Of course, such $\ocJ$ always exists; the goal of the paper is to establish the existence of a smooth $\ocJ$ which is flat over $B$, hence holomorphically symplectic.

 Let $Y$ be a cubic threefold, and $l\subset Y$ a line (both $Y$ and $l$ need to satisfy some mild assumptions, to be specified in the text). The projection from $l$ realizes $Y$ as a conic bundle over $\bP^2$ with discriminant a plane quintic $C$. The lines in $Y$ incident to $l$  are parametrized by a curve $\widetilde C$, which is a double cover (\'etale for general $l$) of $C$. $\calF\rightarrow B$ will denote the universal family of Fano surfaces, and $(\widetilde \calC,\calC)$ the relative family of double covers over (a suitable open of) $\calF$.

 We will say that two hypersurface singularities $(V(f),0)\subset \bC^n$ and $(V(g),0)\subset \bC^{n+k}$ have the same type if they differ by a suspension, i.e. in suitable analytic coordinates $g(x_1,\dots,x_{n+k})=u\cdot (f(x_1,\dots,x_n)+x_{n+1}^2+\dots+x_{n+k}^2)$, with $u$ a unit in $\mathcal O_{(\bC^{n+k},0)}$. The deformation spaces for singularities of the same type (and also the local monodromies if $k\equiv 0 \pmod{2}$) are naturally identified.

\section{Holomorphic $2$-forms on Jacobian fibrations}\label{sectintjacbun}
In  \cite{DM2,DM1}, Donagi and Markman have performed the infinitesimal study of
algebraically integrable systems, also called Lagrangian fibrations, which consist
of a holomorphic family of complex tori, equipped with a nondegenerate $(2,0)$-form for which the fibers are
Lagrangian.
In this section, we provide an alternative way to construct a structure of Lagrangian or rather  isotropic fibration on
certain families of intermediate Jacobians. That is, we give a cycle and Hodge theoretic argument  to construct
a closed holomorphic $2$-form vanishing on fibers of  such families. The non-degeneracy of the holomorphic $2$-form needs to be checked by hand. However, a strong point of our construction is that it easily implies  that the $(2,0)$-form extends
to any algebraic smooth compactification of the family of intermediate Jacobians.

\subsection{The general case}
Let $X$ be a smooth projective variety of dimension $2k$, and let
$L$ be a line bundle on $X$. Assume that the smooth members
$Y$  of the linear system $|L|$ have the following property:
\begin{eqnarray}
\label{eqvan} H^{p,q}(Y)=0,\,p+q=2k-1,\,(p,q)\not\in \{(k,k-1),(k-1,k)\}.
\end{eqnarray}
Note that by \cite{griffiths} (see also \cite[I,12.1,12.2]{voisinbook}) (\ref{eqvan}) is  implied by the following
property:
\begin{eqnarray*}
 \textrm{\it  ($\star$) The Abel-Jacobi map $\Phi_Y:{\rm CH}^k(Y)_{hom}\rightarrow J^{2k-1}(Y)$ is surjective.}
\end{eqnarray*}

Conversely, (\ref{eqvan}) should imply ($\star$), according to the Hodge conjecture (see \cite[2.2.5]{voisindecomp}).
Under the assumption (\ref{eqvan}), the intermediate Jacobians $J(Y)$ are abelian varieties. Let $U\subset B:= |L|$ be the Zariski open set parametrizing  smooth members, $\mathcal{Y}\subset B\times X$ be the universal family, and $\mathcal{Y}_{U}$ its restriction to $U$. The family of intermediate Jacobians
 is under the same assumption a quasiprojective variety $\mathcal{J}_U$ with a smooth projective morphism
$$\pi_U:\mathcal{J}_U\rightarrow U$$
with fiber $J(Y_t)$ over the point $t\in U$.

Let now $\eta\in H^{k+1,k-1}(X)$ and assume $\eta_{\mid Y}=0$ in $H^{k+1,k-1}(Y)$ for any smooth member
$Y$ of $|L|$. (In our main application, where $X$ is the cubic fourfold and $k=2$, this assumption will be automatic, since $H^{k+1,k-1}(Y_t)$ will be $0$.)
We are  going to construct a holomorphic  $2$-form $\sigma_U$ on $\mathcal{J}_U$ associated with the above data and establish Theorem
\ref{theogen2form}. We will first do this assuming ($\star$) and will explain at the end
how to adapt the proof when they only satisfy property (\ref{eqvan}).

\vspace{0.5cm}

\noindent{\bf Construction of the holomorphic $2$-form.} We denote by $u:\mathcal{Y}_U\rightarrow U$ the first projection, where
$\mathcal{Y}_U\subset U\times X$ is the universal hypersurface.
We have:
\begin{lemma}\label{lepenible} Assuming ($\star$), there exists a codimension $k$ cycle
$$\mathcal{Z}\in {\rm CH}^k(\mathcal{J}_U\times_U \mathcal{Y}_U)_\mathbb{Q}$$
such that the Betti  cohomology class $\alpha:=[\mathcal{Z}]\in H^{2k}(\mathcal{J}_U\times_U \mathcal{Y}_U,\mathbb{Q})$, or rather its image $\alpha_0$ in $H^0(U,R^{2k}(\pi_u,u)_*\mathbb{Q})$,
satisfies the condition that $$\alpha_0^*:R^{2k-1}u_*\mathbb{Q}\rightarrow R^1\pi_{U*}\mathbb{Q}$$
is the natural isomorphism.
\end{lemma}
\begin{proof} By assumption ($\star$), for each fiber $Y_t$, $t\in U$, there exist
a smooth and projective variety
 $W_t$ and a family
of codimension $k$ cycles $\mathcal{T}_t\in {\rm CH}^k( W_t\times Y_t)$,  such that $\mathcal{T}_{t\mid \{w\}\times Y_t}$ is homologically trivial, with the property that
  the Abel-Jacobi map $$\Phi_{\mathcal{T}_t}:W_t\rightarrow J(Y_t),\,\,
  w\in W_t\mapsto \Phi_{Y_t}(\mathcal{T}_{t,w})$$ is surjective. It follows
  that there exists a codimension $k$ cycle $\mathcal{Z}'_t\in {\rm CH}^k(J(Y_t)\times Y_t)$
  such that the Abel-Jacobi map
  $$\Phi_{\mathcal{Z}'_t}:J(Y_t)\rightarrow J(Y_t),\,\,
  w\in J(Y_t)\mapsto \Phi_{Y_t}(\mathcal{T}_{t,w})$$ is $N$ times the identity of
  $J(Y_t)$ for some integer $N>0$. Indeed, we may assume that the surjective
  morphism
  $\Phi_{\mathcal{T}_t}:W_t\rightarrow J(Y_t)$ is generically finite of degree $N$, by replacing $W_t$ by a linear section if necessary.
  Then we set
  $$\mathcal{Z}'_t:=(\Phi_{\mathcal{T}_t},Id_{Y_t})_*\mathcal{T}_t.$$
  The  cycles $\mathcal{Z}'_t$ have been defined fiberwise, but standard argument show that
   for an adequate choice of $N$ they  can be constructed in family over a smooth generically finite cover
  $V$ of $U$ by spreading the original cycles $W_{t,s}$. This provides
  a codimension $k$ cycle
  $\mathcal{Z}''\in {\rm CH}^k(\mathcal{J}_V\times_V\mathcal{Y}_V)$ such
  that
  the class $\alpha'':=[\mathcal{Z}'']$ satisfies
  $${\alpha''_0}^*:R^{2k-1}u'_*\mathbb{Q}\rightarrow R^1\pi'_*\mathbb{Q}$$
is $N$ times the natural isomorphism, where $u':\mathcal{Y}_V\rightarrow V,\,\pi':\mathcal{J}_V\rightarrow V$
are the natural maps.
We can choose a partial smooth completion $\overline{V}$ of $V$ such that the morphism
$V\rightarrow U$ extends to a proper morphism $r:\overline{V}\rightarrow U$. We next extend
the cycle $\mathcal{Z}''$ to a cycle
$\overline{\mathcal{Z}''}\in {\rm CH}^k(\mathcal{J}_{\overline{V}}\times_{\overline{V}}\mathcal{Y}_{\overline{V}})$.
If $M={\rm deg}\,r$, the cycle $\frac{1}{MN}\tilde{r}_* \overline{\mathcal{Z}''}\in {\rm CH}^k(\mathcal{J}_{{U}}\times_{{U}}\mathcal{Y}_{{U}})$ satisfies the desired property, where
$\tilde{r}:\mathcal{J}_{\overline{V}}\times_{\overline{V}}\mathcal{Y}_{\overline{V}}\rightarrow \mathcal{J}_{{U}}\times_{{U}}\mathcal{Y}_{{U}}$ is the natural degree $M$ induced map.
\end{proof}
Having the lemma, we now observe that there is a natural proper morphism
$q'=(Id,q):\mathcal{J}_U\times_U\mathcal{Y}_U\rightarrow \mathcal{J}_U\times X$, where $q:\mathcal{Y}_U\rightarrow X$ is the second projection restricted to $\mathcal{Y}_U\subset
U\times X$,   and we  thus get
a codimension $k+1$ cycle $\mathcal{Z}_q:=q'_*\mathcal{Z}\in {\rm CH}^{k+1}(\mathcal{J}_U\times X)_\mathbb{Q}$
with Betti cohomology class
$[\mathcal{Z}_q]\in H^{2k+2}(\mathcal{J}_U\times X,\mathbb{Q})$ and Dolbeault cohomology
class $[\mathcal{Z}_q]^{k+1,k+1}\in H^{k+1}(\mathcal{J}_U\times X,\Omega_{\mathcal{J}_U\times X}^{k+1})$.
For any $\eta\in H^{k-1}(X,\Omega_X^{k+1})$, the corresponding class
$\sigma_U\in H^0(\mathcal{J}_U,\Omega_{J_U}^2)$ is defined by the formula
\begin{eqnarray}\label{eqsigma} \sigma_U=([\mathcal{Z}_q]^{k+1,k+1})^*(\eta),
\end{eqnarray}
where $([\mathcal{Z}_q]^{k+1,k+1})^*:H^{k-1}(X,\Omega_{X}^{k+1})\rightarrow H^0(\mathcal{J}_U,\Omega_{J_U}^2)$ is defined by
$$([\mathcal{Z}_q]^{k+1,k+1})^*(\omega)=pr_{1*}([\mathcal{Z}_q]^{k+1,k+1}\cup pr_2^*\omega ),$$
the $pr_i$'s being   the two projections defined on $\mathcal{J}_U\times X$.
This completes the construction of the form $\sigma_U$.

The following notation will be used below: As $\eta_{\mid Y_t}=0$ in $H^{k-1}(\Omega_{Y_t}^{k+1})$ and $H^{k-2}(Y_t,\Omega_{Y_t}^{k+1})=0$ by (\ref{eqvan}), $\eta$ determines a class
\begin{eqnarray}\label{eqtildeetat}
\tilde{\eta}_t\in H^{k-1}(Y_t,\Omega_{Y_t}^{k}(-L))
\end{eqnarray} using  the exact sequence
$$0\rightarrow \Omega_{Y_t}^{k}(-L)\rightarrow \Omega_{X\mid Y_t}^{k+1}\rightarrow \Omega_{Y_t}^{k+1}\rightarrow 0.$$
\begin{theorem} \label{theogen2form} The  holomorphic $2$-form
$$\sigma_U\in H^0(\mathcal{J}_U,\Omega_{\mathcal{J}_U}^2)$$
constructed above
satisfies the following properties:
\begin{itemize}
\item[(i)] The fibers of the fibration $\pi_U:\mathcal{J}_U\rightarrow U$ are isotropic for $\sigma_U$.
\item[(ii)] At any point $t\in U$, the map
$\lrcorner\sigma_t:T_{U,t}\rightarrow H^0(\mathcal{J}_t,\Omega_{\mathcal{J}_t})=H^{k-1}(Y_t,\Omega_{Y_t}^k)$ induced by $\sigma_U$ using (i) identifies with the multiplication   map
$$T_{U,t}=H^0(Y_t,L_{\mid Y_t})\rightarrow H^{k-1}(Y_t,\Omega_{Y_t}^{k})$$
 by the class $\tilde{\eta}_t$ of (\ref{eqtildeetat}).
\item[(iii)] For any smooth algebraic variety
$\overline{\mathcal{J}}$ containing $\mathcal{J}_U$ as a Zariski open set, the $2$-form $\sigma_U$ extends to a holomorphic $2$-form
on $\overline{\mathcal{J}}$.
\item[(iv)]  The $2$-form $\sigma_U$ is closed.
\end{itemize}
\end{theorem}

\begin{proof} With the notation $p_1:\mathcal{J}_U\times_U\mathcal{Y}_U\rightarrow \mathcal{J}_U$ for the
first projection, $p_2:\mathcal{J}_U\times_U\mathcal{Y}_U\rightarrow \mathcal{Y}_U$ for the second projection,
formula (\ref{eqsigma}) gives as well, using  the projection formula and
the fact that  $[\mathcal{Z}]_q^{k+1,k+1}=q'_*([\mathcal{Z}]^{k,k})$
\begin{eqnarray}\label{eqsigma1} \sigma_U=([\mathcal{Z}]^{k,k})^*({q}^*\eta),
\end{eqnarray}
where $([\mathcal{Z}]^{k,k})^*:H^{k-1}(\mathcal{Y}_U,\Omega_{\mathcal{Y}_U}^{k+1}))\rightarrow H^0(\mathcal{J}_U,\Omega_{J_U}^2)$ is defined by
$$([\mathcal{Z}]^{k,k})^*(\omega)=p_{1*}([\mathcal{Z}]^{k,k}\cup p_2^*\omega ).$$

(i) follows immediately from (\ref{eqsigma1}) which gives for $t\in U$
$$\sigma_{U\mid \mathcal{J}_t}=([\mathcal{Z}]^{k,k}_{\mid \mathcal{J}_t\times Y_t})^*(\eta_{\mid Y_t}),$$
and from the fact that $\eta_{\mid Y_t}=0$ in $H^{k+1,k-1}(Y_t)$ by assumption. (Here
we identify the fiber $Y_t$ of the universal family and its image in $X$.)

(iii) We observe that if $\overline{\mathcal{J}}\supsetneq \mathcal{J}_U$ is a smooth algebraic partial compactification of $\mathcal{J}_U$,
the cycle $\mathcal{Z}_q$ extends to a cycle $\overline{\mathcal{Z}_q}\in{\rm CH}^{k+1}(\overline{\mathcal{J}}\times X)$, so that its cohomology class $[\mathcal{Z}_q]^{k+1,k+1}$ extends to a class
$$[\overline{\mathcal{Z}_q}]^{k+1,k+1}\in H^{k+1}(\overline{\mathcal{J}}\times X,\Omega_{\overline{\mathcal{J}}\times X}^{k+1}).$$
It thus follows that the form  $\sigma_U$ extends to a $2$-form $\sigma\in H^0(\overline{\mathcal{J}},\Omega_{\overline{J}}^2)$
given by the formula
\begin{eqnarray}\sigma=([\overline{\mathcal{Z}_q}]^{k+1,k+1})^*(\eta).
\end{eqnarray}

(iv)  This is an immediate consequence of (iii). Indeed, choosing a smooth projective
compactification $\overline{\mathcal{J}}\supsetneq \mathcal{J}_U$ of $\mathcal{J}_U$, the $2$-form $\sigma_U$ extends by (iii) to a holomorphic $2$-form
$\sigma$  on
$\overline{\mathcal{J}}$. The $2$-form $\sigma$ is closed, hence the original form $\sigma_U$ is closed.

(ii)  Consider the class $q^*\eta\in H^{k-1}(\mathcal{Y}_U,\Omega_{\mathcal{Y}_U}^{k+1})$. As it vanishes
on fibers $Y_t,\,t\in U,$ of $u$, it provides for any $t\in U$ a morphism
\begin{eqnarray}
\label{eqprodint} {\rm int}(\cdot)q^*\eta:T_{U,t}\rightarrow H^{k-1}(Y_t,\Omega_{Y_t}^{k})
\end{eqnarray}
which for $k=1$ is simply obtained by taking interior product of the $2$-form $\eta$ with a local lift
of the considered tangent vector on the base, and for arbitrary $k$ is
constructed as follows:  the cotangent bundle  sequence of $u$
$$0\rightarrow \Omega_{U,t}\otimes \mathcal{O}_{Y_t}\rightarrow \Omega_{\mathcal{Y}_{U\mid Y_t}}\rightarrow \Omega_{Y_t}\rightarrow 0$$
induces an exact sequence
\begin{eqnarray}
\label{eqex1aa}
0\rightarrow \Omega_{U,t}\otimes \Omega_{Y_t}^{k}\rightarrow \Omega_{\mathcal{Y}_U\mid Y_t}^{k+1}/L^2\Omega_{\mathcal{Y}_U\mid Y_t}^{k+1}\rightarrow \Omega_{Y_t}^{k+1}\rightarrow 0,
\end{eqnarray}
where
$L^2\Omega_{\mathcal{Y}_{U\mid Y_t}}^{k+1}\subset \Omega_{\mathcal{Y}_{U\mid Y_t}}^{k+1}$ is the subbundle
$u^*\Omega_{U,t}^2\wedge\Omega_{\mathcal{Y}_{U\mid Y_t}}^{k-1}$.
From (\ref{eqex1aa}), using the fact that $H^{k-2}(Y_t,\Omega_{Y_t}^{k+1})=0$, we deduce that
the class $q^*\eta_{\mid Y_t}\in H^{k-1}(Y_t,{\Omega_{\mathcal{Y}_U}^{k+1}}_{\mid Y_t})$
lifts to a unique class $\widetilde{q^*\eta}_t$ in $\Omega_{U,t}\otimes H^{k-1}(Y_t,\Omega_{Y_t}^{k})={\rm Hom}\,(T_{U,t},H^{k-1}(Y_t,\Omega_{Y_t}^{k}))$, giving the desired morphism ${\rm int}(\cdot)q^*\eta
$ of (\ref{eqprodint}).
We use now formula
(\ref{eqsigma1}) which makes obvious that
for any $t\in U,\,v\in T_{U,t}$,
\begin{eqnarray}
\label{eqintcomp} v \lrcorner \sigma_U=([\mathcal{Z}]^{k,k}_{\mid \mathcal{J}_t\times Y_t})^*({\rm int}(v)q^*\eta)\,\,{\rm in}\,\,H^0(\mathcal{J}_t,\Omega_{\mathcal{J}_t}).
\end{eqnarray}
In the right hand side, we recall that, by construction, the morphism
$$([\mathcal{Z}]^{k,k}_{\mid \mathcal{J}_t\times Y_t})^*:H^k(Y_t,\Omega_{Y_t}^{k-1})\rightarrow H^0(\mathcal{J}_t,\Omega_{\mathcal{J}_t})$$
is the natural isomorphism.
It thus only remains to analyze the morphism
${\rm int}(\cdot)q^*\eta$. We observe now that
the cotangent bundle sequence (\ref{eqex1aa}) is compatible with
the conormal bundle sequence of $Y_t$ in $X$, since via the differential ${q}^*$ of the morphism
$q:\mathcal{Y}\rightarrow X$, we
get the following commutative diagram :

\begin{eqnarray}\label{eqcomdiag}
 \label{diagram} \xymatrix{
0\ar[r]&\mathcal{O}_{Y_t}(-L)\ar[r]\ar[d]&\Omega_{X\mid Y_t} \ar[r]\ar[d]& \Omega_{Y_t} \ar[r]\ar[d]&0&\\
0\ar[r]&\Omega_{U,t}\ar[r]&\Omega_{\mathcal{Y}_U\mid Y_t}\ar[r]& \Omega_{Y_t}\ar[r]&0&}
\end{eqnarray}

The first vertical map  is the natural inclusion dual to the evaluation
map $T_{U,t}=H^0(Y_t,L_{\mid Y_t})\rightarrow L_{\mid Y_t}$.
Taking $(k+1)$-th exterior powers, we get
the following commutative diagram:
\begin{eqnarray}\label{eqcomdiagext}
  \xymatrix{
0\ar[r]&\Omega_{Y_t}^{k}(-L)\ar[r]\ar[d]^f&\Omega_{X\mid Y_t}^{k+1} \ar[r]\ar[d]& \Omega_{Y_t}^{k+1} \ar[r]\ar[d]&0&\\
0\ar[r]&\Omega_{U,t}\otimes \Omega_{Y_t}^{k}\ar[r]&\Omega^{k+1}_{\mathcal{Y}_U\mid Y_t}/L^2\Omega^{k+1}_{\mathcal{Y}_U\mid Y_t}\ar[r]& \Omega_{Y_t}^{k+1}\ar[r]&0&}
\end{eqnarray}
It clearly follows from the commutativity of  diagram (\ref{eqcomdiagext}) that
\begin{eqnarray}\label{eqetaeta}f(\tilde{\eta}_t)=\widetilde{q^*\eta}_t\,\,\,{\rm in}\,\,\Omega_{U,t}\otimes H^{k-1}(Y_t,\Omega_{Y_t}^{k})
\end{eqnarray}
(where $\tilde{\eta}_t$ is as in \eqref{eqtildeetat}).
The proof of (ii) is now a consequence of (\ref{eqintcomp}) and (\ref{eqetaeta}). Indeed, for any $v\in T_{U,t}$ we have
$$v \lrcorner \sigma_U=([\mathcal{Z}]^{k,k}_{\mid \mathcal{J}_t\times Y_t})^*({\rm int}(v)q^*\eta)=([\mathcal{Z}]^{k,k}_{\mid \mathcal{J}_t\times Y_t})^*(\widetilde{q^*\eta}_t(v))
\,\,{\rm in}\,\,H^0(\mathcal{J}_t,\Omega_{\mathcal{J}_t}),$$
where the first equality is (\ref{eqintcomp}) and the second equality is by definition of $\widetilde{q^*\eta}_t$.
The equality (\ref{eqetaeta}) then tells that
$$\widetilde{q^*\eta}_t(v)=v\tilde{\eta}_t,$$
where on the right, $v$ is seen as an element of $H^0(\mathcal{O}_{Y_t}(-L))$ and
$v\tilde{\eta}_t$ is the product of $v$ and ${\eta}_t$. The proof is thus finished since
$([\mathcal{Z}]^{k,k}_{\mid \mathcal{J}_t\times Y_t})^*$ is the natural identification
between $H^{k-1}(Y_t,\Omega^k_{Y_t})$ and $H^0(\mathcal{J}_t,\Omega_{\mathcal{J}_t})$.
\end{proof}
The construction of the $2$-form
$\sigma_U$ and the proof of Theorem \ref{theogen2form}, assuming property ($\star$), are now complete. We conclude this section explaining
how to modify the arguments to get the same results only assuming (\ref{eqvan}).
If we examine the proofs given above, we see that the key tool is the algebraic cycle
$\mathcal{Z}\in{\rm CH}^k(\mathcal{J}_U\times_U \mathcal{Y}_U)_\mathbb{Q}$ and its image $\mathcal{Z}_q\in{\rm CH}^{k+1}(\mathcal{J}_U\times X)_\mathbb{Q}$. These cycles appear only through their
Dolbeault classes
$[\mathcal{Z}]^{k,k}$, $[\mathcal{Z}_q]^{k+1,k+1}$, which are better seen, after extensions to smooth projective varieties,
as Hodge classes.
In the absence of the cycle $\mathcal{Z}$ that we constructed using the assumption ($\star$) we still
have the desired Hodge classes, as follows from the following lemma. Below,  a Hodge class
on a smooth quasi-projective variety $Y$ is by definition the restriction of a Hodge class on a smooth projective compactification
$\overline{Y}$ of $Y$. The set of Hodge classes on $Y$ does not depend on the compactification
$\overline{Y}$. In fact, according to Deligne \cite{deligne}, Hodge classes in $H^{2k}(Y,\mathbb{Q})$ are identified with ${\rm Hdg}^{2k}({W}_{2k}H^{2k}(Y,\mathbb{Q}))$, where ${W}_{2k}H^{2k}(Y,\mathbb{Q})$ is the smallest weight part of
$H^{2k}(Y,\mathbb{Q})$, which is also the image of the restriction map
$H^{2k}(\overline{Y},\mathbb{Q})\rightarrow H^{2k}({Y},\mathbb{Q})$ for any smooth
projective compactification
$\overline{Y}$ of $Y$. Hodge classes $\alpha\in {\rm Hdg}^{2k}(H^{2k}(Y,\mathbb{Q}))$ have
a Dolbeault counterpart $\alpha^{k,k}\in H^k(Y,\Omega_Y^k)$ (which usually does not determine $\alpha$ in the non projective situation).
\begin{lemma} \label{lesubstit} Let $X, \,k,\,\,L$ be as above, satisfying condition (\ref{eqvan}). Then there
exists a Hodge class $\alpha\in{\rm Hdg}^{2k}(\mathcal{J}_U\times_U\mathcal{Y}_U,\mathbb{Q})$ with the property
that the class $\alpha_0\in H^0(U,R^{2k}(\pi_U,u)_*\mathbb{Q})$ induces the natural
isomorphism $H^{2k-1}(Y_t,\mathbb{Q})\cong H^1(\mathcal{J}_t,\mathbb{Q})$ at any point $t\in U$.
\end{lemma}
\begin{proof} The only observation to make is that the canonical
isomorphism $H^{2k-1}(Y_t,\mathbb{Q})\cong H^1(\mathcal{J}_t,\mathbb{Q})$ is an isomorphism of Hodge structures of bidegree $(-k+1,-k+1)$, by the vanishing condition (\ref{eqvan}). Such an isomorphism of Hodge structures
provides a degree $2k$ Hodge class $\alpha_t$ on the product $\mathcal{J}_t\times Y_t$. We thus have
a section of the local system $R^{2k}(\pi_U,u)_*\mathbb{Q}$ which is a Hodge class at any point of
$t$ of $U$. Deligne's global invariant cycle theorem \cite{deligne} then says that
for any smooth projective compactification $M$ of $\mathcal{J}_U\times_U\mathcal{Y}_U$, there exists
a cohomogy class $\beta\in H^{2k}(M,\mathbb{Q})$ such that $\beta_{\mid \mathcal{J}_t\times Y_t}=\alpha_t$
for any $t\in U$.
Using the facts that $\alpha_t$ is a Hodge class and the restriction morphism $H^{2k}(M,\mathbb{Q})\rightarrow H^{2k}(\mathcal{J}_t\times Y_t,\mathbb{Q})$ is a morphism of polarized Hodge structures, the semisimplicity of the category of polarized rational Hodge structures
allows us to  conclude that the class
$\beta$ can be chosen to be Hodge on $M$ (see \cite[2.2.1]{voisindecomp}). The restriction of $\beta$ to $\mathcal{J}_U\times_U\mathcal{Y}_U$ is then the
desired Hodge class $\alpha$ on $\mathcal{J}_U\times_U\mathcal{Y}_U$.
\end{proof}
This concludes the proof of Theorem \ref{theogen2form} assuming only (\ref{eqvan}).
We conclude this section observing that, except for (ii), we did not use the condition that $\mathcal{Y}$ is
the universal family of smooth divisors in $X$. Any smooth projective family mapping to $X$, or even only having a correspondence with $X$, with fibers satisfying
condition (\ref{eqvan})  will do. In practice, (ii) gives a way of deciding whether the constructed
$2$-form is degenerate or not. So
our arguments prove more generally the following variant of Theorem \ref{theogen2form}:
\begin{theorem}\label{theovariant} Let $X$ be a smooth
projective variety of dimension $n$ and let $f:\mathcal{Y}\rightarrow U$ be a smooth projective morphism between smooth quasi-projective varieties. Let $l,\,k$ be  integers
and let $Z\in{\rm CH}^{n-l+k-1}( \mathcal{Y}\times X)_\mathbb{Q}$ be a codimension $n-l+k-1$ cycle.  Assume  that the fibers
$Y_t$ of $f$ satisfy condition (\ref{eqvan}) for the given  integer $k$. Let
$\pi_U:\mathcal{J}_U\rightarrow U$ be the family of intermediate Jacobians $J^{2k-1}(Y_t),\,t\in U$. Then
\begin{itemize}
\item[(i)] For any class $\eta\in H^{l+2,l}(X)$ such that
$Z^*\eta_{\mid Y_t}=0$ in $H^{k+1,k-1}(Y_t)$ for any $t$ in $U$, there is a closed $(2,0)$-form
$\sigma_U\in H^0(\mathcal{J}_U,\Omega_{\mathcal{J}_U}^2)$ for which the fibers of $\pi_U$ are isotropic.
\item[(ii)] For any smooth algebraic partial compactification $\mathcal{J}_U\subset \overline{\mathcal{J}}$, the
$(2,0)$-form $\sigma_U$ extends to a $(2,0)$-form $\sigma$ on
$\overline{\mathcal{J}}$.
\end{itemize}
\end{theorem}
\begin{example}{\rm In  \cite{ilievmanivel}, Iliev and Manivel construct a Lagrangian fibration structure on the
family of intermediate Jacobians of smooth cubic fivefolds containing a given cubic fourfold
$X$. We recover the $(2,0)$-form as an application of Theorem \ref{theovariant}: The family
$\mathcal{Y}_U$ in this case is the universal  family of these cubic fivefolds and the integer $k$ is $3$. The cycle
$Z\subset X\times \mathcal{Y}_U\cong \mathcal{Y}_U\times X$ is isomorphic to
$X\times U$ and will be given by the embedding of $X$ in
$Y_t$ for any $t\in U$, hence we have  $l=1,\,n=4,\,n-l+k-1=5$ in this case.
Of course, some more work as in Theorem
\ref{theogen2form}(ii) is needed to show that the $(2,0)$-form is nondegenerate, but our approach
shows that this forms extends to any smooth projective compactification.
}
\end{example}
\subsection{The case of the cubic fourfold}\label{seccasecub}
 The paper will be devoted to the case where $X\subset
\mathbb{P}^5$ is a cubic fourfold,  $L=\mathcal{O}_X(1)$ and $k=2$. One has ${\rm dim}\,H^{3,1}(X)=1$ by Griffiths' theory, and a generator $\eta$ of $H^{3,1}(X)$ provides thus by Theorem
\ref{theogen2form} a $(2,0)$-form
$\sigma_U$ on the family of intermediate Jacobians of smooth hyperplane sections of $X$.
We have  the following:
\begin{prop}\label{theonondeg} If $X$ is a smooth cubic fourfold, the holomorphic $2$-form
$\sigma_{U}$ is nondegenerate on $\mathcal{J}_{U}$.
\end{prop}
 Proposition \ref{theonondeg}  already appears in \cite{DM2},
\cite{ilievmanivel},
\cite{manimarku}, \cite{marku}. The proof given here is slightly different, being an easy application
of Theorem \ref{theogen2form}.
\begin{proof}[Proof of Proposition \ref{theonondeg}] We apply
Theorem \ref{theogen2form}. In the case of the family of hyperplane sections of a cubic fourfold, the base $U$ and the fiber $J(Y_t)$ of the family $\mathcal{J}_U\rightarrow U$ are of dimension $5$.
The $2$-form $\sigma_U$ vanishes along the fibers of $\pi_U$ and in order to prove it is nondegenerate, it suffices to show that at any point $t\in U$, the map
$\lrcorner(\cdot)\sigma_U: T_{U,t}\rightarrow H^0(\mathcal{J}_t,\Omega_{\mathcal{J}_t})$ is an isomorphism.
Theorem \ref{theogen2form}(ii) tells us that
$\lrcorner(\cdot)\sigma_U$ is the following map: the generator
$\eta$ induces for each point $t\in U$ a class
$\tilde{\eta}_t\in H^1(Y_t,\Omega_{Y_t}^2(-1))$.
Then, using the identification $H^0(\mathcal{J}_t,\Omega_{\mathcal{J}_t})\cong H^1(Y_t,\Omega_{Y_t}^2)$,
 $\lrcorner(\cdot)\sigma_U:T_{U,t}=H^0(Y_t,\mathcal{O}_{Y_t}(1))\rightarrow H^1(Y_t,\Omega_{Y_t}^2)$ is
the multiplication map by $\tilde{\eta}_t$. So the statement of Proposition \ref{theonondeg}
is the following:

\begin{claim*}For any $t\in U$, the class $\tilde{\eta}_t\in H^1(Y_t,\Omega_{Y_t}^2(-1))$ induces an isomorphism
\begin{eqnarray}\label{eqmult} H^0(Y_t,\mathcal{O}_{Y_t}(1))\rightarrow H^1(Y_t,\Omega_{Y_t}^2).
\end{eqnarray}
\end{claim*}
The proof of the claim follows from the following lemma:
\begin{lemma}\begin{itemize} \item[(i)] The class $\tilde{\eta}_t\in H^1(Y_t,\Omega_{Y_t}^2(-1))$ is a nonzero multiple of the extension
class $e$ of the normal bundle sequence \begin{eqnarray}
\label{eqextpourY}
0\rightarrow T_{Y_t}\rightarrow T_{\mathbb{P}^4\mid Y_t}\rightarrow \mathcal{O}_{Y_t}(3)\rightarrow 0,
 \end{eqnarray}
 using the natural identification
$ \Omega_{Y_t}^2(-1)\cong T_{Y_t}(-3)$.
\item[(ii)] The extension class $e$ has the property that
the multiplication map by $e: H^0(Y_t,\mathcal{O}_{Y_t}(1))\rightarrow H^1(Y_t,\Omega_{Y_t}^2)$ is an isomorphism.
\end{itemize}
\end{lemma}
\begin{proof} (ii) is Griffiths' residue isomorphism (see \cite[II, 6.1.3]{voisinbook}) and in this case, the statement immediately follows from the
exact sequence (\ref{eqextpourY}) and
the fact that $H^0(Y_t,T_{\mathbb{P}^4\mid Y_t}(-2))=0$, and
 $H^2(Y_t,T_{Y_t}(-2))=0$. As for (i), this simply follows from the
 fact that the class $\tilde{\eta}_t\in H^1(Y_t,\Omega_{Y_t}^2(-1))$ is nonzero because $\eta\not=0$, and
 $H^1(X,\Omega_X^3(-1))=0$, so that
 $\eta_{\mid Y_t}\not=0$. On the other hand, $H^1(Y_t,\Omega_{Y_t}^2(-1))=H^1(Y_t,T_{Y_t}(-3))$ is $1$-dimensional, as follows
 from the normal bundle sequence
 (\ref{eqextpourY}).
 \end{proof}
The proof of Proposition \ref{theonondeg} is finished.
\end{proof}

\subsection{Another example: Quadric sections of cubic fourfolds}Note that the cubic fourfold $X$ has another family of smooth divisors $Y\subset X$ satisfying condition
(\ref{eqvan}), namely the smooth complete intersections $Q\cap X$, where $Q$ is a quadric in $\mathbb{P}^5$.
The corresponding family $\mathcal{J}_Q\rightarrow U_Q$ of intermediate Jacobians has a basis $U_Q$ of dimension
$20=h^0(\mathbb{P}^5,\mathcal{O}_{\mathbb{P}^5}(2))-1$ and   fibers of dimension $20$. Theorem
\ref{theogen2form} shows that $\mathcal{J}_Q$ has a closed holomorphic $2$-form $\sigma_Q$ which extends
to any smooth algebraic compactification $\overline{\mathcal{J}_Q}$  of $\mathcal{J}_Q$.
However the $2$-form in this case is only generically nondegenerate:
\begin{lemma} The $2$-form $\sigma_Q$ is nondegenerate along a fiber $\mathcal{J}_{Q,t}=J(Y_t)$,
where $Y_t=Q_t\cap X$,
if and only if the quadric $Q_t$ is nondegenerate.
\end{lemma}
\begin{proof} The generator $\eta$ of $H^1(X,\Omega_X^3)=H^1(X,T_X(-3))$
is the extension class of the normal bundle sequence
$$0\rightarrow T_X\rightarrow T_{\mathbb{P}^5\mid X}\rightarrow \mathcal{O}_X(3)
\rightarrow 0.$$
Restricting to
$Y_t=Q_t\cap X$, we get
the exact sequence
$$0\rightarrow T_{X\mid Y_t}\rightarrow T_{\mathbb{P}^5\mid Y_t}\rightarrow \mathcal{O}_{Y_t}(3)
\rightarrow 0,$$
whose extension class must come from
the extension class $e_t$ of the normal bundle sequence
\begin{eqnarray}
\label{eqtruc15}0\rightarrow T_{Y_t}\rightarrow T_{Q\mid Y_t}\rightarrow \mathcal{O}_{Y_t}(3)
\rightarrow 0\end{eqnarray}
of $Y_t$
in $Q_t$.
In other words, the class $\tilde{\eta}_t\in H^1(\Omega_{Y_t}^2(-2))=H^1(Y_t,T_{Y_t}(-3))$
is nothing but the extension class $e_t$.
It follows that the multiplication map
$$\tilde{\eta}_t: H^0(Y_t, \mathcal{O}_{Y_t}(2))\rightarrow H^1(Y_t,\Omega_{Y_t}^2)$$
identifies with the map
$H^0(Y_t, \mathcal{O}_{Y_t}(2))\rightarrow H^1(Y_t,\Omega_{Y_t}^2)\cong H^1(Y_t,T_{Y_t}(-1))$ induced by the
exact sequence (\ref{eqtruc15}) twisted by $ \mathcal{O}_{Y_t}(-1)$.
Looking at the long exact sequence associated to (\ref{eqtruc15}), we find
that this map is an isomorphism if and only if
$H^0(Y_t,T_{Q_t\mid Y_t}(-1))=0$. But
$H^0(Y_t,T_{Q_t\mid Y_t}(-1))=0$ if and only if $Q_t$ is not singular.
\end{proof}
\subsection{Extensions to nodal fibers}
Let $X$ be a smooth cubic fourfold, and $\eta\in H^1(X,\Omega_X^3)$ be a generator of
$H^1(X,\Omega_X^3)$.  We use as before the notation
$\mathcal{Y}\rightarrow B,\,\mathcal{Y}_U\rightarrow U$ for the universal family of hyperplane sections of $X$.
Let $U_1\subset |\mathcal{O}_X(1)|$ be the Zariski open set parametrizing hyperplane sections of $X$ with at most one ordinary double point.
The Jacobian fibration $\pi_U:\mathcal{J}_U\rightarrow U$ has a flat projective extension $\pi_{U_1}:\mathcal{J}_{U_1}\rightarrow U_1$ with smooth total space (see Lemma \ref{leisoenhaut} for the smoothness statement). As the vanishing cycle of the degeneration is not trivial at a point $t\in U_1\setminus U$, the fiber of $\pi_{U_1}$ over $t$
 is a singular compactification
of a $\mathbb{C}^*$-bundle over $J(\widetilde{Y_t})$, where $\widetilde{Y_t}$ is the desingularization of $Y_t$ obtained by blowing-up the node. 
We will denote below by $\mathcal{J}_{U_1}^{\circ}$ the quasiabelian part of $\mathcal{J}_{U_1}$. Note that $\mathcal{J}_{U_1}^{\circ}\subset \mathcal{J}_{U_1}$ has a complement of codimension $2$, consisting of the singular loci of the compactified Jacobians
over $U_1\setminus U$.
By Theorem \ref{theogen2form}(iii), the $2$-form $\sigma_U$ extends to a $2$-form
$\sigma_{U_1}$ on $\mathcal{J}_{U_1}$, for which the fibers of $\pi_{U_1}$ are isotropic by
Theorem \ref{theogen2form}(i) (and the fact that the fibers of $\pi_{U_1}$ are equidimensional).
Next, the smooth locus $J({Y_t})_{reg}=\mathcal{J}_{U_1,t}^{\circ}$ is a quasiabelian variety with cotangent space isomorphic to
$H^{1}(\widetilde{Y_t},\Omega_{\widetilde{Y_t}}^{2}({\rm log} \, E_Y))$,
 where $E_Y$ is the exceptional divisor
of the resolution $\widetilde{Y_t}\rightarrow Y_t$.

Our goal in this section is to reprove the following result which can be found in \cite[\S8.5.2]{DM2}:
\begin{prop} \label{proextnondeg} The extended $2$-form $\sigma_{U_1}$ is everywhere nondegenerate
on $\mathcal{J}_{U_1}^{\circ}$, hence also on 
 $\mathcal{J}_{U_1}$.
\end{prop}
As an immediate corollary, we get the following result:
\begin{cor}\label{coroextend} Assume $\mathcal{J}_{U_1}$ has a smooth  compactification
$\overline{\mathcal{J}}$ with boundary $\overline{\mathcal{J}}\setminus \mathcal{J}_{U_1}$ of codimension $\geq2$ in $\overline{\mathcal{J}}$.
Then $\overline{\mathcal{J}}$ is holomorphically symplectic.
\end{cor}
\begin{proof} Indeed, the extended $2$-form, being nondegenerate away from a codimension $2$ closed analytic subset, is
everywhere nondegenerate. \end{proof}

The proof of Proposition \ref{proextnondeg} is based on the following Lemma
\ref{theoextend}.
 For $t\in U_1\setminus U$, we have the inclusions
$$\widetilde{Y_t}\subset\widetilde{X}\subset \widetilde{\mathbb{P}^5},$$
where $\widetilde{X}$, resp. $\widetilde{\mathbb{P}^5}$, is the blow-up of $X$, resp. $\mathbb{P}^5$ at the singular
point of $Y_t$.
We denote by $E$ the exceptional divisor of $\widetilde{\mathbb{P}^5}$ and
$E_X$  the exceptional divisor of
$\widetilde{X}$, so $E_Y=E_X\cap\widetilde{Y_t}$.
As  $\widetilde{X}$ is transverse to the exceptional
divisor $E$ of  $\widetilde{\mathbb{P}^5}$ and belongs to the
linear system $|\mathcal{O}_{\widetilde{\mathbb{P}^5}}(3)(-E)|$, we have a
logarithmic tangent bundles sequence
\begin{eqnarray}\label{eqlogextnorm}0\rightarrow T_{\widetilde{X}}({\rm log}\,E_X)
\rightarrow T_{\widetilde{\mathbb{P}^5}}({\rm log}\,E)_{\mid \widetilde{X}}\rightarrow \mathcal{O}_{\widetilde{X}}(3)(-E_X)\rightarrow0.
\end{eqnarray}
Here we recall that the logarithmic tangent bundle of a variety equipped with a smooth
divisor $D$ is the dual of the logarithmic cotangent bundle determined by $D$ and can be defined as the
sheaf of vector fields tangent to $D$ along $D$.
As  $E_X\cap \widetilde{Y}_t=E_Y$, we  also get
  natural inclusions for any $l$:
\begin{eqnarray} \label{eqqenplus24oct} T_{\widetilde{Y_t}}({\rm log}\,E_Y)(lE_Y)(-3)\subset T_{\widetilde{X}}({\rm log}\,E_X)(lE_X)(-3)_{\mid \widetilde{Y}_t}.
\end{eqnarray}

\begin{lemma}\label{theoextend} The induced map
$\lrcorner\sigma_{U_1,t}:T_{U_1,t}\rightarrow H^{1}(\widetilde{Y_t},\Omega_{\widetilde{Y_t}}^{2}({\rm log}\, E_Y))$ is constructed as follows.
\begin{itemize}
\item[(i)] The extension class $e\in H^1(T_{\widetilde{X}}({\rm log}\,E_X)(E_X)(-3))$ of (\ref{eqlogextnorm}), maps naturally to an element
$e'\in H^1(T_{\widetilde{X}}({\rm log}\,E_X)(2E_X)(-3))$, which restricted to $\widetilde{Y_t}$ comes via (\ref{eqqenplus24oct}) from a uniquely defined element
\begin{eqnarray}\label{eqeY} e_Y\in H^1(\widetilde{Y_t},T_{\widetilde{Y_t}}({\rm log}\,E_Y)(2E_Y)(-3)).
\end{eqnarray}
\item[(ii)] One has $T_{\widetilde{Y_t}}({\rm log}\,E_Y)(2E_Y)(-3)=\Omega^2_{\widetilde{Y_t}}({\rm log}\,E_Y)(-1)$, thus
$$e_Y\in H^1(\widetilde{Y_t},\Omega^2_{\widetilde{Y_t}}({\rm log}\,E_Y)(-1)).$$
\item[(iii)]  The interior product $\lrcorner\sigma_{U_1,t}:T_{U_1,t}\rightarrow H^1(\widetilde{Y_t},\Omega^2_{\widetilde{Y_t}}({\rm log}\,E_Y))$ is given by
multiplication $H^0(\widetilde{Y_t},\mathcal{O}_{\widetilde{Y_t}}(1))\stackrel{e_Y}{\rightarrow }
H^1(\widetilde{Y_t},\Omega^2_{\widetilde{Y_t}}({\rm log}\,E_Y))$.
\end{itemize}
\end{lemma}
\begin{proof} (i) We write the logarithmic normal bundle sequence for
$\widetilde{Y_t}\subset \widetilde{X}$:
$$0\rightarrow T_{\widetilde{Y_t}}({\rm log}\,E_Y)\rightarrow T_{\widetilde{X}}({\rm log}\,E_X)_{\mid \widetilde{Y_t}}\rightarrow \mathcal{O}_{\widetilde{Y_t}}(1)(-2E_Y)\rightarrow0$$
which we twist by $\mathcal{O}_{\widetilde{Y_t}}(2E_Y)(-3)$.
The conclusion then follows from the following easy vanishing statements:
$$H^1(\widetilde{Y_t},\mathcal{O}_{\widetilde{Y_t}}(-2))=0,\,\,
H^0(\widetilde{Y_t},\mathcal{O}_{\widetilde{Y_t}}(-2))=0.$$
\begin{remark}{\rm It is easy to check that $e_Y$ is in fact the class of the logarithmic tangent bundles exact sequence
$$0\rightarrow T_{\widetilde{Y_t}}({\rm log}\,E_Y)\rightarrow T_{\widetilde{\mathbb{P}^4}}({\rm log}\,E_{\mathbb{P}^4})_{\mid \widetilde{Y_t}}\rightarrow \mathcal{O}_{\widetilde{Y_t}}(3)(-2E_Y)\rightarrow0$$
associated to the embedding of $\widetilde{Y_t}$ in the blow-up
$\widetilde{\mathbb{P}^4}$ of the hyperplane ${\mathbb{P}^4}$
 containing $Y_t$ at the singular point of $Y_t$. }
\end{remark}
(ii) This follows from the fact that $T_{\widetilde{Y_t}}({\rm log}\,E_Y)$ is dual to
$\Omega_{\widetilde{Y_t}}({\rm log}\,E_Y)$ and that the later has determinant
$K_{\widetilde{Y_t}}(E_Y)=\mathcal{O}_{\widetilde{Y_t}}(-2)(2E_Y)$. Thus
$$\Omega^2_{\widetilde{Y_t}}({\rm log}\,E_Y)\cong T_{\widetilde{Y_t}}({\rm log}\,E_Y)(K_{\widetilde{Y_t}}(E_Y))=
T_{\widetilde{Y_t}}({\rm log}\,E_Y)(-2)(2E_Y).$$

(iii)  The morphism $u:\mathcal{Y}_{U_1}\rightarrow U_1$ is smooth along the
smooth locus ${Y}_{t,reg}$ of the fiber ${Y}_{t}$. It follows that
the arguments used in the proof of Theorem \ref{theogen2form}(ii) apply along ${Y}_{t,reg}$, so that we
can  conclude that the conclusion of (iii) holds true in $H^1({Y}_{t,reg},\Omega_{Y_{t,reg}}^2)$, that is, after
composing with the
restriction map $H^1(\widetilde{{Y}_{t}},\Omega_{\widetilde{{Y}_{t}}}^2({\rm log}\,E_Y))\rightarrow H^1({Y}_{t,reg},\Omega_{{Y}_{t,reg}}^2)$. The proof is then finished using the following sublemma:
\begin{sublemma} The restriction map $H^1(\widetilde{{Y}_{t}},\Omega_{\widetilde{{Y}_{t}}}^2({\rm log}\,E_Y))\rightarrow H^1({Y}_{t,reg},\Omega_{{Y}_{t,reg}}^2)$ is an isomorphism.
\end{sublemma}
\begin{proof} Note that ${Y}_{t,reg}=\widetilde{{Y}_{t}}\setminus E_Y$. Denoting
by $j: {Y}_{t,reg}\rightarrow \widetilde{{Y}_{t}}$ the inclusion map,
$j$ is an affine map and we have $\Omega_{{Y}_{t,reg}}^2=j^*(\Omega_{\widetilde{{Y}_{t}}}^2({\rm log}\,E_Y))$, so that
\begin{eqnarray}\label{eqomegayreg} H^1({Y}_{t,reg},\Omega_{{Y}_{t,reg}}^2)=
H^1(\widetilde{{Y}_{t}},R^0j_*(\Omega_{\widetilde{{Y}_{t}}}^2({\rm log}\,E_Y)_{\mid{Y}_{t,reg}}))\\
\nonumber
=\underset{\underset{k}\rightarrow}{\rm lim} \, H^1(\widetilde{{Y}_{t}},\Omega_{\widetilde{{Y}_{t}}}^2({\rm log}\,E_Y)(kE_Y)).
\end{eqnarray}
The lemma then follows from the following exact sequence:
\begin{eqnarray}
\label{eqexpourEYlog}
0\rightarrow \Omega_{E_Y}^2\rightarrow \Omega_{\widetilde{Y_t}}^2({\rm log}\,E_Y)_{\mid E_Y}\rightarrow \Omega_{E_Y}\rightarrow 0.
\end{eqnarray}
Indeed, we recall that $E_Y\cong \mathbb{P}^1\times \mathbb{P}^1$ and that
$\mathcal{O}_{E_Y}(E_Y)=\mathcal{O}_{E_Y}(-1,-1)$. It follows that
for any $k>0$,
$$H^1(E_Y,\Omega_{E_Y}(kE_Y))=0,\,\,H^1(E_Y,\Omega_{E_Y}^2(kE_Y))=0,$$
and for any $k\geq0$,
$$H^0(E_Y,\Omega_{E_Y}(kE_Y))=0,\,\,H^0(E_Y,\Omega_{E_Y}^2(kE_Y))=0,$$
Using the exact sequence (\ref{eqexpourEYlog}), this
shows that $$H^1(E_Y,\Omega_{\widetilde{{Y}_{t}}}^2({\rm log}\,E_Y)_{\mid E_Y}(kE_Y))=0$$
for $k>0$ and
$H^0(E_Y,\Omega_{\widetilde{{Y}_{t}}}^2({\rm log}\,E_Y)_{\mid E_Y}(kE_Y))=0$
for $k\geq0$. It follows that the map
$$H^1(\widetilde{{Y}_{t}},\Omega_{\widetilde{{Y}_{t}}}^2({\rm log}\,E_Y)(kE_Y))
\rightarrow H^1({{Y}_{t}},\Omega_{{{Y}_{t}}}^2({\rm log}\,E_Y)((k+1)E_Y))$$
is an isomorphism for $k\geq0$, proving
the lemma by (\ref{eqomegayreg}).
\end{proof}
The proof of Lemma \ref{theoextend} is now complete.
\end{proof}

\begin{proof}[Proof of Proposition \ref{proextnondeg}] We  have to prove that $\sigma_{U_1}$ is nondegenerate  at any point of $\mathcal{J}_{U_1}^{\circ}$ over  $t\in U_1\setminus U$. This is equivalent
to proving that
the map
$$\lrcorner\sigma_{U_1}:T_{U_1,t}\rightarrow H^0_{inv}(\mathcal{J}_{U_1,t}^{\circ},\Omega_{\mathcal{J}_{U_1,t}^{\circ}})= H^1(\widetilde{Y_t},\Omega^2_{\widetilde{Y_t}}({\rm log}\,E_Y)(-E_Y))$$
is an isomorphism, where $H^0_{inv}$ here denotes the space of translation invariant $1$-forms.
 Using Lemma \ref{theoextend}, the last statement is equivalent to
 the fact that the multiplication map
 \begin{eqnarray}
 \label{eqmulteY}e_Y:H^0(\widetilde{Y_t},\mathcal{O}_{\widetilde{Y_t}}(3))\rightarrow
 H^1(\widetilde{Y_t},\Omega^2_{\widetilde{Y_t}}({\rm log}\,E_Y)(-E_Y))
 \end{eqnarray}
 is an isomorphism, where
 $e_Y\in H^1(\widetilde{Y_t},\Omega^2_{\widetilde{Y_t}}({\rm log}\,E_Y)(-E_Y)(-1))$.
 We have
 $$\Omega^2_{\widetilde{Y_t}}({\rm log}\,E_Y)(-E_Y)(-1)\cong T_{\widetilde{Y_t}}({\rm log}\,E_Y)(-3)(E_Y)$$
 and the class $e_Y$ maps to the extension class $e'_Y$
 \begin{eqnarray}\label{eqnouveext}
 0\rightarrow T_{\widetilde{Y_t}}({\rm log}\,E_Y)
 \rightarrow T_{\widetilde{\mathbb{P}^4}}({\rm log}\,E)_{\mid\widetilde{Y_t}}
 \rightarrow \mathcal{O}_{\widetilde{Y_t}}(3-2E_Y)\rightarrow0\end{eqnarray}
 of the logarithmic normal bundle sequence
 of $\widetilde{\mathcal{Y}}_t$ in $\widetilde{\mathbb{P}^4}$, via the
 natural map
 $$H^1(\widetilde{Y_t},T_{\widetilde{Y_t}}({\rm log}\,E_Y)(-3)(E_Y))\rightarrow H^1(\widetilde{Y_t},T_{\widetilde{Y_t}}({\rm log}\,E_Y)(-3)(2E_Y)).$$
 An element in the kernel of the multiplication map by $e_Y$ is thus also
 in ${\rm Ker}\,e'_Y: H^0(\widetilde{Y_t},\mathcal{O}_{\widetilde{Y_t}}(1))\rightarrow
 H^1(\widetilde{Y_t},T_{\widetilde{Y_t}}({\rm log}\,E_Y)(-2)(2E_Y))$ induced by
 (\ref{eqnouveext}), hence comes from an element of $$H^0(\widetilde{Y_t},T_{\widetilde{\mathbb{P}^4}}({\rm log}\,E)_{\mid\widetilde{Y_t}}(-2)(2E_Y)),$$
  and it is easily shown that this space reduces to zero.  The map
 $\lrcorner\sigma_{U_1}:T_{U_1,t}\rightarrow H^0_{inv}(\mathcal{J}_{U_1,t}^{\circ},\Omega_{\mathcal{J}_{U_1,t}^{\circ}})$ is thus injective, and hence an isomorphism.
 \end{proof}

\section{Good and very good lines \label{secgoodline}}
Our main tool for studying degenerations of intermediate Jacobians of cubic threefolds is Mumford's description of the intermediate Jacobian $J(Y)$ as a Prym variety $\Prym(\widetilde C/C)$. The curve $C$ (and its \'etale double cover $\widetilde C$) are obtained by projecting from a generic line on the smooth cubic $Y$. In \cite{casaetal} (this is also subsequently used in \cite{casaetal2}), it is noted that much of the Prym construction carries on to the mildly singular case provided a careful choice of a line $l$ on (the possibly singular) $Y$. This provides a powerful tool for studying the degenerations of intermediate Jacobians. We caution the reader that the context in the current paper is slightly different from that of \cite{casaetal,casaetal2} (e.g. see Remark \ref{diffcml}) forcing us to reprove (under slightly different hypotheses) and strengthen certain results. For convenience, we have tried to make the exposition below mostly self-contained.

\begin{notation} If $X$ is a cubic fourfold and $Y\subset X$ is a hyperplane section,
we denote by $F(Y)$, resp. $F(X)$, $F(Y)\subset F(X)$, the varieties of lines in $Y$, resp. $X$.
We denote by
$[l]\in F(Y)$, resp. $F(X)$  the point parametrizing $l\subset Y$, resp. $l\subset X$.
\end{notation}

\subsection{Good Lines}
\begin{definition}[{Cf. \cite[Def. 3.4]{casaetal}}]
 \label{defigood} Let $Y\subset \mathbb{P}^4$ be a cubic threefold not containing any plane.  A line $l\subset Y$ is {\it good} if for any plane $P\subset \mathbb{P}^4$ containing $l$,  $P\cap Y$ consists in three distinct lines.
\end{definition}
The notion of a good line is obviously important from the point of view of the Prym construction of the
intermediate Jacobian of a cubic threefold. Projecting $Y$ from $l$, we get a conic bundle
$\widetilde{Y}_l\rightarrow \mathbb{P}^2$, where
$\widetilde{Y}_l$ is the blow-up of $Y$ along $l$, and
the discriminant curve $C_l\subset \mathbb{P}^2$ parametrizing reducible conics has degree $5$.
The curve $\widetilde C_l$ of lines in $Y$ intersecting $l$ is the double cover of $C_l$ with fiber
over the point $c$ parametrizing a reducible conic
$C$ the set of
components of $C$. Thus if $l$ is good  the natural involution acting on $\widetilde C_l$ has no fixed point.

\begin{prop}\label{propgoodline} Let $X$ be a general cubic fourfold. Then any hyperplane section
$Y$ of $X$ has a good line.
\end{prop}
\begin{remark}\label{diffcml}The existence of a good line is proved in \cite{casaetal} when $Y$ has singularities of type
$A_k$ for $k\leq 5$ or $D_4$ (i.e. the singularities relevant in the GIT context). Unfortunately, we need to allow some additional ADE singularities (e.g.  $D_5$)  as these can appear as singularities of hyperplane sections of general cubic fourfolds. It is very likely
that the arguments of \cite{casaetal} could be extended to cover the cases needed in this paper, but we prefer to give an alternative proof here.
\end{remark}
\begin{proof}[Proof of Proposition \ref{propgoodline}] Let us say that a line $l$ in  $X$ is {\it special} in $X$, resp. in $Y$, if the restriction map
$J_X\rightarrow H^0(l,\mathcal{O}_l(2))$, resp.  $J_Y\rightarrow H^0(l,\mathcal{O}_l(2))$,
has rank $\leq2$, where $J_X$, resp. $J_Y$ denotes the degree $2$ part of the Jacobian ideal of $X$, resp. $Y$.
As $X$ is general,
lines which are special in $X$ are parametrized  by a smooth surface
$\Sigma_{sp}\subset F(X)$ (see \cite{amerik}).
Recall from \cite{voisinfano} that the variety $F(X)$ has a rational self-map
$\phi:F(X)\dashrightarrow F(X)$. The map $\phi$ associates
to $[l]\in F(X)$ the point $[l']$ parametrizing the line $l'\subset X$
constructed as follows: if $l$ is not special in $X$, that is $[l]\not\in \Sigma_{sp}$,
there is a unique plane $P_l\subset \mathbb{P}^5$ such that
$P_l\cap X=2l+l'$ as a divisor of $P_l$, where $l'\subset X$ is the desired  line in $X$.
When $X$ contains no plane, the indeterminacy locus of $\phi$ is exactly the surface $\Sigma_{sp}$, along which the plane $P_l$ above
is not unique. Furthermore, the indeterminacies of the map $\phi$ are solved after blowing-up
the surface $\Sigma_{sp}$ and the induced morphism
$\tilde\phi:\widetilde{F(X)}\rightarrow F(X)$ is finite if $X$ is general (see \cite{amerik}).
Note that the condition on a line $l'\subset Y$ to being good will be implied by the slightly stronger
 fact that $l'$ is non-special in $X$ (so $\phi$ is well-defined at $[l']$)
and for no
point $[l]\in F(Y)$, one has $\tilde\phi([l])=[l']$ or $\phi([l'])=[l]$. (For the first of these conditions,
one has  rather to consider a point over $[l]$  in $\widetilde{F(X)}$.)

We first have:
\begin{lemma} \label{leFYirred} \begin{itemize}
\item[(i)] If $X$ is smooth, $F(Y)$ is a surface for any hyperplane section $Y$ of $X$. Furthermore $F(Y)_{red}\subset F(X)$ is Lagrangian for the holomorphic $2$-form
     $\sigma$ on the smooth hyper-K\"ahler manifold $F(X)$ (see \cite{bedo}).
\item[(ii)] If $X$ is general, $F(Y)$ is irreducible and reduced for any hyperplane section of $X$.
\end{itemize}
\end{lemma}
\begin{proof} (i)  It is classical that $F(Y)$ is smooth of dimension $2$ at any point $[l]$
parametrizing a line $l$ contained in the smooth locus of $Y$. Moreover, $Y$ has isolated singularities, and
the families of lines in $Y$ through any point $y\in Y$ cannot be $3$-dimensional as otherwise it would be
the whole set of lines in $H_Y$ passing through $y$, where we denote by $H_Y$ the hyperplane cutting $Y$ in $X$. This proves the first statement.
 The $(2,0)$-form
$\sigma$ on $F(X)$ is deduced from the class $\alpha$ generating $H^{3,1}(X)$ by the formula
$$\sigma=P^*\alpha\,\,{\rm in}\,\,H^{2,0}(F(X)),$$
where $P\subset F(X)\times X$ is the incidence correspondence, so that $p:P\rightarrow F(X)$ is a $\mathbb{P}^1$-bundle
over $F(X)$.
 Denoting by $P_Y\subset F(Y)\times Y$ the incidence correspondence of $Y$,
we observe that, since $Y$ has only isolated singularities, $P_Y$ lifts
to a correspondence
$P_{\widetilde{Y}}\subset F(Y)\times \widetilde{Y}$, where $\widetilde{Y}$ is a desingularization of $Y$.
If $U$ is any open set contained in the regular locus of $F(Y)_{red}$, we then have
$$ \sigma_{\mid U}=P^*(\alpha)_{\mid U}=P_{\widetilde{Y}}^*(j^*\alpha)\,\,{\rm in}\,\,H^{2,0}(U),$$
where $j:\widetilde{Y}\rightarrow X$ is the desingularization map.
Thus the vanishing of $ \sigma_{\mid U}$ follows from the vanishing
of $j^*\sigma$ in $H^{3,1}(\widetilde{Y})$. To get the last vanishing, observe that $\widetilde{Y}$ is smooth of dimension
$3$ and rationally connected, so that we have  $H^{2,0}(\widetilde{Y})=0$ hence also $H^{3,1}(\widetilde{Y})=0$.
Thus $F(Y)_{red}$ is Lagrangian for $\sigma$.

(ii) The stated property is Zariski open, so it suffices to prove it when
$X$ is very general. In this case, the space ${\rm Hdg}^4(F(X))$ of rational Hodge classes of degree $4$ on $F(X)$
is of dimension $2$. Let us say that a class $\gamma\in {\rm Hdg}^4(F(X))$ is  Lagrangian if
$\gamma\cup [\sigma]=0$ in $H^6(F(X),\mathbb{C})$. The class $l^2$, where $l$ is a Pl\"ucker hyperplane section
of $F(X)$, is not Lagrangian by the second Hodge-Riemann bilinear relations and thus the space ${\rm Hdg}^4(F(X))_{lag}$ of Lagrangian rational Hodge classes  on $F(X)$ is of dimension $\leq 1$.
 It follows that
the class $[F(Y)]\in {\rm Hdg}^4(F(X))_{lag}$ cannot be written as the sum of two nonproportional Lagrangian classes.
In fact, coming back to integer coefficients, it can neither be written as the sum of two proportional nonzero effective classes.
Indeed, the class $[F(Y)]\in {\rm Hdg}^4(F(X),\mathbb{Z})$ is primitive, that is,
not divisible by any nonzero integer $\not=\pm1$, because when
$X$ contains a plane $P$, $F(X)$ contains the dual plane $P^*$ and
$[F(Y)]\cdot[P^*]=1$. We thus  proved that $F(Y)$ is irreducible and reduced.
\end{proof}

Coming back to the proof of Proposition \ref{propgoodline}, it is clear that
for any hyperplane section $Y$ of $X$, there is a  line  contained in $Y$ which is nonspecial in $X$.
Indeed, the surface $\Sigma_{sp}$ of  lines which are special in $X$ is irreducible  and not contained in the surface
of lines in $Y$ because  it is smooth
connected and not Lagrangian, see \cite{amerik}; thus it can intersect $F(Y)$ only along a proper subset. Next, assume to the contrary that
there is no good line in $Y$. This then means that
for a general  $[l]\in F(Y)$, (hence nonspecial for $X$), either (1) there is
a $[l']\in F(Y)$ such that for some plane $P\subset H_Y$,
$P\cap Y=2l+l'$, that is $\phi(l)=l'$,  or  (2) there is
a $[l']\in F(Y)$ such that for some plane $P\subset H_Y$, $P\cap Y=l+2l'$. In  case (1), the map $\phi$ is well-defined at the general point
$[l]\in F(Y)$ hence of maximal rank at $[l]$ because $\phi^*\sigma=-2\sigma$, hence $\phi(F(Y))=F(Y)$ by irreducibility of $F(Y)$. But then  $[l']$ is also general in $F(Y)$, which implies that
(2) occurs as well. So we just have to exclude (2).
Note that the line $l'$ is then special for $Y$.

There are two possibilities:
\begin{itemize}
\item[(a)]  The point $[l']\in F(Y)$ moves in a surface contained in $F(Y)$, hence by Lemma \ref{leFYirred}, any line in $Y$ is special
for $Y$.
\item[(b)] The point $[l']$ moves in a curve  $D\subset F(Y)$ and this curve is contained $\Sigma_{sp}$.
Furthermore, for any $[l']\in D$, the $3$-dimensional
projective space $Q_l=\cap_{x\in l}H_{X,x}$ is contained in $H_Y$, where  $H_{X,x}$
 denotes the hyperplane tangent to $X$ at $x$.
\end{itemize}

In case (a), we get a contradiction as follows: the general line $l'\subset Y$ does not pass through a
singular point of $Y$ and the fact that $l'$ is special for $Y$
says exactly, by taking global sections in the normal bundle sequence
$$0\rightarrow N_{l'/Y}(-1)\rightarrow N_{l'/\mathbb{P}^4}(-1)\rightarrow \mathcal{O}_{l'}(2)\rightarrow 0,$$
that $H^0(N_{l'/Y}(-1))\not=0$, hence that $N_{l'/Y}\cong \mathcal{O}_{l'}(1)\oplus \mathcal{O}_{l'}(-1)$.
But the fact that $N_{l'/Y}$ has this form says equivalently that the
map
$q:P_Y\rightarrow Y$ is not submersive at any point of  the fiber $P_{Y,[l']}\subset P_Y$ of $p:P_Y\rightarrow
F(Y)$ over the point $[l']$.
As $q(P_Y)=Y$, this contradicts the fact that $[l']$ is general in $F(Y)$.

The case (b) is excluded by the following lemma which we will use again later.

\begin{lemma}\label{lefinitesingsurfline} Let $X$ be a general cubic fourfold. Then any
hyperplane section $Y\subset X$ contains only finitely many cubic surfaces which are singular
along a line.
\end{lemma}

\begin{proof} Assume to the contrary that there is a curve $D$ of such surfaces and such lines $l'$ of singularities.
We note that any line $l'$ parametrized  by
a point $[l']\in D$ has to pass through a singular point of $Y$. Indeed, if $Y$ is smooth along
$l'$, then its Gauss map given by the partial derivatives of the defining equation
$f_Y$ of $Y$ in $H_Y$ is well-defined along $l'$, and
thus it cannot be constant along $l$, hence equivalently $Y\cap \mathbb{P}^3$ cannot be singular
at all points of $l'$ for any $\mathbb{P}^3$ containing $l$.
Next, a hyperplane section $Y$ of $X$ has finitely many singular points, hence we can assume
that in case (b), the curve $D$ consists of lines passing through a given
singular point $y$ of $Y$. By Lemma \ref{leFYirred}, the family $C_y$ of lines
in $X$ passing through
$y$ is a curve, and thus $D$ must be an irreducible  component of $C_y$.
In adequate homogeneous coordinates $X_0,\ldots,X_4$
on $H_Y$, the point $y$ has equations $X_i=0,\,i=0,\ldots,3$ and
$Y$ has equation $X_4Q(X_0,\ldots,X_3)+T(X_0,\ldots, X_3)$, where
$Q$ and $T$ are homogeneous of respective degrees $2,\,3$.
The curve $C_y$ of lines through $y$ (in $Y$ or $X$)
is defined by the equations $Q=T=0$.
Let $[l']\in C_y$ parametrize a line $l'$ in $Y$ such that some hyperplane
$H'$
in $H_Y$ containing $l'$ is tangent to $Y$ everywhere along $l'$. This
is saying that the equation
$f:=X_4Q(X_0,\ldots,X_3)+T(X_0,\ldots, X_3)$, restricted to a hyperplane
$H'$
of $\mathbb{P}^3$ passing through the point $[l']$,
has zero derivatives along $l'$. Thus the equations $Q$ and $T$ restricted to
$H'$, must have $0$ derivative at $[l']$. It follows that the two polynomials
$Q$ and $T$ have nonindependent derivatives at $[l']$, so that $[l']$ is a singular point
of the curve $C_y$. In conclusion, we found that under our assumption, the curve $C_y$ has a nonreduced component.
Hence the proof of Lemma \ref{lefinitesingsurfline} is concluded by the proof of  Lemma \ref{leredCy} below.
\end{proof}
\begin{lemma}\label{leredCy} If $X$ is general, the curve $C_y$ of lines through any point
$y\in X$ is reduced.
\end{lemma}
\begin{remark}{\rm It is not true that $C_y$ is irreducible for any $y$. Indeed, a general $X$ contains a cubic surface which is  a cone over an elliptic curve, with vertex $y\in X$. Hence the elliptic curve is an irreducible
component of $C_y$ in this case.}
\end{remark}
\begin{proof}[Proof of Lemma \ref{leredCy}]
 For any $y\in X$, the curve $C_y$ has degree $6$, and the cone over the curve $C_y$, with vertex $y$, is
a degree $6$ surface contained in $X$. We use now the fact that if $X$ is general, ${\rm Hdg}^4(X,\mathbb{Z})=\mathbb{Z}h^2$, where
$h=c_1(\mathcal{O}_X(1))$. Hence any surface in $X$ has degree divisible by
$3$. Applying this to the components of this cone, the only way the curve $C_y$ can be nonreduced is  if
$C_y$ is everywhere nonreduced with multiplicity $2$. The curve $C_{y,red}$ is then a curve of degree $3$
which can be either a plane cubic or a normal rational curve
of degree $3$. If $S$ is the cone over a normal cubic curve in $\mathbb{P}^3$, the
set of cubic hypersurfaces in $\mathbb{P}^5$ containing $S$ has codimension $ 22$ while the dimension
of the Hilbert scheme parametrizing  such an $S$ in $\mathbb{P}^5$ is $5+4+15-3=21$, so that a general
cubic does not contain such a surface.
In the case of the cone over an irreducible  plane cubic, the curve $C_y$ is the complete intersection
of a quadric and a cubic in $\mathbb{P}^3$ which contains an irreducible plane cubic with multiplicity $2$.
The only possibility is then that the quadric itself is a double plane. However, one can easily check  that
for general $X$, there is no point $y\in X$ where the Hessian of the defining equation
of $X$ defines a nonreduced quadric in $\mathbb{P}^3$.
\end{proof}

The proof of Proposition \ref{propgoodline} is now complete.
\end{proof}
\subsection{Existence of very good lines} For constructing compactified Jacobians (and similarly compactified Pryms) irreducibility assumptions are crucial. This leads us to the following strengthening of the requirement of good line.
\begin{definition} \label{defiverygood} Let $Y$ be a cubic threefold. We will say that a line $l\subset Y$ is {\it very good} if
$l$ is good (see Definition \ref{defigood}) and the curve $\widetilde C_l=\widetilde C_{l,Y}=\{{\rm lines\,\,in \,\,}Y\,\,{\rm meeting} \,\,l\}$ is irreducible.
\end{definition}
\begin{prop} \label{propirred} Let $X$ be a general cubic $4$-fold. Then for any hyperplane section
$Y\subset X$, there exists a line $l\subset Y$ such that
the curve $\widetilde C_{l,Y}$ is irreducible.
\end{prop}
\begin{cor} \label{coroverygood} If $X$ and $Y$ are as above, a general line in $Y$ is very good.
\end{cor}
\begin{proof}  Proposition
\ref{propgoodline} shows the existence of a good line, and this is an open property on $F(Y)$.
Proposition \ref{propirred} shows the existence a line $l\subset Y$ such that
the curve $\widetilde C_{l,Y}$ is irreducible
and this is also an open property on $F(Y)$. As we know by Lemma \ref{leFYirred} that $F(Y)$ is irreducible,
it follows that a general line is very good.
\end{proof}
\begin{proof}[Proof of Proposition \ref{propirred}] The incidence variety
$P_Y\subset F(Y)\times Y$ is a $\mathbb{P}^1$-bundle $p:P_Y\rightarrow F(Y)$ over $Y$.
We proved in Lemma \ref{leFYirred} that $X$ and $Y$ being as above, $F(Y)$ is irreducible and reduced, thus
$P_Y$ satisfies the same properties. In particular, the degree of the map
$q:P_Y\rightarrow Y$ is $6$ as for a smooth $Y$, and the degree of the map
$q'=pr_2: P_Y\times_YP_Y\setminus \Delta_{P_Y}\rightarrow P_Y$  is $5$.
 We have the following lemma:

\begin{lemma}\label{leintermediaire} Let $X$ be a general cubic fourfold and let $Y$ be any hyperplane section of $X$. If for all lines $l\subset Y$ the curve $\widetilde C_l$ is reducible, then
$P_Y\times_Y P_Y\setminus \Delta_{P_Y}$ has at least two irreducible components dominating $Y$.
\end{lemma}
\begin{proof} We observe first that for any line $l\subset Y$, the curve $\widetilde C_l\subset F(Y)$ (minus the point $[l]$ when
 $l$ is special for $Y$) identifies naturally  with
$q^{-1}(l)$ away from its intersection with the vertical curve $P_{Y,[l]}=p^{-1}([l])$. Indeed, $q^{-1}(l)$ is, away from
the vertical fiber $P_{Y,[l]}$, the set of pairs $([l'],x)$ such that $[l]\not=[l']$ and
$x\in l\cap l'$. The curve $\widetilde C_l$ (away from $[l]$ when $l$ is special) thus maps to it  via the
map $$[l']\mapsto ([l'],x),\,\{x\}=l\cap l'.$$  We will in fact see $\widetilde C_l$ (minus the point $[l]$) as contained in
$P_Y\times_YP_Y\setminus \Delta_{P_Y}$ by the
map
\begin{eqnarray}
\label{eqmap}[l']\mapsto ([l'],[l],x),\,\{x\}=l\cap l'.
\end{eqnarray}

Now suppose that $\widetilde C_l$ is reducible for all $l$. For general $l$, $\widetilde C_l$ is reduced and the morphism
$q'_l=q'_{\mid \widetilde C_l}:\widetilde C_l\rightarrow l$ has degree $5$. So,  for general $l$,  we must be in one of the following situations:
\begin{enumerate}
\item\label{item1} $\widetilde C_l$ has no component of degree $1$ over $l$ and has a unique component $\widetilde C_{l,2}$ of degree $2$ over $l$. We then denote $\widetilde C_{l,3}$ the Zariski closure of $\widetilde C_{l}\setminus \widetilde C_{l,2}$.
\item\label{item2} $\widetilde C_l$ has $1\leq k<4$ components $\widetilde C_{l,1,i}$ of degree $1$ over $l$ and the rest $\widetilde C_{l,rem}=\overline{\widetilde C_l\setminus\cup_i{\widetilde C_l,1,i}}$ has all its components of degree $>1$ over $\widetilde C_l$.
\item \label{item3} $\widetilde C_l$ has five components of degree $1$ over $l$.
\end{enumerate}
If case \ref{item1} or case \ref{item2} happens, then $P_Y\times_Y P_Y\setminus \Delta_{P_Y}$ has at least two
irreducible components dominating $P_Y$, namely the two varieties $\cup_{[l]\in F(Y)}\widetilde C_{l,2}$ and
$\cup_{[l]\in F(Y)}\widetilde C_{l,3}$
  in case \ref{item1}, and the two varieties
 $\cup_{[l]\in F(Y), 1\leq i\leq k}\widetilde C_{l,1,i}$ and $\cup_{[l]\in F(Y)}\widetilde C_{l,rem}$ in case \ref{item2}.
 It thus suffices to show that \ref{item3} cannot happen. This however follows from
Proposition \ref{propgoodline}. Indeed we may assume that $l$ is good, so that the involution on
$\widetilde C_l$ has no fixed point. But if $\widetilde C_l$ has five irreducible components all isomorphic
to $\mathbb{P}^1$, one of them is fixed under the involution which then has fixed points.
\end{proof}

Proposition
\ref{propirred} now follows from Lemma \ref{leintermediaire} and from the  following:
\begin{lemma} \label{leirredfiberprod} If $X$ is general, for any hyperplane section $Y$ of $X$,
the variety $P_Y\times_Y P_Y\setminus \Delta_{P_Y}$ is irreducible.
\end{lemma}
Let us prove a few intermediate statements:
\begin{lemma}\label{lecompsurj} If $X$ is general, and $Y$ is any hyperplane section of $X$,
any irreducible component of $P_Y\times_Y P_Y\setminus \Delta_{P_Y}$ dominates $P_Y$ by the second projection.
\end{lemma}
\begin{proof} The only possibility for an irreducible component $Z$ of $P_Y\times_Y P_Y$ not to dominate $P_Y$
by the second projection is if there is a curve $W\subset Y$ such that for any $y\in W$, there is a curve $D_y$
of lines in $Y$
passing through  $y$. However, this is impossible by the following claim.
\begin{claim}\label{claim} For general $X$, and for any hyperplane section $Y$ of $X$, there are only finitely many points
$y\in Y$ such that the set  of lines in $Y$ through $y$ contains a curve $D_y$.
\end{claim}
\begin{proof} The stated property is Zariski open, so it suffices to
prove it assuming $X$ is very general. Assume by contradiction that this set is a curve $W$. Then the union over
$y\in W$ of the $D_y$'s would then be a surface contained in $F(Y)$, and since we know by Lemma \ref{leFYirred} that $F(Y)$ is irreducible,
this surface would be the whole of $F(Y)$. Thus $Y$  has the property
that any line in $Y$ meets the curve $W\subset Y$. But the general  point $[l]\in F(Y)$ is  a smooth
point of $F(Y)$ parametrizing a line $l$
with normal bundle $N_{l/X}\cong \mathcal{O}_l\oplus \mathcal{O}_l$, which means that the morphism
$q:P_Y\rightarrow Y$ is \'etale in a neighborhood of  the fiber
$P_{Y,[l]}$ of $P_Y$ over $[l]\in F(Y)$, so that  the general  deformation of $l$
in $Y$ does not intersect $Z$.
\end{proof}
Lemma \ref{lecompsurj} is thus proved.
\end{proof}
The variety $P_Y\times_Y P_Y\setminus \Delta_{P_Y}$ has several rational involutions. We will denote by
$\tau$ the involution exchanging factors and by $i$
 the  involution of $P_Y\times_Y P_Y\setminus \Delta_{P_Y}$ which maps
 $(l_1,l_2),\,l_1\cap l_2\not=\emptyset$ to $(l_3,l_2)$ where $l_3$ is the residual line
 of the intersection $P_{l_1,l_2}\cap Y$, $P_{l_1,l_2}$ being the plane generated by $l_1$ and $l_2$.
 Recall from (\ref{eqmap}) that the fiber of the composite map
 $$ P_Y\times_Y P_Y\stackrel{pr_2}{\rightarrow} P_Y\stackrel{p}{\rightarrow} F(Y)$$
 over $[l]\in F(Y)$ identifies with  the curve $\widetilde C_l$ of lines in $Y$ meeting $l$ and that
 $i$ acts on $\widetilde C_l$ as the Prym involution. The quotient $\widetilde C_l/i$ is the discriminant curve
 $C_l$ of the conic bundle $\pi_l:\widetilde{Y}_l\rightarrow \mathbb{P}^2$.
\begin{lemma}\label{leClirred} If $X$ is general, and $Y$ is any hyperplane section
of $X$, the quotient $$(P_Y\times_Y P_Y\setminus \Delta_{P_Y})/i$$ is irreducible.
 \end{lemma}
\begin{proof}  Using Lemma \ref{lecompsurj}, it suffices to show that, $X$ and $Y$ being as above, for general
$[l]\in F(Y)$, the curve $C_l$ is irreducible.
The curve $C_l$ is a quintic curve, so if it is not irreducible, it must decompose either
as

(a) the union of a line and a quartic, or

(b)  the union of a smooth conic and a cubic.

Case (a) is excluded as follows. A line component in the discriminant curve
provides a cubic surface $S=P_3\cap Y$, where $P_3$ is a $\mathbb{P}^3$ contained in $H_Y$, which contains a one parameter family of lines meeting $l$.
The surface $S$ is irreducible because $X$ is general so does not contain a plane or a quadric surface.
 Furthermore $S$ cannot be a cone over an elliptic curve, because $l$ is general, hence by Claim  \ref{claim}, through any point
$y$ of $l$ there are only finitely many lines in $Y$ through $y$.
By assumption, for any plane $P\subset P_3$ containing $l$, the intersection $P\cap S$ contains $l$ and a residual
reducible conic. The singular point of the conic moves in $S$, hence by Bertini,
 the singular locus of $S$ consists of a curve $\overline{Z}$ and thus must be a line $l'$. Indeed,
  any bisecant line to $\overline{Z}\subset {\rm Sing}\,S$ is contained in $S$, and the only alternative possibility would be that
 $\overline{Z}$ is a conic and  $S$ has a component which is a plane which  is excluded since $X$ is general. The line $l'$
   is  then a special line of $X$
 whose associated $\mathbb{P}^3_{l'}=\cap_{x\in l'}T_{X,x}$ is equal to $P_3$.
We know by Lemma \ref{lefinitesingsurfline} that there are finitely many such $\mathbb{P}^3_{l'}$ contained in $H_Y$ (or equivalently, cubic surfaces singular along a line
$l'$ and contained in $Y$), so the general line
$l$ in $Y$ cannot be contained in such a $\mathbb{P}^3_{l'}$.

Case (b) is excluded as follows : Suppose the discriminant curve $C_l$ has a component which is a smooth conic
$C$. As $l$ is a good line,
the double cover $r:\widetilde C_l\rightarrow C_l$ is \'etale, hence split over $C$: $r^{-1}(C)=C_1\cup C_2$.
Let $\pi_l:\widetilde{Y}_l\rightarrow \mathbb{P}^2$ be the linear projection from $l$, and let
$T:=\pi_l^{-1}(C)$. Then $T$ is a reducible surface, $T=T_1\cup T_2$, where $T_1$ is swept-out by lines in $C_1$ and
$T_2$ is swept-out by lines in $C_2$. On the other hand, as $X$ is very general,
any surface in $X$ has degree divisible $3$, thus $T_1$ and $T_2$ must have degree $3$.
The surfaces $T_1$ and $T_2$ are ruled surfaces using their $1$-parameter family of lines intersecting $l$, and none of them can
be contained in a projective subspace $P_3\subset H_Y$, since otherwise ${P}_3$ would contain
$l$ and thus would project via $\pi_l$ to a line in $\mathbb{P}^2$ while the image  $\pi_l(T_i)$ is our smooth
conic.

Finally, a ruled nondegenerate degree $3$ surface in $\mathbb{P}^4$ is a cone over a rational
normal curve or a projection of a Veronese surface from one of its points. The first case corresponds to the vector bundle $\mathcal{O}\oplus \mathcal{O}(3)$ on
$\mathbb{P}^1$ and the second case corresponds to the  vector bundle $\mathcal{O}(1)\oplus \mathcal{O}(2)$
on $\mathbb{P}^1$. We already explained by counting  parameters  that the general cubic hypersurface $X$ in $\mathbb{P}^5$
does not contain a cone over a rational cubic curve in $\mathbb{P}^3$. It is also true
that the general cubic hypersurface $X$ in $\mathbb{P}^5$
does not contain the projection of a Veronese surface from one of its point, but this does not follow from
an immediate dimension count. One has to argue as follows: this surface $V$ is smooth with
$c_1(V)^2=8$, $c_2(V)=4$. If $V\subset X$, the normal bundle
$N_{V/X}$ fits into the exact sequence
\begin{eqnarray}\label{eqex1}0\rightarrow T_V\rightarrow T_{X\mid V}\rightarrow N_{V/X} \rightarrow 0.
\end{eqnarray}
As $$c_2(T_X)=6h^2,\,\,c_1(T_X)=3h,$$
one gets from (\ref{eqex1}) the following equalities in $H^*(V,\mathbb{Q})$:
\begin{eqnarray}\label{eqex2} c_1(N_{V/X})=-c_1(T_V)+3h_V,\,\,\,c_2(N_{V/X})=-c_2(T_V)-c_1(N_{V/X})\cdot c_1(T_V)+6h_V^2
\\
\nonumber
=-c_2(T_V)-(-c_1(T_V)+3h_V)\cdot c_1(T_V)+6h_V^2
,\end{eqnarray}
where $h_V=c_1(\mathcal{O}_V(1))$ and $h_V^2=3$.
In the ruled surface $T=\mathbb{P}(\mathcal{O}_{\mathbb{P}^1}(1)\oplus \mathcal{O}_{\mathbb{P}^1}(2))\stackrel{\pi}{\rightarrow}{\mathbb{P}^1}$, one has
$$K_V=-c_1(T_V)=-2h_V+\pi^*\mathcal{O}_{\mathbb{P}^1}(3),$$
which combined with (\ref{eqex2}) gives
$$c_2(N_{V/X})=-4-(-2h_V+\pi^*c_1(\mathcal{O}_{\mathbb{P}^1}(3))+3h_V)\cdot (2h_V-\pi^*c_1(\mathcal{O}_{\mathbb{P}^1}(3)))+18=5
.$$
This shows that the self-intersection of $V$ in $X$ is equal to $5$ so that the class of $V$ is not equal to $h^2$.
Hence such a surface does not exist for general $X$.
\end{proof}
We get the following corollary (where
again $X$ is general and $Y$ is any hyperplane section of $X$):
\begin{cor} \label{coro2comp} \begin{itemize} \item[(i)] The  fibered product $P_Y\times_YP_Y\setminus \Delta_{P_Y}$ has at most two irreducible components, and if it is reducible, they are exchanged by the rational involution $i$ acting on each curve $\widetilde C_l\times[l]\subset P_Y\times_YP_Y$.
\item[(ii)] Let $\tau$ be the involution of $P_Y\times_YP_Y\setminus \Delta_{P_Y}$ exchanging factors. If $P_Y\times_YP_Y\setminus \Delta_{P_Y}$ is reducible,
its two components are exchanged by $\tau$.
\end{itemize}
\end{cor}
\begin{proof} (i) is an immediate consequence of Lemmas \ref{leClirred} and
\ref{lecompsurj}. The proof of (ii) goes as follows:
We observe that the two rational involutions $\sigma$ and $\tau$ are
part of an action of the symmetric group $\mathfrak{S}_3$ on $P_Y\times_YP_Y\setminus \Delta_{P_Y}$ by
birational maps. Indeed, $P_Y\times_YP_Y\setminus \Delta_{P_Y}$ can also be seen as the set
of labelled triangles, that is, triples $(l_1,\,l_2,\,l_3)$ of lines  in $Y$, such that
 for some plane $P\subset H_Y$, $P\cap Y=l_1+l_2+l_3$. The action of $\mathfrak{S}_3$
is simply the permutation of the labels.
Note that these birational maps are well defined at each generic point of $P_Y\times_YP_Y\setminus \Delta_{P_Y}$
by Lemma \ref{lecompsurj}.
The involution $i$ is the involution $(l_1,l_2)\mapsto (l_3,l_2)$ while the involution
$\tau$ is the involution $(l_1,l_2)\mapsto (l_2,l_1)$. In any case, these two involutions are conjugate
in $\mathfrak{S}_3$. We know by (i) that if there are two irreducible
 components, they are exchanged by $i$. Hence they must be also
exchanged by $\tau$.
\end{proof}
\begin{proof}[Proof of Lemma \ref{leirredfiberprod}] let $X$ be a general cubic fourfold and
let $Y$ be any hyperplane section of $X$. Assume by contradiction that $P_Y\times_YP_Y\setminus \Delta_{P_Y}$ is not irreducible.  By Corollary
\ref{coro2comp}, it has then exactly two components
$\mathcal{C}_1,\,\mathcal{C}_2$. Both components dominate
$Y$ by Lemma \ref{lecompsurj}. Let $1\leq k_1<k_2,\,k_1+k_2=5$ be the respective degrees of $pr_2:\mathcal{C}_1
\rightarrow P_Y,\,\mathcal{C}_2\rightarrow P_Y$. One has $(k_1,k_2)=(2,3)$ or $(k_1,k_2)=(1,4)$.
For a general point $y\in Y$, denote by $\{l_1,\ldots,l_6 \}=q^{-1}(y)=:E_y\subset P_Y$.
For $l_i\not= l_j$, we have $(l_i,l_j)\in P_Y\times_YP_Y\setminus \Delta_{P_Y}$ and
thus we can write $E_y\times E_y\setminus \Delta_{E_y}$ as a disjoint union
\begin{eqnarray}\label{equnion} E_y\times E_y\setminus \Delta_{E_y}=E_{1,y}\sqcup E_{2,y},
\end{eqnarray}
where $$E_{1,y}:=(E_y\times E_y\setminus \Delta_{E_y})\cap \mathcal{C}_1,\,\,E_{2,y}:=(E_y\times E_y\setminus \Delta_{E_y})\cap \mathcal{C}_2.$$
The partition (\ref{equnion}) satisfies:

(a) For any $i\in\{1,\ldots,6\}$, the set of $j\not=i$ such that $(l_i,l_j)\in E_{1,y}$ has cardinality $k_1$ and the set
of $j\not=i$ such that $(l_i,l_j)\in E_{2,y}$ has cardinality $k_2$.

(b)  For any $i,\,j\in\{1,\ldots,6\}$ with $i\not=j$, $(l_i,l_j)\in E_{1,y}\Leftrightarrow (l_j,l_i)\in E_{2,y}$.

Indeed, (b) is exactly Corollary \ref{coro2comp}(ii).

The contradiction is now obvious: Indeed, (a) shows that the cardinality of $E_{1,y}$
is $6k_1$ and the cardinality of $E_{2,y}$ is $6k_2$, with
$6k_1\not=6k_2$, while (b) implies that both sets have the same cardinality.
\end{proof}
Proposition \ref{propirred} is thus proved.
\end{proof}

\section{Transversality arguments} \label{sectrans18janvier}
This section is devoted to  applying  transversality arguments in order to deduce that some statements
which hold in large codimension for   cubic threefolds in $\mathbb{P}^4$  hold for {\it any} hyperplane sections
 of a general cubic fourfold. In particular, we will first prove  Lemma
\ref{letransXgen}  which guarantees that
the versality statement of \cite{casaetal}  is actually satisfied by the family of quintic curves associated
to the family of hyperplane sections of a general cubic fourfold and a local choice of good lines in them.
This will be needed in  Section \ref{sectprym}.
In Section \ref{sectranpf}, we will extend this result to the case of a general Pfaffian cubic fourfold.
By applying a  similar transversality argument, we will also prove  the existence of a very good line in any hyperplane section
of a  general Pfaffian cubic fourfold. This will be needed in order to make the arguments of Section \ref{sectprym}
apply as well when the cubic fourfold is  a general Pfaffian cubic (see Section \ref{sectpfaffian}).
\subsection{Transversality results for general cubic fourfolds}
 Below, we denote by $\Hyp_{4,3}^0$ the open subset of
$\mathbb{P}(H^0(\mathbb{P}^5,\mathcal{O}_{\mathbb{P}^5}(3)))$ parametrizing smooth cubic fourfolds, and
by $\Hyp_{3,3}$ the projective space $\mathbb{P}(H^0(\mathbb{P}^4,\mathcal{O}_{\mathbb{P}^4}(3)))$.
By restriction from $\mathbb{P}^5$ to a given $\mathbb{P}^4\subset \mathbb{P}^5$ we get a morphism
$$r:\Hyp_{4,3}^0\rightarrow \Hyp_{3,3}$$
which is obviously smooth, since the fibers are Zariski open subsets of $H^0(\mathbb{P}^5,\mathcal{O}_{\mathbb{P}^5}(2))$. With these notations, the following transversality lemma holds:
\begin{lemma} \label{legentrans} \begin{itemize}\item[(i)] Let $Z\subset \Hyp_{3,3}$ be a closed algebraic subset of
codimension $\geq 6$, which is invariant under the action of
$\PGL(5)$. Then for a general $[X]\in \Hyp_{4,3}$, no hyperplane section of $X$ is isomorphic to a cubic threefold $Y$ parametrized  by a point of $Z$.
\item[(ii)] Let $\mathcal{M}'\subset \Hyp_{4,3}^0$ be a  hypersurface which is invariant under
$\PGL(5)$, and let $Z\subset \Hyp_{3,3}$ be a closed algebraic subset of
codimension $\geq 7$, which is invariant under the action of
$\PGL(5)$. Then for a general $[X]\in \mathcal{M}'$, no hyperplane section of $X$ is isomorphic to a cubic threefold  $Y$ parametrized  by a point of $Z$.
\end{itemize}
\end{lemma}
\begin{proof} (i) Indeed, as $r$ is smooth (actually, flat would suffice), $r^{-1}(Z)\subset \Hyp_{4,3}^0$ has codimension $6$ in $\Hyp_{4,3}^0$. The group $\PGL(6,5)\subset \PGL(6)$ of automorphisms of $\mathbb{P}^5$ preserving $\mathbb{P}^4$
acts on $\Hyp_{4,3}^0$ preserving $r^{-1}(Z)$ since $Z$ is invariant under $\PGL(5)$.
It thus follows that
$${\rm dim}\,\PGL(6)\cdot r^{-1}(Z)\leq 5+{\rm dim}\,r^{-1}(Z),$$
or equivalently that ${\rm codim}\,\PGL(6)\cdot r^{-1}(Z)\geq 1$. Thus $\PGL(6)\cdot r^{-1}(Z)$ is not open in $\Hyp_{4,3}^0$, which proves (i).

(ii)   The same argument as in (1) shows that
$\PGL(6)\cdot r^{-1}(Z)$ has codimension at least $2$ in $\Hyp_{4,3}^0$, hence cannot contain the hypersurface
$\mathcal{M}'$.
\end{proof}
\begin{remark} We will see in Section \ref{sectranpf} an improved version of Lemma \ref{legentrans} (ii), where under a certain assumption on the hypersurface $\mathcal{M}'$,
the estimate on codimension of $Z$ will be also $6$, not $7$. The hypersurface of interest for us will be the locus of Pfaffian cubics.
\end{remark}

The above lemma will allow  us to exclude from our study highly singular cubic threefolds
 and to restrict ourselves to mildly singular cubic threefolds with  the following precise meaning:
\begin{definition}\label{defallow} Let $Y$ be a cubic threefold. We say $Y$ is {\it allowable} (or mildly singular) if $Y$ has at worst isolated singularities and $\tau_{tot}(Y)\le 6$.   Here $\tau$ denotes the Tjurina number of an isolated hypersurface singularity, and $\tau_{tot}(Y)$ is the sum of the associated Tjurina numbers, i.e. $\tau_{tot}(Y)=\sum_{p\in \Sing(Y)}\tau({Y_p})$ (where $Y_p$ denotes the germ of $Y$ at $p$).
\end{definition}
\begin{remark}\label{deftjurina}
We recall that for an isolated hypersurface singularity $(V(f),0)\subset \bC^n$, the Tjurina number  is defined to be $\tau(f)=\dim_\bC \bC[x_1,\dots,x_n]/\langle f,\frac{\partial f}{\partial x_1},\dots, \frac{\partial f}{\partial x_n}\rangle$, and  is the expected codimension in moduli to encounter that singularity. By dimension count, we thus expect that all hyperplane sections $Y$ of a general cubic fourfold $X$ are allowable in the sense of Definition \ref{defallow}. The results below say that this is indeed the case.
\end{remark}
\begin{remark}
The arguments involving good lines are closely related and inspired by those in \cite{casaetal}. However, as already noted in Remark \ref{diffcml}, the results of \cite{casaetal} do not suffice here. Namely, in \cite{casaetal} the focus was on GIT stable/semistable cubic threefolds, while here we focus on hyperplane sections of general cubic fourfolds (or general Pfaffian cubics). Thus, our notion of allowable is slightly different from that of \cite[Definition 2.2]{casaetal}.
\end{remark}

 \begin{prop}\label{proallow}
Let $Y$ be an allowable cubic threefold in the sense of  Definition \ref{defallow}. Then, the following hold:
\begin{itemize}
\item[(0)]  $Y$ has at worst ADE (in particular planar) singularities.
\item[(1)] The deformations of $Y$ in $\mathbb{P}^4$  induce  a simultaneous versal deformation of the singularities of $Y$.  This means
equivalently that
the natural map from the first order deformation space of $Y$, that is $H^0(Y,\mathcal{O}_Y(3))$, to
the product $\prod_{p\in  {\rm Sing}\,Y}T^1_{{Y_p}}$, where $T_{{Y_p}}^1$ classifies the first order deformations
of the germ of singularities of $Y$ at $p$, is surjective.
\item[(2)]  Assume additionally that there exists a good line $l\subset Y$. Let
 $\widetilde C_l$ be the curve of lines in $Y$ meeting $l$.
Then the singular points of $Y$ are in bijection with the singular points of the curve
$C_l=\widetilde C_l/\iota$ (which is  a plane quintic curve), the analytic types of corresponding singularities of
 $Y$ and $C_l$ coincide and the deformation theory of corresponding singular points of $Y$
and $C_l$ coincide. Furthermore, the deformations of $C_l$ give  simultaneous versal deformations of the singularities of $C_l$ (which is compatible with the deformations of the singularities of $Y$).
\item[(3)]  The locus of cubic  hypersurfaces $Y\subset \mathbb{P}^4$ with non-allowable singularities
has codimension $\ge 7$ in the space $\Hyp_{3,3}$ of all cubic threefolds.
\item[(4)] $Y$ has finite stabilizer.
\end{itemize}
\end{prop}
\begin{proof}
A non-ADE hypersurface singularity has Tjurina number $\tau\ge 7$, giving (0).

(1) The simultaneous versality statement (1) is a specialization of a result of Shustin--Tyomkin \cite[Main Theorem]{shustin}  to the case of cubic threefolds (in fact, $\tau_{tot}(Y)\le 7$ suffices; see also \cite[Lemma 3.3(i)]{dpw} which gives the stronger results that $\tau_{tot}(Y)\le 15$ suffices for cubic threefolds;  \cite[(5) on p. 35]{casaetal} gives the simultaneous versality for GIT stable cubics).

(2) The correspondence of singularities under the projection from a good line is
 \cite[Proposition 3.6]{casaetal}. Clearly we get $\tau_{tot}(C_l)\le 6$, which then implies that the deformations of $C_l$ give  simultaneous versal deformations of its singularities (in fact, \cite{shustincurves} says that $\tau_{tot}(C_l)< 4(d-1)=16$ suffices). Finally, the compatibility between the global-to-local deformations of $Y$ and $C_l$ is discussed in \cite[\S3.3]{casaetal2}.

 (3) The expected codimension for the equisingular deformations of a singular cubic $Y$ is $\tau_{tot}(Y)$. Thus, the locus of cubics with $\tau_{tot}(Y)\ge 7$ is expected to have codimension $7$ in $\Hyp_{3,3}$. The simultaneous versality statements cited above (esp. \cite[Lemma 3.3(i)]{dpw}) guarantee that the expected codimension is the actual codimension for cubics with $\tau_{tot}(Y)\le 15$. It remains to check that the more degenerate cases (cubics with $\tau_{tot}(Y)> 15$ or non-isolated singularities) have still codimension $\ge 7$. This is an easy case by case analysis that we omit (the main tool for this analysis is to study a singular cubic via the associated $(2,3)$ complete intersection in $\bP^3$, see \cite[\S3.1]{casaetal}).

(4) The cases when $Y$ is GIT semistable (in particular, if $Y$ has at worst $A_1,\dots,A_5$ or $D_4$ singularities) and has positive dimensional stabilizer are classified by Allcock (e.g. \cite{allcock}). It follows that either $\tau_{tot}(Y)\ge 10$ or $Y$ has non-isolated singularities (in fact $Y$ is the chordal cubic). Assuming that $Y$ is not GIT semistable and that $\tau_{tot}(Y)\le 6$ leads to a small number of cases that can be excluded by a case by case analysis. Namely, $Y$ is stabilized either by $G_m$ or $G_a$. The unipotent case can be seen not to occur using the classification of \cite{dpw2}. Finally, if $Y$ is stabilized by $G_m$ it has at least two singularities. Under our assumptions (in particular, GIT unstable), the only possibility is that $Y$ has a $D_5$ singularity and an $A_1$ singularity, which can be then excluded by studying the associated $(2,3)$ curve obtained by projecting from the $A_1$ singular point.
\end{proof}

An immediate consequence of Corollary \ref{coroverygood}  and Proposition \ref{proallow}, (3) is:
\begin{cor}\label{corallow} If $X$ is a general cubic fourfold, any hyperplane section $Y$
of $X$ is allowable, hence satisfies properties (0)-(4) of Proposition \ref{proallow}. Moreover, $Y$ has a very good line.
\end{cor}

Note however that since we are restricting to the universal family $\cY/B$ of hyperplane sections of a fixed cubic fourfold $X$, the simultaneous versal statement of  Proposition \ref{proallow}(1) does not suffice for our purposes. What is needed instead is the following lemma which follows  from Proposition \ref{proallow} and a transversality argument.

\begin{lemma} \label{letransXgen} Let $X\subset \mathbb{P}^5$ be a general cubic fourfold, and let $Y$ be any hyperplane section
of $X$. Then  the natural morphism
$H^0(Y,\mathcal{O}_Y(1))\rightarrow \prod_{p\in {\rm Sing}\,Y}T^1_{{Y_p}}$ is surjective. In other words, the  family of deformations of
$Y$ in $X$ induces a versal deformation of the  singularities of $Y$.
\end{lemma}
Combined with Proposition \ref{proallow}, (2), Lemma \ref{letransXgen}  gives:
\begin{cor} \label{corotransdef} In the situation of Lemma \ref{letransXgen},  denote by
$\mathcal{F}_{good}$ the universal family of good lines in hyperplane sections of $X$:
$$\mathcal{F}_{good}=\left\{([l],t)\in G(2,6)\times B\mid \, l \textrm{ is a good line of } Y_t\right\},$$
Then  if $X$ is general, $Y_0\subset X$ is any hyperplane section and  $l\subset Y_0$ is a general good line,
the natural map $T_{\mathcal{F}_{good},([l],0)}\rightarrow \oplus_{p\in{\rm Sing}\,C_{l,Y_0}} T^1_{C_{l,p}}$ is surjective.
Furthermore, for a   local analytic or \'etale section
$B\subset \mathcal{F}_{good}$ of the second projection defined near $0$, the natural map $T_{B,([l],0)}\rightarrow \oplus_{p\in{\rm Sing}\,C_{l,Y_0}} T^1_{C_{l,p}}$ is surjective.
\end{cor}

\begin{proof}[Proof of Lemma \ref{letransXgen}] Using Lemma \ref{legentrans} and Proposition \ref{proallow}, we see that there exists a (non-empty) Zariski open subset
$\Hyp_{4,3}^{00}\subset \Hyp_{4,3}^{0}$ of the space of cubic fourfolds such that: (i) $X$ has trivial automorphism, and (ii) any hyperplane section $Y$ of $X$ has finite stabilizer and satisfies property (1) of Proposition \ref{proallow} (i.e. the space of cubic threefolds $\Hyp_{3,3}$ gives a simultaneous versal deformation of the singularities of $Y$). We want to obtain the stronger statement that the hyperplane sections of $X$ (giving a $\bP^5$ non-linearly embedded in $\Hyp_{3,3}$) give a versal deformation of the singularities of $Y$.

Let $k$ be a number and
 $z=(z_1,\ldots,z_k)$ be the data of $k$ analytic isomorphism classes of germs of allowable hypersurfaces singularities. Let
 $N:=\sum_i{\rm dim}\,T^1_{z_i}(=\tau_{tot})$. Let
 $\Hyp_{3,3,z}\subset \Hyp_{3,3}$ be the set of cubic threefolds admitting exactly $k$ singular points with local germs $z_i$. Note that, by Lemma \ref{letransXgen}, we can assume $N\le 6$, i.e. a cubic threefold $Y$ with higher $N$ will not occur as a hyperplane section of a cubic fourfold $[X]\in \Hyp_{4,3}^{0,0}$.
 Using property (1), we conclude that $\Hyp_{3,3,z}$ is smooth locally closed of codimension $N$ in
 $\Hyp_{3,3}$. It follows that its inverse image $r^{-1}(\Hyp_{3,3,z})$ is smooth of codimension
 $N$ in $\Hyp_{4,3}^{00}$. The group $\PGL(6,5)$ of
 automorphisms of $\mathbb{P}^5$ preserving $\mathbb{P}^4$
 acts now on $\Hyp_{4,3}^{00}$ preserving $r^{-1}(\Hyp_{3,3,z})$ and
 using the definition of
  $\Hyp_{4,3}^{00}$, we find that the fiber $L_X$ over a general point  $[X]\in \Hyp_{4,3}^{00}/\PGL(6)$ of the quotient map
  $$\Hyp_{4,3}^{00}/\PGL(6,5)\rightarrow \Hyp_{4,3}^{00}/\PGL(6)$$
  is smooth, isomorphic to $\mathbb{P}(H^0(\mathcal{O}_X(1))$. (This statement is in fact not completely correct due to
  the presence of hyperplane sections of $X$ which have finite automorphisms, but it is true at the infinitesimal
  level.)
    Sard's theorem then
  tells us that for general $X$, the locus
  $L_X\cap r^{-1}(\Hyp_{3,3,z})$ is smooth of codimension
 $N$ in $L_X$, which exactly means that for the given type
 $z$, and for any $Y\subset X$ having $z$ as singularities, the map
  $H^0(Y,\mathcal{O}_Y(1))\rightarrow \prod_{p\in {\rm Sing}\,Y}T^1_{{Y_p}}$ is surjective.
  The conclusion then follows from the fact that there are finitely many analytic isomorphism classes of allowable singularities (by Proposition \ref{proallow}, (0), all are ADE with $\tau_{tot}\le 6$).
 \end{proof}
\subsection{Transversality results in the Pfaffian case}\label{sectranpf}
Recall that a Pfaffian cubic  hypersurface is a linear section of the Pfaffian cubic hypersurface in
$\mathbb{P}^{14}=\mathbb{P}(\bigwedge^2W_6)$ defined by the vanishing of  $\omega^3$ in $\bigwedge^6W_6$.
Pfaffian cubic fourfolds are parametrized  by  a hypersurface $\mathcal{P}$ in the moduli space of all smooth cubic fourfolds
(see \cite{bedo}). Restricting to cubic fourfolds without automorphisms, this hypersurface is smooth away from the locus where the cubic has two different Pfaffian structures. In general, it is a divisor with normal crossings, with one branch for each Pfaffian
structure. This follows from the fact that the period map for cubic fourfolds is \'etale.
As we want to apply the results of Section \ref{sectprym} also to the case
of a general Pfaffian cubic fourfold, we have to prove that a general Pfaffian cubic fourfold
satisfies the needed assumptions, namely
Lemmas \ref{lepftrans} and \ref{lepfgood}, that will be obtained as easy consequences of the following
lemma.

\begin{lemma} \label{letechnique18jan} Let $X$ be a general Pfaffian cubic fourfold. Then
for any hyperplane section $Y\subset X$ with equation
$f_Y\in H^0(X,\mathcal{O}_X(1))$, the subspace
$f_Y H^0(X,\mathcal{O}_X(2))\subset H^0(X,\mathcal{O}_X(3))$ is not contained in the tangent space to the Pfaffian hypersurface at the point $[X]$.
\end{lemma}
The Pfaffian locus $\mathcal{P}$ is an open set in the hypersurface
    $\mathcal{C}_{14}$ in the space
of all cubic fourfolds parametrizing  special cubics with discriminant $14$
(see \cite{hassett}). The subspace $f_Y H^0(X,\mathcal{O}_X(2))\subset H^0(X,\mathcal{O}_X(3))$ is the
space of first order deformations of $X$ containing $Y$.

\begin{proof}[Proof of Lemma \ref{letechnique18jan}] The Pfaffian cubic fourfolds are characterized by the fact that they contain quintic del Pezzo
surfaces: if $X$ contains a quintic del Pezzo surface $\Sigma$, the Pfaffian rank $2$ vector bundle $\mathcal{E}$ with $c_2=2$ and $c_1=0$ on $X$ is  deduced from $\Sigma $ by the Serre construction. Conversely, if $X$ is Pfaffian with Pfaffian rank $2$ vector bundle $\mathcal{E}$, there is a five dimensional
family of quintic del Pezzo surfaces $\Sigma$ in $X$, obtained as zero-sets of sections of $\mathcal{E}$ (see \cite[Proposition 9.2]{beau}).
Let $\Sigma\subset X$ be such a pair, and let $\sigma:=[\Sigma]\in H^2(X,\Omega_X^2)$ be the cohomology class of $\Sigma$. The numerical condition characterizing the Pfaffian class $\sigma$ is
$\sigma^2=13$ and $\sigma\cdot h^2=5$, where we use the intersection pairing
on $H^4(X,\mathbb{Z})$ and $h=c_1(\mathcal{O}_X(1))$.
The cup-product with $\sigma$ induces a composite morphism
\begin{eqnarray}\label{eqvhs}
\sigma\cup: H^0(X,\mathcal{O}_X(3))\stackrel{\rho}{\rightarrow} H^1(X,T_X)\rightarrow  H^{1,3}(X),
\end{eqnarray}
where the first map $\rho$ is the Kodaira-Spencer map,
and the general theory of variations of Hodge structures tells us that
the tangent space to the Pfaffian locus $\mathcal{P}$ at
$[X]$ identifies to ${\rm Ker}\,\sigma\cup$. Note that
$\rho$  identifies to
the quotient map
$$H^0(X,\mathcal{O}_X(3))\rightarrow R^3_f:=H^0(X,\mathcal{O}_X(3))/J_f^3,$$
where $f$ is the defining equation for $X$ and $J_f^3$ is the  degree $3$ piece of the  Jacobian ideal of $f$.
Griffiths' residue theory (see \cite[II,6.2]{voisinbook}) provides  isomorphisms
$$H^{2,2}(X)_{prim}\cong R^3_f,\,\,H^{1,3}(X)\cong R^6_f$$
such that the second map in (\ref{eqvhs}) identifies to multiplication
by $\tilde{\sigma}:R^3_f\rightarrow R^6_f$, where $\tilde{\sigma}\in R^3_f$ is the representative of
$\sigma$, or rather of its projection in $H^{2,2}(X)_{prim}=H^{2,2}(X)/\langle h^2\rangle$.
Lemma \ref{letechnique18jan} can thus be rephrased as follows:
For a general Pfaffian cubic fourfold
with equation $f$,
the Pfaffian class $\tilde{\sigma}\in R^3_f$ is not annihilated by
$yR^2_f$, for any $0\not=y\in H^0(X,\mathcal{O}_X(1))$.
Note that by Macaulay's theorem \cite[II,6.2.2]{voisinbook}, to say that
$y\tilde{\sigma}R^2_f=0$ in $R^6_f$ is equivalent to saying that
$y\tilde{\sigma}=0$ in $R^4_f$.
So what we have to prove is the following:

\begin{claim*} For a general Pfaffian cubic fourfold
with equation $f$ and (primitive)
 Pfaffian class $\tilde{\sigma}\in R^3_f$, and for any
$0\not=y\in H^0(X,\mathcal{O}_X(1))$, $y\tilde{\sigma}\not=0$ in $R^4_f$.
\end{claim*}

 In order to prove the claim, we use  the fact\footnote{We are grateful to the referee for pointing out this fact, which simplified our original argument.} (see  \cite{hassett}) that cubic fourfolds containing two non-intersecting planes
 $P_1,\,P_2$
 are parametrized by points in the closure of the divisor $\mathcal{C}_{14}$.
 In fact, if $p_i$ is the cohomology class of $P_i$, one has
 $$p_i^2=3,\,p_1\cdot p_2=0,\,h^2\cdot p_i=1$$
 and thus $\sigma=h^2+p_1+p_2$ satisfies the numerical conditions
 $\sigma^2=13,\,\sigma\cdot h^2=5$.
 It thus suffices  to  prove that for a   general cubic fourfold containing two nonintersecting   planes
 $P_1,\,P_2$  and for any $0\not=y\in S^1:=H^0(X,\mathcal{O}_X(1))$,  the class
 $h+p_1+p_2\in H^{2,2}(X)_{prim}=R^3_f$ satisfies
 $y(h+p_1+p_2)\not=0$ in $R^4_f$.
 This computation can be made explicitly on the Fermat cubic $X_f$
 with equation $f=\sum_{i=0}^5x_i^3$, where such configurations of planes are easy to exhibit: we can take
 $P_1$ to be defined by $x_0=\zeta x_1,\,x_2=\zeta x_3,\, x_4=\zeta x_5$, with $\zeta^3=-1$,
 and $P_2$ to be defined by $x_0=\zeta' x_1,\,x_2=\zeta' x_3, x_4=\zeta' x_5$, with $(\zeta')^3=-1$, and $\zeta'\not=\zeta$.
 The computations in the Jacobian ring $R_f$ are easy to perform. In this ring, $x_i^2=0$, hence in every degree
 $\leq 6$,  we get as free generators the monomials $\prod_{i\in I}x_i$ with no repeated indices.
 The primitive class $\tilde{p}_1\in R^3_f$ defined as the projection of $p_1$ is annihilated
 by multiplication by  $x_0-\zeta x_1,\,x_2-\zeta x_3,\, x_4-\zeta x_5$ because
 these are hyperplane sections vanishing on $P_1$, and similarly
 the primitive class $\tilde{p}_2$ of $p_2$ is  annihilated
 by multiplication by  $x_0-\zeta' x_1,\,x_2-\zeta' x_3,\, x_4-\zeta' x_5$.
 It follows  that $S^1\cdot \tilde{p}_1\subset R^4_f$ is  orthogonal with respect to Macaulay duality
 (see \cite[II, 6.2.2]{voisinbook}) to
 the subspace $S^1\cdot \langle x_0-\zeta x_1,\,x_2-\zeta x_3,\, x_4-\zeta x_5\rangle\subset R^2_f$
 and similarly for $p_2$. But  then the two spaces
 $S^1\cdot \tilde{p}_1$ and $S^1\cdot \tilde{p}_2$ have trivial intersection, as otherwise the spaces
 $S^1\cdot \langle x_0-\zeta x_1,\,x_2-\zeta x_3,\, x_4-\zeta x_5\rangle$ and
 $S^1\cdot \langle x_0-\zeta' x_1,\,x_2-\zeta' x_3,\, x_4-\zeta' x_5\rangle$ would not generate
 $R^2_f$. Thus if $y\in S^1$ satisfies
 $y(\tilde{p}_1+\tilde{p}_2)=0$ in $R^4_f$, one has
 $$y\tilde{p}_1=0,\,y\tilde{p}_2=0\,\,{\rm in}
 \,\,R^4_f.$$
 This  easily implies that $y=0$.
\end{proof}
We have the following applications.
\begin{lemma}\label{lepftrans}  Let $X$ be a general Pfaffian cubic fourfold. Then for any
hyperplane section $Y$ of $X$, the natural map
$H^0(Y,\mathcal{O}_Y(1))\rightarrow \oplus_{p\in{\rm Sing}\,Y} T^1_{{Y_p}}$
is surjective.
\end{lemma}

\begin{proof} Let $\mathcal{P}^{00}$ be the Zariski open subset of  the Pfaffian hypersurface
which is defined as the intersection of $\mathcal{P}$ with the Zariski open set
$\Hyp_{4,3}^{00}$. Note that $\mathcal{P}^{00}$ is non-empty by Lemma \ref{legentrans}(ii), using the fact
that   the set of $[Y]\in \mathbb{P}(H^0(\mathcal{O}_{\mathbb{P}^4}(3)))$ admitting a nontrivial vector field has codimension
$\geq 7$ (see Proposition \ref{proallow}, (4)). We now consider the natural map
$r_{pf}:\mathcal{P}^{00}\rightarrow \mathcal{H}yp_{3,3}$ defined as the restriction
to $\mathcal{P}^{00}\subset \Hyp_{4,3}^{00}$ of $r:\Hyp_{4,3}^{00}\rightarrow \Hyp_{3,3}$.
The fiber of $r_{pf}$ over $[Y]\in \Hyp_{3,3}$ consists in those Pfaffian cubic fourfolds which intersect
$\mathbb{P}^4$ along $Y$.
Let $X$ be a general Pfaffian cubic fourfold. Then
$r_{pf}$ is smooth at any $[X']$ parametrizing a cubic isomorphic to $X$. Indeed, the map
$r$ is smooth, and $r_{pf}$ is the restriction of $r$ to $\mathcal{P}^{00}$. Thus, if
$r_{pf}$ was not smooth at a point $[X']$
with $[Y]=r_{pf}([X'])$, then the fiber of $r$ would be  tangent to $\mathcal{P}^{00}$ at $[X']$ which exactly means
that
$f_Y H^0(X',\mathcal{O}_{X'}(2))\subset H^0(X',\mathcal{O}_{X'}(3))$ is  contained in the tangent space to the Pfaffian hypersurface at the point $[X']$. As $X$ is general, Lemma \ref{lepftrans} tells us that this does not happen at any
$X'$ isomorphic to $X$.
The end of the proof is now identical to  the proof of Lemma \ref{letransXgen}. Indeed, by Lemma
\ref{lepfgood} below and Proposition \ref{proallow},  property (1) of Proposition
\ref{proallow} is  satisfied by any hyperplane section of a
general Pfaffian
cubic, replacing $r^{-1}(\Hyp_{3,3,z})$ by its Pfaffian analogue $r_{pf}^{-1}(\Hyp_{3,3,z})$ which we know to be smooth.
\end{proof}
Recall from Definition
 \ref{defiverygood} that a very good line in a cubic threefold containing no plane  is a line
 which is good and such that the curve $\widetilde C_l$ of lines in $Y$ meeting $l$ is irreducible.
\begin{lemma}\label{lepfgood}  Let $X$ be a general Pfaffian cubic fourfold. Then  any
hyperplane section $Y$ of $X$ is allowable and admits  a very good line.
\end{lemma}
\begin{proof} We know by Propositions
\ref{propgoodline} and \ref{propirred} and Corollary \ref{corallow} that if $X$ is a general cubic fourfold, then any hyperplane section of $X$ is allowable and contains  a very good line. Let us say that $Y$ is bad if it does not admit a very good line
 or has non-allowable singularities
 and $X$ is bad if it  has a hyperplane
section $Y$ which is bad.
The locus of bad cubic fourfolds is a proper closed algebraic subset $\Hyp_{4,3}^{00,bad}$
of $\Hyp_{4,3}^{00}$ and its irreducible components are constructed as follows: for each
irreducible component $Z\subset  \Hyp_{3,3}$ of the   locus of bad cubic threefolds, $r^{-1}(Z)\subset \Hyp_{4,3}$ is the locus of bad cubic fourfolds such that the cubic threefold
$X\cap \mathbb{P}^4$ is parametrized  by a point of  $Z$. Thus
$\PGL(6)\cdot r^{-1}(Z)$ is the set of cubic fourfolds $X$ such that some hyperplane section of
$X$ is isomorphic to a cubic threefold parametrized  by a point of $Z$. It is thus clear that we get any irreducible component of $\Hyp_{4,3}^{00,bad}$
as $\PGL(6)\cdot r^{-1}(Z)$, with $Z$ as above.
We conclude from this that any irreducible component $Z'$ of $\Hyp_{4,3}^{00,bad}$ has the property
that, for any $[X]\in Z'$, there exists a hyperplane section $Y\subset X$ such that
all cubic fourfolds containing $Y$ as  a hyperplane section are parametrized  by points of   $Z'$.
In particular, if $f_Y\in H^0(X,\mathcal{O}_X(1))$ is the equation of $Y$ in $X$,
$f_YH^0(X,\mathcal{O}_X(2))$ must be contained in the Zariski tangent space of $Z'$.
If all Pfaffian cubic fourfolds were bad, then the Pfaffian hypersurface
$\mathcal{P}^{00}$ would be an irreducible component $Z'$
of $\Hyp_{4,3}^{00,bad}$ and we would get a contradiction
with Lemma \ref{letechnique18jan}.
\end{proof}

 \section{Relative compactified Prym varieties \label{sectprym}}
As previously mentioned, our main tool for compactifying the intermediate Jacobian fibration $\cJ_U\rightarrow U$ is the Prym construction that identifies the intermediate Jacobian $J(Y)$ with a Prym variety $\Prym(\widetilde C/C)$ (where the pair $(\widetilde C, C)$ is obtained from $Y$ via the projection from a general line). The Prym construction works well in a relative setting over the smooth locus $U\subset (\bP^5)^\vee$ (and more precisely over the open set
 $\calF^0/U$ of very good lines in the fibers), reducing (at least locally) the problem of understanding degenerations of intermediate Jacobians to that of understanding degenerations of Prym varieties. This is of course a well studied problem: Beauville \cite{bprym} gave a compactification of the moduli of pairs $(\widetilde C, C)$, and many people studied degenerations of Pryms as abelian varieties (e.g. \cite{FS}, \cite{ABH}, \cite{casaprym}). Here we need to understand a specific compactification, and its local structure, over a given base $B$. A few instances of this have already been studied in  \cite{mark_tik}, \cite{ASF}, and \cite{thesisg}. Below, we define and prove a number of results for the relative compactified Prym of families of \'etale double covers of irreducible curves, which we then apply to our context. Specifically, the results of Sections \ref{secgoodline} and \ref{sectrans18janvier} say  that for a general cubic fourfold $X$, we can replace (locally on $B$) the family $\cY/B$ of hyperplane sections by a family of double covers $(\wtC,\mc C)$ such that each fiber $(\widetilde C_t, C_t)$ is an \'etale double cover with both curves irreducible (see  Corollary \ref{coroverygood}). Furthermore, the singularities of $\widetilde C_t$ and $C_t$ are planar, and we can assume (see Corollary \ref{corotransdef}) that the family $\mc C$ gives a simultaneous versal deformation of the singularities of any fiber $C_t$.

\begin{notation}
From now on in this section, $B$ will stand for an arbitrary base, not necessarily $(\bP^5)^\vee$ as elsewhere in the paper.
\end{notation}

We proceed as follows: as in \cite{ASF} we can define, for any family $\wtC_B \to \mc C_B$ of \'etale double covers of irreducible locally planar curves, parametrized by a base $B$, a relative compactified Prym variety  $\Prym{(\wtC_B \slash \mc C_B)} \to B$ whose fibers over the locus parametrizing smooth curves are usual Prym varieties. The relative Prym variety is defined as (one component) of the fixed locus of an involution on the relative compactified Jacobian $\J{\wtC_B}$ of the family $\wtC_B \to B$. From this definition, it follows immediately that if $\J{\wtC_B}$ is smooth then so is $\Prym{(\wtC_B \slash \mc C_B)} $. Unfortunately,  in general  $\J{\wtC_B}$ is not smooth. However, one can sometimes think of $\wtC_B \to B$ as the restriction of a larger family $\wtC_{\wt B} \to {\wt B}$, $B \subset \wt B$, with the property that $a)$ there exist two compatible involutions on $\wtC_{\wt B}$ and on $\wt B$, such that the first is an extension of the given involution on $\wtC_{B}$ and second has the property that the fixed locus on $\wt B$ is equal to $B$; $b)$ the relative compactified Jacobian $\J{\wtC_{\wt B}}$ is smooth.  Under these assumptions the relative Prym variety $\Prym{(\wtC_B \slash \mc C_B)} $ is smooth. An instance of this already appeared in \cite{ASF}. As discussed below, the versality statements valid in our setup allow us to conclude that the relative compactified Prym is indeed smooth in our situation.

The results in this section build on an important result for compactified relative Jacobians, namely the Fantechi--G\"ottsche--van Straten \cite{FGvS} smoothness criterion.

\subsection{Relative Compactified Pryms (the \'etale case)}
Let $f: \wt C \to C$ be an \'etale double cover of smooth projective curves, and let $\iota: \wt C \to  \wt C$ be the corresponding involution on $\wt C$. We denote by $g$ be the genus of $C$, and by $h$ the genus of $\wt C$, so that $h=2g-1$. Recall that the Prym variety of $\wt C$ over $C$, which we will denote by $\operatorname{Prym}(\wt C \slash C)$, is the identity component of the fixed locus of the involution
\[
\tau:=-\iota^*: \Pic^0(\wt C ) \to \Pic^0(\wt C).
\]
The Prym variety $\operatorname{Prym}(\wt C \slash C)$ is a principally polarized abelian variety \cite{Mumford} of dimension $g-1$.
Equivalently \cite{Mumford}, the Prym variety can be defined as the identity component of the Norm map
\[
\Nm: \Pic^0(\mc C) \to \Pic^0(C), \,\quad \mc O_{\wt C}(\sum p_i) \mapsto \mc O_{ C}(\sum f(p_i)),
\]
or as the image of
\[
(1-\iota^*): \Pic^0(\wt C ) \to \Pic^0(\wt C).
\]
Now suppose that $f: \wt C \to C$ is an \'etale double cover of singular, but irreducible curves, and let $\wt n:  \wt D \to \wt C$ and $n: D \to C$ be the normalizations of the two curves.
The involution $\iota$ on $\wt C$ lifts to a compatible involution
\be \label{involution normalization}
\varepsilon: \wt D \to \wt D
\ee
so that the natural morphism $\wt D \to D$ is an \'etale double cover with associated involution $\varepsilon$. Let $\{x_1, \dots, x_k\}$ be the singular points of $C$ and let
\[
\{p_1, \dots, p_k, q_1, \dots, q_k\}
\]
be the singular points of $\wt C$, with $f^{-1}(x_i)=\{p_i, q_i\}$. The identity component of the Picard group, or generalized Jacobian, of $\wt C$ fits into the natural short exact sequence of groups
\be \label{definition of A}
 1 \to A \times A \longrightarrow \Pic^0(\wt C) \stackrel{\wt n^*}{\longrightarrow} \Pic^0(\wt D) \to 1,
\ee
where $A:=H^0( \wt C, \oplus_{i=1}^k ( n_* \mc O_{\wt D}^\times \slash \mc O_{\wt C}^\times )_{p_i})\cong H^0( \wt C, \oplus_{i=1}^k ( n_* \mc O_{\wt D}^\times \slash \mc O_{\wt C}^\times )_{q_i})$ is a commutative affine group. The involution $-\iota^*$ still acts on $\Pic^0(\wt C)$ and we can define, in analogy with the generalized Jacobian, the generalized Prym variety of $\wt C$ over $C$ to be the identity component of the fixed locus of $-\iota^*$
\[
\operatorname{Prym}(\wt C \slash C):=\Fix( -\iota^* )_{\id}.
\]
The involution $-\iota^*$ is compatible, via $n^*$, with $-\varepsilon^*$ on $\Pic^0(\wt D)$ and it acts on $A \times A$ via $(a, b) \mapsto (b^{-1}, a^{-1})$ (note that we use multiplicative notation for these groups, even though $A$ can be a product of both additive and multiplicative groups). We therefore get a short exact sequence
\[
1 \to A  \longrightarrow \operatorname{Prym}(\wt C \slash C) \stackrel{\wt n^*}{\longrightarrow} \operatorname{Prym}(\wt D \slash D) \to 1,
\]
where the inclusion $A \hookrightarrow A \times A=\ker \wt n^*$ is given by $a \mapsto (a, a^{-1})$. Set $\delta:=\dim A$, so that $g(D)=g-\delta$, and $g(\wt D)=h-2\delta$. Since $\dim \operatorname{Prym}(\wt D \slash D) = g -\delta -1$, we see that
\[
\dim \operatorname{Prym}(\wt C \slash C)=g-1,
\]
as in the smooth case. For example, if $C$ is a nodal irreducible curve, then $\operatorname{Prym}(\wt C \slash C)$ is a semi--abelian variety.

The strategy to compactify the generalized Prym variety is to extend the involution to the compactified Jacobian of $\wt C$. By definition, the degree $d$ compactified Jacobian of an integral projective curve $\Gamma$ is the moduli space of rank $1$ torsion free sheaves on $\Gamma$ of degree $d$. If $\Gamma$ has planar singularities, i.e. if locally around every singular point, $\Gamma$ is isomorphic to a plane curve, then every component is irreducible  \cite{Rego}  of dimension equal to the arithmetic genus of $\Gamma$. The degree zero component $\ov \Jac(\Gamma)$ contains $\Pic^0(\Gamma)$ as an open dense subset. For the proof of Propositions \ref{base change} and \ref{smoothness properties} we will need further properties of the compactified Jacobian, and in particular, we will need a description of the complement $\ov \Jac(\Gamma) \setminus \Pic^0(\Gamma)$ (cf. \cite{Rego} and \cite{Cook};  see Proposition \ref{cook rego} below), and a smoothness criterion due to Fantechi--G\"ottsche--van Straten \cite{FGvS} (recalled below in Proposition \ref{smoothness fgvs}). If $\Gamma$ has locally planar singularities, then any torsion free coherent sheaf $F$ on $\Gamma$ is reflexive,  that is
\[
(F^\vee){}^\vee =F, \quad \text{where} \quad F^ \vee:=\Shom_{\mc O _\Gamma}(F, \mc O_\Gamma).
\]
Moreover, if
\[
j: \Gamma \subset Z,
\]
is an embedding of $\Gamma$ in a smooth projective variety $Z$ of dimension $d$, then using  \cite[Prop. 1.1.10]{Huy-Lehn} we can see that
\be \label{which ext vanish}
\Shext^{c}_Z(F, \mc \omega_Z)=0 \quad \text{ for all} \,\, c \neq d-1
\ee

The following lemma is well--known, and is crucial to define the involution as a regular morphism on the
family of compactified Jacobians.
\begin{lemma} \label{ext dual} The only non--zero $\mathcal{E}xt$ sheaf satisfies
\[
\Shext^{d-1}_Z(F, \mc O_Z) =F^\vee\otimes \det N_{\Gamma|Z}
\]
 (notice that since $\Gamma$ has locally planar singularities, the embedding in $Z$ is l.c.i.).
\end{lemma}
\begin{proof}
This is simply \cite[III, Lemma 7.4 and Theorem 7.11]{hartshorne}. The only thing to remark is that the proof of Lemma 7.4 of loc. cit. goes through unchanged for the ext--sheaves.
\end{proof}

\begin{lemma} \label{involution}
Assume that $C$, and therefore $\wt C$, has planar singularities. Then the assignment
\be
\begin{aligned}
\tau: \ov\Jac(\wt C) & \longrightarrow  \ov\Jac(\wt C)\\
F & \longmapsto \iota^* F^ \vee
\end{aligned}
\ee
defines a regular involution which extends $-\iota^* $ over the complement of $ \Pic^0(\wt C)$.
\end{lemma}
\begin{proof}
By \cite[Thm 3.4]{AKII}, there is a universal sheaf $\mc F$ on $\ov\Jac(\wt C)\times \wt C$. Let $\wt C \subset Z$ be an embedding of $\wt C$ in a smooth projective variety $Z$ of dimension $d \ge 2$. We may view $\mc F$ as a sheaf on $\ov\Jac(\wt C)\times Z$, i.e., as a family of pure codimension $d-1$ sheaves on $Z$ parametrized by $\ov \Jac(\wt C)$. Let $p_Z:\Jac(\wt C)\times Z \to Z$ be the second projection.
We claim that the sheaf $\Shext^{d-1}_{\ov\Jac(\wt C)\times Z}( \mc F, p_Z^* \mc O_Z)$ satisfies base change, i.e. that
\[
\Shext^{d-1}_{\ov\Jac(\wt C)\times Z}( \mc F,p_Z^* \mc O_Z)_{|\{t\} \times  Z}=\Shext^{d-1}( \mc F_t, \mc O_Z) =\mc F_t ^ \vee \otimes \det N_{C|Z}.
\]
Indeed, by \cite[Thm 1.10]{AK} it is enough to verify that $\Shext^{c}(\mc F_t, \mc O_Z)=0$, for $c=d-2$ and $c=d$, which follows directly from  (\ref{which ext vanish}).


Using Lemma \ref{ext dual} we see that the sheaf $\Shext^{d-1}_{\ov\Jac(\wt C)\times Z}( \iota^* \mc F,p_Z^* \mc O_Z)\otimes p_Z^*   \det N_{C|Z}^ \vee$  provides a family of rank one torsion free sheaves on $C$ and it  determines a morphism $\ov\Jac(\wt C) \to \ov\Jac(\wt C)$  which sends a sheaf $F$ to $\iota^* F^ \vee$. Since the sheaves are reflexive, this morphism is an involution.
 \end{proof}

\begin{definition}
The compactified Prym variety  $\ov{\operatorname{Prym}}(\wt C \slash C)$ of an \'etale double cover $\wt C \to C$ of integral curves with planar singularities is the irreducible component containing the identity of the fixed locus $\Fix(\tau) \subset \ov\Jac(\wt C)$.
\end{definition}

\begin{remark} \label{Pt irreducible}
Notice that $ \operatorname{Prym}(\wt C \slash C) \subset \ov{\operatorname{Prym}}(\wt C \slash C)$ is a dense open subset, so that, in particular, $$\dim \ov{\operatorname{Prym}}(\wt C \slash C)=g-1.$$ For example, if $C$ is  irreducible  with one node, then $\ov{\operatorname{Prym}}(\wt C \slash C)$ is a rank one degeneration of an abelian variety. For an explicit description of Prym varieties of other singular curves, see \cite{thesisg}.
\end{remark}

\begin{remark}
One could also drop the assumption of integrality, and define the relative compactified Prym variety  for \'etale double covers of arbitrary curves with locally planar singularities. In this situation, however, a choice has to be made, namely that of a polarization on the family of double covers and the relative compactified Prym depends on this choice (for the case of family of curves lying on smooth projective surfaces see \cite{ASF} and \cite{thesisg}). Since the curves we will deal with in our situation turn out to be integral by Proposition \ref{propirred}, we will restrict ourselves to the case of integral curves.
\end{remark}

Let now  $B$ be an irreducible base scheme and let
\begin{equation} \label{family of covers}
\xymatrix{
\wt \calC_B \ar[rr]^f_{2:1} \ar[dr] & & \calC_B  \ar[dl] \\
& B &
}
\end{equation}
be a family, parametrized by $B$, of \'etale double covers of reduced and irreducible curves with planar singularities. Let us denote by $g$ the genus of the curves in the family $\mc C_B \to B$, and by $h=2g-1$ the genus of their double covers $\wt \calC_B \to B$. We let
\[
\iota: \wt \calC_B  \to \wt \calC_B,
\]
be the involution associated to the covering, and define a $2$--torsion line bundle $\eta_B$ on $ \calC_B$ by setting
\[
f_* \mc O_{\wt \calC_B}=\mc O_{ \calC_B} \oplus \eta_B.
\]
For any $b \in B$, we denote by
\[
f_b: \wt \calC_b  \to  \calC_b, \quad \iota_b: \wt \calC_b  \to \wt \calC_b, \quad \eta_b \in \Pic^0(\mc C_b)
\]
the respective restrictions of $f$, $\iota$ and  $\eta$ to the fiber over $b$.
Consider the degree $0$ relative compactified Jacobians
\[
\xi:  \J{\wt \calC_B}  \to B, \quad \quad \pi: \J{\calC_B} \to B.
\]
The fiber of $\xi$ (respectively $\pi$) over a point $b \in B$ is the compactified Jacobian of the curve $\wt \calC_b$ (respectively $\calC_b$) and is a reduced and irreducible l.c.i. of dimension $h$ (respectively $g$). The smooth locus of $\xi$ (respectively $\pi$)  is the relative generalized Jacobian $\Pic^0_{\wt \calC_B}$ (resp. $\Pic^0_{\mc C_B}$) which parametrizes line bundles. Since they are group schemes, both admit a zero section, and the morphisms $-\iota^*$ and $1-\iota^*$ are well defined group homomorphisms of $\Pic^0_{\wt \calC_B}$. Also, $\Pic^0_{\mc C_B}$ has a section determined by $\eta_B$.

\begin{lemma}\label{lemmataumor}
The ``$-1$'' morphism on the group scheme $\Pic^0_{\wt \calC_B}$ extends to a regular morphism on the whole $ \J{\wt \calC_B}$. In particular, there is a regular involution
\[
\begin{aligned}
\tau: \ov\Jac(\wt \calC_B) & \longrightarrow  \ov\Jac(\wt \calC_B)\\
F & \longmapsto \iota^* F^ \vee
\end{aligned}
\]
\end{lemma}
\begin{proof} Again, this follows from \cite[Thm 3.4]{AKII}. Locally on $B$ we can find a  section of the smooth locus of $\wt \calC_B \to B$, hence by {\it loc. cit.} there is a universal sheaf on $ \ov\Jac(\wt \calC_B) \times \wt \calC_B$. Since the morphism $\wt \calC_B \to B$ is projective, we can find (up to restricting $B$) a smooth projective $Z$ of dimension $d$ such that $\wt \calC_B \to B$ can be factored by an embedding $\wt \calC_B \subset  Z\times B$ followed by the second projection to $B$. We can hence argue as in Lemma \ref{involution} and define an involution by considering the family of torsion free sheaves
\[
\Shext^{d-1}_{Z \times B}( \iota^* \mc F,p_Z^* \mc O_Z)\otimes \det N_{\wt \calC_B |Z\times B}^ \vee
\]
Since the assignment is independent of the choice of a section and of $Z$, these local morphisms glue together to a global involution on  $\ov\Jac(\wt \calC_B)$.
\end{proof}


We can now define the relative compactified Prym.

\begin{definition} \label{definition prym}
Let $f: \wt \calC_B \to \calC_B$ be as in (\ref{family of covers}) an \'etale double cover between two families of reduced and irreducible curves with planar singularities.
The relative compactified Prym variety of $\wt \calC_B $ over $ \calC_B$, denoted $\ovPrym{\wt\calC_B \slash \calC_B}$ is defined to be the irreducible component of the fixed locus $\Fix(\tau) \subset \ov\Jac(\wt C_B)$ that contains the zero section:
\[
\ovPrym{\wt\calC_B \slash \calC_B}:=\Fix(\tau)_0 \subset \ov\Jac(\wt C_B)
\]
\end{definition}

\begin{lemma}
The fixed locus $\Fix(-\iota^*) \subset \Pic^0_{\wt \calC_B}$ has four isomorphic connected components.
\end{lemma}
\begin{proof}
In \cite{Mumford}, it is proved that for any $b \in B$ such that $\calC_b$ is smooth, $\Nm^{-1}(\mc O_{ \calC_b}) \subset \Pic^0_{\wt \calC_b}$ has two connected components. Hence, so has $\Nm^{-1}(\eta_b)$. Moreover,
\[
f^*(\Nm(\mc O_{\wt C_b}(\sum c_i)))=\mc O_{\wt C_b}(\sum c_i+\iota c_i),
\]
and $\ker(f^*)=\langle \mc O_{\wt C_b}, \eta_b\rangle$, so that  $\Fix(-\iota_b^*)=\Nm^{-1}(\eta_b) \cup \Nm^{-1}(\mc O_{ \calC_b})$ and hence the fixed locus of $-\iota^*_b$ in $\Pic^0(\wt \calC_b)$ has four connected components. Over $B$, since $f$ is \'etale and  $f_* \mc O_{\wt \calC_B}=\mc O_{\calC_B} \oplus \eta_B$ is locally free, we can consider the relative norm map \cite[\S 6.5]{EGAII} $\Nm_B: \Pic^0_{\wt \calC_B} \to \Pic^0_{\calC_B}$ and hence we may consider the inverse images of the zero section of $\Pic^0_{\calC_B} \to B$ and of the section determined by $\eta_B$. By  \cite[Prop. 6.5.8]{EGAII} on each fiber the norm map is compatible with the norm map associated to the double cover $\wt D_b \to D_b$ between the normalizations of the two curves (cf. (\ref{involution normalization})). Hence, it restricts to a norm map $m: A \times A \to A$ (notation as in (\ref{definition of A})), which is nothing but the multiplication map $(a,b) \to ab$. Hence, on every fiber the kernel of the norm map is an extension of the kernel of the norm map of an \'etale double cover of smooth curves (namely, the normalizations) by $\ker m \cong A$. It follows that the inverse image under $\Nm_B$  of the zero section has two connected components: the one containing the zero section of $\Pic^0_{\wt \calC_B}$ and the remaining one. Hence, also the inverse image of the section determined by $\eta_B$ has two connected components and the lemma is proved.
\end{proof}

Let
\be \label{components}
\mc P_1, \dots, \mc P_4
\ee
be the four connected components of $\Fix(-\iota^*) \subset \Pic^0_{\wt \calC_B}$, with $\mc P_1$ the component
\[
\Prym{(\wt \calC_B \slash \calC_B)}:=\Fix(-\iota^*)_0= \im(1-\iota^*) \subset \Pic^0_{\wt \calC_B}
\]
containing the zero section. We will call this component, which is a group scheme of dimension $g-1$ over $B$,  the relative generalized Prym variety of $\wt \calC_B$ over $\calC_B$. It is dense in the relative compactified Prym, so that its closure $\ov {\mc P_1}$ satisfies
\[
\ov {\mc P_1}=\ovPrym{\wt \calC_B \slash \calC_B}= \ov{\im(1-\iota^*)}.
\]

By restricting $\pi$ to the relative Prym we get a morphism
\[
\nu_B: \ovPrym{\wt \calC_B \slash \calC_B} \to B,
\]
whose fiber over a point $b\in B$ corresponding to a double cover between smooth curves $\wt \calC_b \to \calC_b$ is isomorphic to the usual Prym variety $\Prym(\wt \calC_b \slash \calC_b)$. Notice that
\[
\dim  \ovPrym{\wt \calC_B \slash \calC_B}= \dim B + g-1.
\]

We now need to show that this definition is well posed, in the sense that it is compatible with base change. For example, we would like to verify that the fiber over an arbitrary point $b \in B$ of the relative compactified Prym variety is the compactified Prym variety of $\wt \calC_b$ over $\mc C_b$.

\begin{prop} \begin{enumerate} \label{base change}
\item For every $\, b\in B$ we have $\ovPrym{\wt \calC_B \slash \calC_B}_{|b} =\ovPrym{\wt \calC_{b} \slash \calC_{b}}.
$
\item The base change property holds for arbitrary base change, i.e., for any morphism $B' \to B$, setting $\wt \calC_{B'}=\wt \calC_B \times_B B'$ and $\calC_{B'}=\calC_B \times_B B'$, we have
\[
\ovPrym{\wt \calC_B \slash \calC_B}\times_B B' =\ovPrym{\wt \calC_{B'} \slash \calC_{B'}}.
\]
\item The morphism $\nu_B: \ovPrym{\wt \calC_B \slash \calC_B} \to B$ is equidimensional.
\end{enumerate}
\end{prop}

We claim that it is sufficient to prove $(1)$, since it implies the other two points.  Indeed, to show that $(1) \Rightarrow (2)$ we may assume without loss of generality that $B'$ is irreducible. Hence $\ovPrym{\wt \calC_{B'} \slash \calC_{B'}}$ is  irreducible. Since it is contained in $\ovPrym{\wt \calC_B \slash \calC_B}\times_B B'$, we only need to show that the later is also irreducible of the same dimension. By $(1)$ and Remark \ref{Pt irreducible}, the fiber $\ovPrym{\wt \calC_B \slash \calC_B}_{|b}$ over any $b \in B$  is irreducible of dimension $g-1$. Hence, the morphism $\ovPrym{\wt \calC_B \slash \calC_B}\times_B B' \to B'$ is equidimensional and $ \ovPrym{\wt \calC_B \slash \calC_B}\times_B B' $ is irreducible of dimension $\dim B'+g-1$. The claim is proved.\\

To prove $(1)$ we first need a few lemmas, as well as the first statement of Proposition \ref{smoothness properties} below.
Before stating and proving these lemmas, we need to recall a few properties of compactified Jacobians (cf. \cite{Rego}, \cite{Cook}).

Let $\Gamma$ be an integral projective curve, with normalization $n: \ov \Gamma \to \Gamma$. Given  a rank $1$ torsion free sheaf $F$ on $\Gamma$, there exists \cite{Cook} a partial normalization
\[
n': \Gamma' \to \Gamma,
\]
with the property that $\Shend(F)\cong n'_* \mc O_{\Gamma'}$ and that the rank $1$ torsion free sheaf $F':={n'}^*F \slash Tors$ on $\Gamma$, satisfies
\[
F=n'_* F'.
\]
For later use, we highlight that, given $F$, we can define the curve $\Gamma'$ by setting
\[
\Gamma':=\Spec_{\mc O_{\Gamma}}\Shend(F)
\]

We define a \emph{local type} \cite{Cook} of rank one torsion free sheaf to be a collection $\{M_p\}_{p \in \Sing (\Gamma)}$ of isomorphism classes of  rank one torsion free $\mc O_{\Gamma,p}$--modules, where $p$ runs in the set of singular point of $ \Gamma$.

Two rank one torsion free sheaves $F$ and $G$ on $\Gamma$ are said to be of the same local type if for any $p \in \Gamma$ the localizations $F_p$ and $G_p$ are isomorphism $\mc O_{\Gamma,p}$--modules.

\begin{prop}[{\cite{Cook}, \cite{Rego}}]\label{cook rego}
The relative compactified Jacobian $\ov \Jac(\Gamma)$ is stratified based on the local type: for every local type $\{M_p\}_{p \in \Sing (\Gamma)}$ there exists a rank one torsion free sheaf $F$ of degree zero with $F_p \cong M_p$ for every $p$. Furthermore, $\Pic(\Gamma)$ acts transitively on the set of rank one torsion free sheaves of a fixed type $\{M_p\}$, with stabilizer $\ker[\Pic^0(\Gamma) \to \Pic^0(\Gamma')]$, with $\Gamma':=\Spec_{\mc O_{\Gamma}}\Shend(F)$ as above. If $\Gamma$ has planar singularities, then  $\ov \Jac(\Gamma)$ contains the generalized Jacobian as a dense open subset.
\end{prop}

Let us now return to our situation. Our aim is to understand the fixed locus of $\tau$ in  $\ov \Jac(\wt \calC_B) $.
We start by viewing the group homomorphism $(1-\iota^*):  \Pic^0_{\wt \calC_B} \to \Pic^0_{\wt \calC_B} $, $ L \mapsto L \otimes \iota^*L^\vee$ as a rational map
\be \label{rational iota}
(1-\iota^*): \ov \Jac(\wt \calC_B) \dashrightarrow \ov \Jac(\wt \calC_B).
\ee
Let us focus on a neighborhood of the fiber $ \ov \Jac(\wt \calC_{b_0})$ over a point $b_0 \in B$.  Let $\{x_1, \dots, x_k\}$ be the singular points of $\mc C_{b_0}$ and let $\{p_1, \dots, p_k, q_1, \dots, q_k\}$ be the singular points of $\wt \calC_{b_0}$, with $f^{-1}(x_i)=\{p_i, q_i\}$.
For every subset $I\subset \{1, \dots, k\}$ , we can consider the open subset
\[
V_I \subset  \ov \Jac(\wt \calC_B)
\]
of sheaves that are locally free in a neighborhood of $\{p_i\}_{i \in I}$ and of $\{q_j\}_{j \notin I}$, so that
\[
V=\cup_{I } V_I,
\]
is the open set of sheaves that for every $i$ are locally free at least at one of the two points $p_i$ or $q_i$.
\begin{lemma}
The rational map (\ref{rational iota}) is defined in an open neighborhood of  $\ov \Jac(\wt \calC_B)$ containing $V$.
\end{lemma}
\begin{proof}
It is enough to show that if $F \in V_I$ then $F \otimes \iota^* F^\vee $ is torsion free. We only need to check this condition at the singular points $\{p_1, \dots, p_k, q_1, \dots, q_k\}$ and, by symmetry, it is enough to check at $p_i$, for every $i$. We have $(F \otimes \iota^* F^\vee)_{p_i}=F_{p_i} \otimes F_{q_i}^\vee$. Since by construction at least one between $F_{p_i} $ and $F_{q_i}^\vee $ is locally free, while the other is torsion free,  their tensor product is torsion free.
\end{proof}

\begin{lemma} \label{G}
Given  $F \in \Fix(\tau) \subset \ov \Jac(\wt \calC_{B})$, supported on $\wt \calC_{b_0}$, there exists a $G' \in V \subset   \ov \Jac(\wt \calC_B)$ such that the rank $1$,  $\tau$--invariant, torsion free sheaf
\[
G:=(1-\iota^*) G' \in   \ovPrym{\wt \calC_B \slash \calC_B}
\]
is of the same local type as $F$.
\end{lemma}
\begin{proof} Set $\Gamma:=\wt \calC_{b_0}$.
Since $\tau (F)=F$, we have
\[
F_{p_i}=(\iota^*F^\vee)_{p_i}=F^\vee_{q_i}
\]
so that the local type of $F$ is determined by the localizations $F_{p_1}, \dots,F_{p_k}$ at only half of the singular points. Consider the local type $\{F_{p_1}, \dots, F_{p_k}, \mc O_{q_1}, \dots, \mc O_{q_k}\}$, where $\mc O_{q_i}:= \mc O_{\Gamma, q_i}$, and let $G'$ be a sheaf in $ \ov \Jac(\wt \calC_{b_0})$  with this local type, which exists  by Proposition \ref{cook rego}. With this notation it is clear that
\[
G:=G' \otimes \iota^* {G'}^\vee
\]
is of the same local type as $F$. Since $ G \in \ov{\im(1-\iota^*)}=\ovPrym{\wt \calC_B \slash \calC_B}$, the lemma is proved.
\end{proof}

Set
\[
\Gamma':=\Spec_{\mc O_{\Gamma}}\Shend(F),
\]
and let $n': \Gamma' \to \Gamma$ be the natural partial normalization morphism.
Since $G$ and $F$ are of the same local type, by Proposition \ref{cook rego} we know that there exists an $L \in \Pic^0_{\wt \calC_{t_0}}$, well defined up to an element of  $\ker[{n'}^*: \Pic^0(\Gamma) \to \Pic^0(\Gamma')]$, such that
\[
G=F \otimes L.
\]
\begin{lemma} \label{L tau invariant}
Up to changing $L$ by an element of $\ker[\Pic^0(\Gamma) \to \Pic^0(\Gamma')]$, we can assume that $\tau(L)=L$.
\end{lemma}
\begin{proof}
Since $\tau(F)=F$ and $\tau(G)=G$, we have that $M:=L \otimes \tau(L)^\vee$ lies in $\ker[ \Pic^0(\Gamma) \to \Pic^0(\Gamma')]$ and satisfies $M^\vee \cong \tau(M)$. This last equality implies that $M=\iota^*(M)$, and it is not hard to see (cf. for example \cite[Lem. 2.8]{Sacca}) that this implies the existence of an $M' \in \ker[ \Pic^0(\Gamma) \to \Pic^0(\Gamma')]$ such that $M=M' \otimes \iota^* M'$. From
\[
L=\iota^* L^ \vee \otimes M=\iota^* L^ \vee \otimes M' \otimes \iota^* M',
\]
we deduce that
\[
L':= L \otimes {M'}^ \vee
\]
satisfies $\iota^* {L'}^ \vee=L'$ and since $M' \in \ker[ \Pic^0(\Gamma) \to \Pic^0(\Gamma')]$ we still have
\[
G=F \otimes L'.
\]
\end{proof}

Let us now go back to the components $\mc P_1, \dots, \mc P_4$, defined in (\ref{components}).
The morphism
$$ (L, M)  \longmapsto L \otimes M,\,\,
\mc P_1 \otimes \mc P_i  \longrightarrow \mc P_i
$$

extends to a morphism
\be \label{tensor Pi}(G, M)  \longmapsto G \otimes M,\,\,
\ov{\mc P}_1 \otimes \mc P_i   \longrightarrow \ov{\mc P}_i
,
\ee
where for every $i$, $\ov{\mc P}_i$ denotes the closure of ${\mc P}_i$ and hence is an irreducible component of $\Fix(\tau)$. Recall that we defined $\mc P_1$ so that $\ov{\mc P}_1=\ovPrym{\wt \calC_B \slash \calC_B}$.

The last element we need is the following Lemma
\begin{lemma} \label{closures of Pi}
For $i \neq j$ the closures $\ov{\mc P}_i $ and $\ov{\mc P}_j$ do not intersect.
\end{lemma}
\begin{proof}
This is Corollary \ref{disjoint components} below, which is based only on the smoothness of the relative compactified Jacobian over the versal family of an integral, locally planar curve.
\end{proof}

\begin{cor} \label{irreducible components}
Every irreducible component of $\Fix(\tau) \subset \ov \Jac(\wt \calC_B)$ is of the form $\ov{\mc P}_i$ for some $i=1, \dots, 4$. In particular, in Definition \ref{definition prym} we can replace ``irreducible'' component with ``connected'' component.
\end{cor}
\begin{proof}
Consider an $F \in \Fix(\tau)$ and let $G$ and $L$ be as in Lemmas \ref{G} and \ref{L tau invariant} respectively. Since $\tau(L)=L$ and the fixed locus of $\tau$ on $\Pic^0_{\wt \calC_B}$ is equal to $ \coprod  \mc P_i$, $L \in \mc P_i$ for some $i$. Since $G \in \ov{\mc P}_1$, it follows by (\ref{tensor Pi}) that  $F \in\ov{\mc P}_i$.
\end{proof}

We finally get to the proof of $(1)$ of Proposition \ref{base change}:

\begin{cor} For any $b_0\in B$, one has $\ovPrym{\wt \calC_B \slash \calC_B}_{|b_0}=\ovPrym{\wt \calC_{b_0} \slash \calC_{b_0}}$.
\end{cor}
\begin{proof} We only need to prove that   $\ovPrym{\wt \calC_B \slash \calC_B}_{|b_0} \subset \ovPrym{\wt \calC_{b_0} \slash \calC_{b_0}}$ since the reverse inclusion is clear.
Consider an $F \in \ovPrym{\wt \calC_B \slash \calC_B}_{|b_0}$. As in Lemmas \ref{G} and \ref{L tau invariant}, we can find a $\tau$--invariant $L$ and a $G\in \ov{\mc P}_1$, with $G=(1-\iota^*) G'$ such that $G=F \otimes L$.  By Lemma \ref{closures of Pi}, and the fact that $F \in \ov  {\mc P}_1$, we necessarily have $L\in \mc P_1$. Since $\mc P_1=(1-\iota^*) \Pic^0_{\wt \calC_B}$, we can find $L'' \in \Pic^0_{\wt \calC_{t_0}}$ such that $L=(1-\iota_{t_0}^*) L''$. By construction, $G=(1-\iota^*) G'$ and hence
\[
F=(1-\iota_{t_0}^*)( F'\otimes {L''}^ \vee),
\]
from which we see that
\[
F \in \ov{\im(1-\iota_{t_0}^*)}=\ovPrym{\wt \calC_{t_0} \slash \calC_{t_0}}.
\]
\end{proof}

\subsection{Smoothness results for the relative compactified Prym}

The next step is to study the local structure of the relative compactified Prym variety. This will allows us to formulate a criterion that has to be satisfied by a family $\wt \calC_B \to \calC_B$ of \'etale double covers of irreducible, locally planar curves in order for the relative compactified Prym to be smooth. Since this criterion will be deduced by an analogue criterion for the smoothness of the relative compactified Jacobian, we start by reviewing rapidly, following closely \cite{FGvS}, the results we need on this topic.

Let $D$ be a reduced projective curve, with planar singularities. We denote by $\Def(D)$  the deformation functor of the curve $D$ and, for any $p \in D$, we let $\Def(D_p)$ be the deformation functor of the local ring $\mc O_{D,p}$. For more precise definitions, see \cite[\S 2.4.1]{sernesi}. Letting $ \Sing (D) \subset D$ denote the singular locus of $D$, set
\[
\Def^{loc}(D):= \prod_{p \in \Sing (D) } \Def(D_p)
\]
and consider the natural transformation of functors
\[
 \Phi: \Def(D) \to \Def^{loc}(D)
\]
which to a deformation of the global curve assigns the induced deformation of local rings at the singular points. Since $D$ is reduced,  $\Def(D)$ and $\Def(D_p)$ are unobstructed \cite[Example 2.4.9]{sernesi} and hence they admit smooth semi--universal deformations spaces, i.e., there exists a smooth affine scheme $S$, a point $ s \in S$, and a transformation of functors
\[
\Psi: (S,s) \to \Def(D),
\]
(here, we denote by $(S,s)$ the deformation functor induced by the germ of the complex space) that is smooth and an isomorphism at the level of tangent spaces, and analogously for the $\Def(D_p)$. We say that the semi--universal space is centered at $s \in S$. The tangent spaces to these deformation functors fit into the local to global exact sequence
\[
0 \to H^1(T_D) \longrightarrow \underbrace{\Ext^1(\Omega^1_D, \mc O_D)}_{T\Def(D)} \longrightarrow \bigoplus_{p \in \Sing(D)}  \underbrace{H^0(\Shext^1_{\mc O_{D,p}}(\Omega^1_{D,p}, \mc O_{D,p}))}_{T \Def(D_p)} \to 0,
\]
where $T_D:=\Shom(\Omega^1_D, \mc O_D)$, and where $H^1(T_D)$ is the tangent space to the subfunctor $\Def(D)'$ of the deformations of $D$ that are locally trivial.
Let $\mc D \to S$ be the semi--universal family for $D$, centered at $s$, and let $\ov \Jac(\mc D) \to S$ be the relative compactified Jacobian. For any sheaf $F \in \ov \Jac(\mc D)$ we can consider the deformation functor of the pair $\Def(F, D)$ and, for any $p \in \Sing(D)$, also of the pair $\Def(F_p, D_p)$. By \cite[Prop. A3]{FGvS}, the $ \Def(F_p, D_p)$ are smooth functors. As above, there are natural transformations $\Psi': (\ov \Jac(\mc D), F)  \to \Def(F,D)$ and $\Phi': \Def(F,D) \to  \prod \Def(F_p, D_p)$.
There is a commutative diagram of functors
\be \label{fundamental diagram}
\xymatrix{
(\ov \Jac(\mc D), F) \ar[r]^{\Phi' \circ \Psi'  \phantom{mmmm}}  \ar[d] &  \prod_{p \in \Sing(D)} \Def(F_p, D_p)  \ar[d] \\
(S, s) \ar[r]_{\Phi \circ \Psi  \phantom{mmmmm}} &   \prod_{p \in \Sing(D)} \Def(D_p)
}
\ee
This diagram is not necessarily Cartesian, but the horizontal maps are smooth maps of functors by what was said above and by \cite[Prop. A1]{FGvS}. In particular, $\ov \Jac(\mc D)$ is smooth along $\ov \Jac(D)=\ov \Jac(\mc D_s)$.
Now let $\mc D_B \to B$ be a family of integral locally planar curves, with $\mc D_0=D$ for some $0 \in B$ and with $B$ smooth. There is a morphism $B \to S$, mapping $0 $ to $ s$, inducing a diagram
\[
\xymatrix{
\ov \Jac(\mc D_B) \ar[r] \ar[d] & \ov \Jac(\mc D) \ar[d] \\
B \ar[r] & S
}
\]
Since the diagram is Cartesian and we are assuming that $B$ is smooth, $\ov \Jac(\mc D_B)$ is smooth at a point $F \in \ov \Jac(\mc D_0)$ if and only if the image of the tangent space $T_0 B$ in $T_sS$ is transversal to the image of $T_F \ov \Jac(\mc D)$ in $T_s S$. Hence, in order to be able to check whether $\ov \Jac(\mc D_B)$ is smooth at a point $F$, we need to understand the image of $T_F \ov \Jac(\mc D)$ in $T_s S$. This is done by analyzing, in the following way, what happens in diagram (\ref{fundamental diagram}) at the level of tangent spaces

Set $P:= \mathbb C [[x,y]]$. Since $D$ has locally planar singularities, for any $p_i \in \Sing(D)$ there exists an $f_i \in P$ such that the completion of the local ring $\mc O_{D,p_i}$ is isomorphic to $R_i=P \slash f_i$. With this notation, we have
\[
T^1_{D_{p_i}}(=T \Def(D, p_i))=P\slash (f_i, \partial_x f_i, \partial_y f_i)
\]
(note that $T^1_{D_{p_i}}$ is a vector space of dimension $\tau(f_i)$, see Remark \ref{deftjurina}). For any $i$, let $\ov R_i \supset R_i$ be the normalization of $R_i$. We denote by $I_i \subset R_i$ the conductor ideal, i.e., $I_i:=\Hom(\ov R_i, R_i)$, and we let
\be \label{Vi}
V(D_{p_i}) \subset T^1_{D_{p_i}}
\ee
be the image in the Jacobian ring $P\slash (f_i, \partial_x f_i, \partial_y f_i)$ of the conductor ideal. It is a codimension $\delta_i$ subspace, with $\delta_i:= \dim \ov R_i \slash R_i=\dim R_i \slash I_i$.  Let
\be \label{V}
V(D) \subset T_s S
\ee
be the inverse image of $\prod V(D_{p_i})$ under the tangent map $T_sS  \to  \prod T^1_{D_{p_i}}$. It is known that $V(D)$ is the support of the tangent cone to the deformations of $D$ that keep the geometric genus constant, and has codimension in $T_s S$ equal to the cogenus $\delta= \sum \delta_i$, i.e. the difference between the arithmetic and the geometric genera of $D$. Let $M$ be a rank one torsion free $R_i$--module, viewed as $P$-module. Recall that $M$ admits a length one free resolution
\[
0 \to P^n \stackrel{\varphi}{\longrightarrow} P^n \to M \to 0
\]
and that the $j$--th Fitting ideal $\mc F_{j}(M) \subset R_i$  of $M$ is the ideal of $R_i$ generated by the $(n-j)$--minors of the matrix $\varphi$. It is independent of the choice of the resolution. For example, the $0$--th Fitting ideal is the ideal generated by the local equation $f_i$ of the curve at $p_i$, i.e., $(\det \varphi)= (f_i) \subset R_i$. As for the first Fitting ideal, by \cite[Prop. C2]{FGvS} it is the image in $R_i$ of the evaluation map
$
M \times \Hom(M, R_i) \to R_i$. For later use, we highlight the following remark :

\begin{remark} \label{fitting of dual}
$\mc F_{1}(M) =\mc F_{1}(M^\vee)  \subset R_i$.
\end{remark}

We can now formulate the following key consequence of \cite{FGvS}.

\begin{prop} \cite{FGvS} \label{smoothness fgvs} \begin{enumerate}
\item[a)] For any $F \in \ov \Jac(D)$, the image of the tangent space $T_F\ov \Jac(\mc D)$ in $T_s S$ contains the space $V(D)$ defined in (\ref{V}).\\
\item[b)] There exists an $F$ in  $\ov \Jac(D)$ such that the image is exactly $V(D)$.\\
\item[c)] $\ov \Jac(\mc D_B)$ is smooth along $\ov \Jac(D)$ if and only if the image of $T_0 B$ in $T_s S$ is transversal to $V(D)$.
\end{enumerate}
\end{prop}
\begin{proof}
By \cite[Proposition C1]{FGvS} the image of the tangent space $T\Def(F_{p_i}, D_{p_i})$ in $T\Def(p_i)$ equals  the image
\be \label{image of fitting}
W(F_{p_i}) \subset R_i\slash (\partial_x f_i, \partial_y f_i)=T\Def(p_i)
\ee
of the first Fitting ideal $\mc F_{1}(F_i) \subset R_i$. By \cite[Corollary C3]{FGvS}, $\mc F_{1}(F_i)  \supset I_i$ and hence the first statement follows from the definition of $V(D)$ and the fact that the tangent map $T_F\ov \Jac(\mc D) \to \prod T\Def(F_{p_i}, D_{p_i})$ is surjective. The third statement follows from the observation made in  Remark C4 of \cite{FGvS}  that $\mathcal F_1(\ov R_i)=I_i$, and from the fact that by Proposition \ref{cook rego}  above there exists an $F \in \ov \Jac(D)$ with local type $\{\ov R_i\}$. Statement $c)$ is clear, once we recall that we are assuming that $B$ is smooth and that $\ov \Jac(D_B)=\ov \Jac(D)\times_S B$.
\end{proof}

Let us now get back to our situation and consider
\be \label{f and i}
f: \wt C \to C, \quad \text{ and }\,\, \iota: \wt C \to \wt C,
\ee
an \'etale double cover of reduced and irreducible curves with planar singularities and the corresponding involution on $\wt C$.
We denote by $\Def(\wt C, C)$ the deformation functor of the map $\wt C \to C$ \cite[Def.3.4.1]{sernesi} whose tangent space can be identified by \cite{ran} with $T^1_C=\Ext^1(\Omega^1_C, \mc O_C)$ (that is to $T\Def(C))$, viewed as the $\iota$--invariant part of $T^1_{\wt C}=\Ext^1(\Omega^1_{\wt C}, \mc O_{\wt C})$. A semi--universal family for this functor can be described as follows. Let $\mc C \to S$ be a semi--universal family for $C$, centered at a point $s \in S$. As in \cite{casaetal2}, we can consider the finite group scheme $\Pic^0(\mc C)[2] \subset \Pic^0(\mc C)$ over $S$ which  parametrizes $2$--torsion line bundles on the curves in $\mc C \to S$. Since we are in characteristic zero, the morphism $\Pic^0(\mc C)[2] \to S$ is \'etale and therefore the natural transformation
\[
(\Pic^0(\mc C)[2], \eta) \to \Def(\wt C, C),
\]
which one can easily check to be smooth, is an isomorphism at the level of tangent spaces. Hence, $\Pic^0(\mc C)[2]$ is the base of a semi--universal family for $\Def(\wt C, C)$.
Since $\Pic^0(\mc C)[2] \to S$ is \'etale, and an isomorphism on tangent spaces, we can replace $S$ by $\Pic^0(\mc C)[2]$ so that $(S, s)$ is a semi--universal space for both $\Def(C)$ and $\Def(\wt C, C)$. In particular, we have a family
\[
f: \wt \calC_S \to \calC_S:=\calC, \quad \iota: \wt \calC_S \to \wt \calC_S
\]
of \'etale double covers of integral curves with planar singularities.

We can finally state and prove the smoothness criterion for the relative Prym variety.

\begin{theorem} \label{smoothness properties} Let the notation be as above.
\begin{enumerate}
\item[(1)] The relative compactified Prym variety $\ov\Prym{(\wt \calC_S \slash \calC_S)}$ over the semi--universal deformation space is smooth.
\item[(2)] For any smooth base $B$ and any family of double covers as in (\ref{family of covers}), the relative compactified Prym variety $\ov \Prym{(\wt \calC_B \slash \calC_B)}$ is smooth along $\ov\Prym{(\wt \calC_b \slash \calC_b)}=\nu_B^{-1}(b)$ if and only if the image of the tangent map $T_b B \to T\Def(\calC_b)$ of the classifying morphism is transversal to the space $V(\calC_b) \subset T\Def(\calC_b)$ defined in (\ref{V});
\item[(3)] $\ov\Prym{(\wt \calC_B \slash \calC_B)}$ is smooth along $\ov\Prym{(\wt \calC_b \slash \calC_b)}=\nu^{-1}(b)$ if and only if $\J{\mc C_B}$ is smooth along $\J{\mc C_b}=\pi^{-1}(b)$;
\end{enumerate}
\end{theorem}
\begin{proof}
Consider the curve $\wt C$ and the involution $\iota: \wt C \to \wt C$. A result of Rim (Corollary in \cite{Rim}) ensures that we may consider a $\iota$--equivariant semi--universal family $\wt {\mc C}_{\wt S} \to \wt S$ for $\wt C$. By definition, this is a semi--universal family for $\Def(\wt C)$ that has the additional property of admitting compatible actions of $\iota$ on $\wt {\mc C} $ and on $\wt S$. Let us then consider such a family. We set $T=:\Fix(\iota) \subset \wt S$ and we denote by $\wt \calC_T \to T$ the restriction of the semi--universal family to $T$. Then
\[
\calC_T:= \wt \calC_T  \slash \iota \to T
\]
is a family of integral curves with locally planar singularities. If $\wt S$ is centered at $s \in  \wt S$, then $\calC_s=C$, and the tangent space of $T$ at $s$ is the $\iota$--invariant subspace $(T_s \wt S)^\iota= T^1_C=T\Def(\wt C, C)$. This shows that $T$ is a semi--universal space for $\Def(\wt C, C)$. Hence, to prove $(1)$ it is enough to prove that $\ov\Prym{(\wt \calC_T \slash \calC_T)}$ is smooth. But this is clear, since $\ov\Prym{(\wt \calC_T \slash \calC_T)}$ is  by definition just the component of the fixed locus of $\tau$ that contains the zero section of $\ov \Jac(\wt \calC_{\wt S})$. Since $\ov \Jac(\wt \calC_{\wt S})$ is smooth, so is every component of fixed locus of an involution acting on it.

For item $(2)$, we can reason as in the proof of Proposition \ref{smoothness fgvs}, provided we understand for, any $F \in \ov\Prym{(\wt \calC_b \slash \calC_b)}$,  the image  in $T\Def(\calC_b , \calC_b)$ of the tangent space of the relative compactified Prym variety over a semi--universal deformation space $T$ for $\Def(\wt \calC_b , \calC_b)$.  Indeed, since by Proposition \ref{base change} $\ov\Prym{(\wt \calC_B \slash \calC_B)}=\ov\Prym{(\wt \calC_T \slash \calC_T)} \times_T B$, it is sufficient to prove that for any $F \in \ov\Prym{(\wt \calC_b \slash \calC_b)}$, the image of $T_F\ov\Prym{(\wt \calC_T \slash \calC_T)}$ in $T\Def(\wt \calC_b , \calC_b)=T^1_{\calC_b}$ contains the support $V(\calC_b)$ of the tangent cone to the equigeneric locus. To see this, we argue as follows. Set $C=\calC_b$ and $\wt C=\calC_b $ and let $\{p_1, \dots, p_k, q_1, \dots, q_k\}$ be the singular points of $\wt C$, with $f^{-1}(x_i)=\{p_i, q_i\}$.
Consider an $F \in \ov\Prym{(\wt C \slash C)}$ and let $\wt S$ and $T$ be as above.
The tangent map
\[
\Xi: T_F\ov \Jac(\wt \calC_{\wt S}) \longrightarrow \prod_{i=1}^k T^1_{\wt C_{p_i}} \times T^1_{\wt C_{q_i}},
\]
is equivariant with respect to the two involutions $\tau$, which acts on $T_F\ov \Jac(\wt \calC_{\wt S})$ with fixed locus $T_F\ov\Prym{(\wt \calC_T \slash \calC_T)}$, and $\iota$, which acts on $\prod_{i=1}^k T^1_{\wt C_{p_i}} \times T^1_{\wt C_{q_i}}$ by interchanging $T^1_{\wt C_{p_i}}$ with $ T^1_{\wt C_{q_i}}$ (which are isomorphic since $\iota(p_i)=q_i$). By item (a) in Proposition \ref{smoothness fgvs}, we know that the image of $\Xi$ is
\[
\prod_i W(F_{p_i}) \times W(F_{q_i}) .
\]
Here as in (\ref{image of fitting}) $W(F_{p_i})$ denotes the image of the first Fitting ideal of $F_{p_i}$ in $T^1_{\wt C_{p_i}}$. Since $\tau(F)=F$, $F_{q_i}=F_{p_i}^\vee$ and hence, by Remark \ref{fitting of dual}, $W(F_{p_i}) \cong W(F_{q_i})$. It follows that the image of $T_F\ov\Prym{(\wt \calC_T \slash \calC_T)}$ in $\prod_{i}T\Def(\wt C, p_i) \times T\Def(\wt C, q_i)$, which is nothing but the $\iota$--invariant subspace of $\im \Xi$, is equal to the product of diagonals $ \prod \Delta_{W(F_{p_i})}$. Under the identification $T^1_C=(T_s \wt S)^ \iota$, the subspace $V(C) \subset T^1_C$ corresponds to the preimage in $T_s \wt S$ of the product $\prod \Delta_{V(\wt C, p_i)} \subset \prod T^1_{\wt C_{p_i}} \times T^1_{\wt C_{q_i}}$. Since we know that $W(F_{p_i}) \supset V(\wt C, p_i)$, it follows that $ \prod \Delta_{W(F_{p_i})} \supset \prod \Delta_{V(\wt C, p_i)}$ and hence that the image of $T_F\ov\Prym{(\wt \calC_T \slash \calC_T)}$ in $T^1_C$ contains $V(C)$.

The only thing we are left to prove is that there exists an $F \in \ov\Prym{(\wt C \slash C)}$ such that this image is exactly $V(C)$. This is done, like in Lemma \ref{G}, by considering a sheaf $F' \in \ov \Jac(\wt C)$ of local type $\{ \ov {\mc O}_{\wt C,p_1}, \dots, \ov {\mc O}_{\wt C,p_k}, \mc O_{\wt C,p_1}, \dots,  \mc O_{\wt C,p_k}\}$, where $\ov {\mc O}_{\wt C,p_i}$ is the normalization of $ \mc O_{\wt C,p_i}$, and setting
\[
F=F' \otimes \iota^* {F'}^\vee.
\]
As for statement $(3)$, it follows from $(2)$ and statement $(c)$ of Proposition \ref{smoothness fgvs}.
\end{proof}

\begin{cor}[of the proof of $(1)$ of Proposition \ref{smoothness properties}] \label{disjoint components}
For any base $T$, the closure $\ov {\mc P}_i$ of the connected components of $\Fix(-\iota^*) \subset \Pic^0_{\wt \calC_B}$ do not intersect.
\end{cor}
\begin{proof}
Since $\Fix(\tau) \subset \ov \Jac(\wt \calC_{\wt S})$ is smooth, its irreducible components are smooth and disjoint. In particular, the closure in  $\ov \Jac(\wt \calC_{\wt S})$ of the components of $\Fix(-\iota^*) \subset \Pic^0_{\wt \calC}$, which are irreducible components of $\Fix(\tau)$ are smooth and disjoint. If this is true over the semi--universal family, it is \emph{a fortiori} true that the closures of the  $\mc P_i$ are disjoint over an arbitrary base.
\end{proof}

\begin{cor}
If $\ov \Prym{(\wt \calC_B \slash \calC_B)}$ and $B$ are smooth, then $\nu_B: \ov \Prym{(\wt \calC_B \slash \calC_B)} \to B$ is flat.
\end{cor}

\section{Descent -- from the relative Prym to the relative intermediate Jacobian} \label{descent section}
In the previous section, we  developed a method for associating to any family of double covers of irreducible locally planar curves a relative compactified Prym variety. We now apply these results to the double cover of curves that come up in our situation and get a relative compactified Prym $\overline \calP$. The transversality arguments of Section \ref{sectrans18janvier} guarantee the smoothness of $\overline \calP$. Unfortunately, this flat family $\overline \calP$ (of relative dimension $5$) lives over the relative Fano variety $\calF$ (or more precisely an open subset of it) and not over the base $B=(\bP^5)^\vee$ as would be needed in order to compactify the intermediate Jacobian fibration $\mathcal J_U$.  It is therefore necessary to descend $\overline \calP$ to a family $\ocJ$ over $B$ that will give the desired compactification of $\cJ_U\rightarrow U$. This descent argument is the content of this section.

Let $X$ be a general cubic fourfold (or a general Pfaffian cubic). Let $p: \calF \to B$ be the relative Fano surface, let $\calF^0/B$ the non empty open subset of very good lines, in particular
not passing through the singular points of the considered
hyperplane section. Then $\calF^0 \to B$ is smooth and  by Proposition \ref{propirred},  it is  surjective. Let
\[
\xymatrix{
\wtC_{\calF^0} \ar[rr] \ar[dr] & & \calC_{\calF^0}\ar[dl] \\
& \calF^0 &
}
\]
be the associated family of plane quintic curves with their \'etale double covers. By Proposition \ref{propirred}, the curves in the two families are reduced and irreducible. We may therefore apply the results of Section \ref{sectprym} and construct the relative compactified Prym variety
\[
 \nu: \ov P_{\calF^0} := \ov \Prym{(\wtC_{\calF^0} \slash\calC_{\calF^0})} \to \calF^0.
\]

\begin{prop} \label{relative prym smooth}
The relative compactified Prym
$ \ov P_{\calF^0} $ is smooth and the morphism $ \nu:  \ov P_{\calF^0} \to {\calF^0}$ is flat of relative dimension $5$.
\end{prop}
\begin{proof}
This follows immediately from (2) of Proposition \ref{smoothness properties} and from Corollary \ref{corotransdef}.
\end{proof}

As usual, we let $U$ (respectively $U_1$) be the open subset of $B$ parametrizing hyperplane sections that are smooth (respectively  that have a single ordinary node). We set $\calF^0_U=\calF^0\times_U B$ and let $\wtC_{\calF^0_U}$ be the restriction of the family of curves. We use the analogous notation for $U_1$. For any $t \in B$, and any $\ell \in \calF^0_t=F( \cY_t)$, the curve $\wtC_{(\ell,t)}$ is the curve of lines in $\cY_t$ meeting the line $ \ell \subset \cY_t$. We let $\calL_U \subset  \wtC_{\calF^0_U} \times_U \cY_U$ be the corresponding universal family of lines of the smooth hyperplane sections. For any $x \in \wtC_{(\ell,t)}$ we let $\calL_{x}$ be the corresponding  line in $\cY_t$. There is a relative Abel--Jacobi map
\[
\Phi_{\calL_U}: \wtC_{\calF^0_U} \to \mathcal{J}_U, \quad \wtC_{(\ell,t)} \ni x \mapsto \Phi_{\cY_t}(\calL_{x} -\ell) \in \mathcal{J}_t=J(\mathcal{Y}_t)
\]
inducing a morphism
\[
\Psi:\Jac(\wtC_{\calF^0_U}) \to  \mathcal{J}_{\calF^0_U}=\mathcal{J}_U\times_{U}\calF^0_{U}.
\]
Since for every $x \in \wtC_{(\ell,t)}$, the rational equivalence class of the cycle $\calL_{x}+\iota \calL_{x}$ in $\cY_t$ is constant, the morphism $\Psi$ factors via $(1-\iota) \Jac(\wtC_{\calF^0_U})=\Prym{(\wtC_{\calF^0_U} \slash \calC_{\calF^0_U} )}$, thus inducing a morphism from the generalized relative Prym variety
\be \label{descent iso}
 \Prym{(\wtC_{\calF^0_U} \slash \calC_{\calF^0_U} )} \to \mathcal{J}_{\calF^0_U},
\ee
which is an isomorphism by a result of Mumford \cite{tjurin}. In particular, over the smooth locus $U$, the relative Prym variety is the pull-back to $\mathcal{F}^0_U$ of  the intermediate Jacobian fibration. The following lemma shows that it is the case also over the locus $U_1$ parametrizing one--nodal hyperplane sections.

\begin{lemma} \label{leisoenhaut}
The isomorphism (\ref{descent iso}) extends to an isomorphism 
\begin{eqnarray}\label{eqisodu31oct} \Prym{(\wtC_{\calF^0_{U_1}}\slash \calC_{\calF^0_{U_1}})} \to  \mathcal{J}_{\calF^0_{U_1}}=\mathcal{J}_{U_1}\times_{U_1}\calF^0_{U_1}.
\end{eqnarray}
In particular $\mathcal{J}_{U_1}$ is smooth.
\end{lemma}
\begin{proof}
The extensions of the families $\Prym{(\wtC_{\calF^0_U} \slash \calC_{\calF^0_U} )}$ and $\mathcal{J}_{\calF^0_U}$ to families $\Prym{(\wtC_{\calF^0_U} \slash \calC_{\calF^0_U} )}^{\circ}$, resp. $\mathcal{J}_{\calF^0_U}^{\circ}$ of semi--abelian varieties over the boundary $ \calF^0_{U_1} \setminus  \calF^0_U$, are determined by the monodromy of the local systems over $ \calF^0_U$: indeed, if $\mathcal H^{0,1}_{\operatorname{Prym}}$ and $\mathcal H^{0,1}_{\mathcal{J}}$ are the Hodge bundles of the families $ \nu: \Prym{(\wtC_{\calF^0_U} \slash \calC_{\calF^0_U} )} \to {\calF^0_U}$ and $\rho: \mathcal{J}_{\calF^0_{U}} \to\calF^0_{U}$, respectively, then the two families of semi--abelian varieties (or rather their sheaves of local sections) $\Prym{(\wtC_{\calF^0_{U_1}}\slash \calC_{\calF^0_{U_1}})}^{\circ}$ and $\mathcal{J}_{\calF^0_{U_1}}^{\circ}$ are given, respectively, by
\[
\ov{\mathcal{H}}^{0,1}_{\operatorname{Prym}} \slash j_* R^1 \nu_* \mathbb{Z}, \quad \text{and} \quad  \ov{\mathcal{H}}^{0,1}_{\mathcal{J}} \slash j_* R^1 \rho_* \mathbb{Z}.
\]
Here $\ov{\mathcal{H}}^{0,1}_{\operatorname{Prym}}$ and $\ov{\mathcal{H}}^{0,1}_{\mathcal{J}}$ are the canonical extensions of the Hodge bundle across $ \calF^0_{U_1} \setminus  \calF^0_{U}$, which is smooth, and $j:  \calF^0_U \to  \calF^0_{U_1}$ is the inclusion. Since the two families are isomorphic over $ \calF^0_U$ by (\ref{descent iso}), so are the corresponding local systems and Hodge bundles, and hence so are the canonical extensions. 
We thus get an isomorphism
 $$\Prym{(\wtC_{\calF^0_{U_1}}\slash \calC_{\calF^0_{U_1}})}^{\circ} \cong \mathcal{J}_{\calF^0_{U_1}}^{\circ}=\mathcal{J}_{U_1}^{\circ}\times_{U_1}\calF^0_{U_1}.$$
 The fact that this isomorphism extends to the Mumford compactifications
 where along the boundary, the $\mathbb{C}^*$-bundles are replaced with the corresponding
 $\mathbb{P}^1$-bundle with the sections $0$ and $\infty$ glued via
 a translation follows from the fact that the Mumford compactification is canonical. We thus get the desired isomorphism (\ref{eqisodu31oct}).
\end{proof}

 From now on, we will use the following notation (justified by Lemma \ref{leisoenhaut}):
For any morphism $f: M\rightarrow \calF^0$ with induced morphism
$f'=p\circ f: M\rightarrow B$, we will denote by $\mathcal{J}_M$ the pull-back
$\ov P_{\calF^0}\times_{\calF^0}   M$ and $\pi_M: \mathcal{J}_M\rightarrow M$ the second projection. Over $M_1:={f'}^{-1}(U_1)$, one has  ${\mathcal{J}_M}_{\mid M_1}={\mathcal{J}_{U_1}}\times_{U_1}M_1$ by Lemma
\ref{leisoenhaut}.
The aim of this section is to show a  result extending in some sense
Lemma \ref{leisoenhaut} over
the whole of $B$, that is, to construct  a projective compactification $\overline{\mathcal{J}}$
of $\mathcal{J}_{U_1}$ that is flat over $B$, whose pull-back to $\calF^0$ will be isomorphic to $\ov P_{\calF^0}=\mathcal{J}_{\calF^0}$. Then $\overline{\mathcal{J}}$ will be clearly smooth.
This is a descent problem which will use the following Proposition \ref{prothetaample}.
The morphism $\pi_{U_1}:\mathcal{J}_{U_1}\rightarrow U_1$ is projective. In fact, there is a canonical Theta divisor
$\Theta_1\subset \mathcal{J}_{U_1}$ defined as the Zariski closure in $\mathcal{J}_{U_1}$ of the
 canonically defined divisor $\Theta\subset \mathcal{J}$ (see Lemma \ref{leiinvtheta} for more detail). Using Lemma \ref{leisoenhaut}, we get by pull-back
a divisor $\widetilde{\Theta_1}$ on $\mathcal{J}_{U_1}\times_{U_1}\calF^0_{U_1}=\mathcal{J}_{\calF^0_{U_1}}$, and as
the morphism  $ \pi_{\calF^0}:  \mathcal{J}_{\calF^0} \to {\calF^0}$ is flat, $
\mathcal{J}_{\calF^0_{U_1}}\subset
\mathcal{J}_{\calF^0}$ has a complement of codimension $\geq 2$, and thus
$\widetilde{\Theta_1}$ extends uniquely to a divisor $\overline{\widetilde{\Theta_1}}$ on
$\mathcal{J}_{\calF^0}$.

\begin{prop}\label{prothetaample} The divisor $\overline{\widetilde{\Theta_1}}$ on
$\mathcal{J}_{\calF^0}$ is $\pi_{\calF^0}$-ample.
\end{prop}
The proof will use several lemmas.  We first recall  the following:
\begin{lemma}\label{leiinvtheta} For any smooth cubic threefold $Y$, there is a canonically defined
Theta divisor $\Theta_Y\subset J(Y)$ which is invariant under the involution $-1$.
\end{lemma}
\begin{proof} The divisor $\Theta_Y$ is
defined as follows:  Consider the family $\mathcal{C}\rightarrow \mathcal{H}$
of rational cubic curves in $Y$. We have the Abel-Jacobi map
$$\Phi_{\mathcal{C}}:\mathcal{H}\rightarrow J(Y),\,\,s\mapsto \Phi_{Y}(\mathcal{C}_{s}-h^2),$$
where $h=c_1(\mathcal{O}_{Y}(1))\in {\rm CH}^2(Y)$, so that $\mathcal{C}_{s}-h^2$ is homologous to $0$
in $Y$.  The fact that the image of this map  is a Theta divisor in $J(Y)$ is proved
in  \cite[Section 13]{CG} (see also \cite{harrisrothstarr}). The fact that it is a $(-1)$-invariant divisor in $J(Y)$
follows from  the following observation
(\cite{harrisrothstarr}):
\begin{sublemma} Let $C\subset Y$ be a general cubic rational normal curve in a smooth cubic threefold. Then, there exists two lines $L,L'\subset Y$ such that $C$ is rationally equivalent in $Y$ to $c_1(\mathcal{O}_Y(1))^2+L-L'$. Conversely, for two general lines $L,L'$ on $Y$, $c_1(\mathcal{O}_Y(1))^2+L-L'$  is rationally equivalent in $Y$ to a smooth rational cubic curve $C$.
\end{sublemma}
\begin{proof} 
The curve $C$
generates a $\mathbb{P}^3$ which intersects $Y$ in a  cubic surface
$S\subset Y$ which is smooth because $C$ is generic. The linear system
$|\mathcal{O}_S(C)|$ is a $\mathbb{P}^2$ which provides a birational map
$\phi:S\rightarrow \mathbb{P}^2$, contracting $6$ lines $L_i$ in $S$
to points $p_i$.  The curve $C$  belongs to the linear system
$|\phi^*\mathcal{O}_{\mathbb{P}^2}(1)|$. Choose a  line $L$ contracted to a point $p$, and consider the proper transform
$L'$ in $S$ of a conic passing through all points $p_i$ except $p$. Then $L'$ is a line in $S$
which belongs to the linear
system $|\phi^*\mathcal{O}_{\mathbb{P}^2}(2)(-\sum_{p_i\not=p}L_i)|$ and thus
$L'-L$ is rationally equivalent in $S$ to  $\phi^*\mathcal{O}_{\mathbb{P}^2}(2)(-\sum_{i}L_i)$.
But $K_S=-h_{\mid S}=\phi^*\mathcal{O}_{\mathbb{P}^2}(-3)(\sum_{i}L_i)$ (with $h=c_1(\mathcal{O}_Y(1))$), hence
we get
$$-h_{\mid S}+\phi^*\mathcal{O}_{\mathbb{P}^2}(1)=L-L'$$
in ${\rm CH}^1(S)$, and thus $C=h^2+L-L'$ in ${\rm CH}^2(Y)$. Conversely, if $L,\,L'$ are two lines
in $Y$, the $\mathbb{P}^3=P_3$ generated by $L$ and $L'$ intersects $Y$ along a smooth cubic surface,
and for a given point $x\in L$, the plane $\langle L',x\rangle\subset P_3$ intersects $Y$ along the union of
$L'$ and a conic $C'$ meeting $L$ at $x$. The curve $L\cup C'$ is a reducible rational cubic curve $C$ in $S$ (which deforms
to a smooth rational cubic curve).
\end{proof}
It follows that the divisor ${\rm Im}\,\Phi_{\mathcal{C}}$ (or rather $6$ times this divisor) is also equal to the image
in $J(Y)$ of the difference map
$F(Y)\times F(Y)\mapsto J(Y)$,\,$(l_1,l_2)\mapsto l_1-l_2$. Thus it is invariant
under the involution $(-1)$ of $J(Y_t)$.
\end{proof}

Coming back to our family $\pi_U:\mathcal{J}_U\rightarrow U$, on which the $(-1)$-involution acts over
$U$, we constructed  from Lemma \ref{leiinvtheta}  a canonical Theta divisor $\Theta\subset \mathcal{J}_U$ which is $-1$-invariant, with Zariski closure  $\Theta_1\subset \mathcal{J}_{U_1}$.
We now have :
\begin{lemma} \label{lepic} For any
dominant morphism $f':M_1\rightarrow U_1$ with $M_1$ irreducible,   with canonical lift
$f:\mathcal{J}_{M_1}\rightarrow \mathcal{J}_{U_1}$, the $(-1)$-invariant part of ${\rm Pic}\,\mathcal{J}_{M_1}/M_1$, where
$\mathcal{J}_{M_1}:=\mathcal{J}\times_{U_1}M_1$,  is
generated modulo torsion  by $f^*\Theta_1$.
\end{lemma}
\begin{proof} The relative Picard group ${\rm Pic}\,\mathcal{J}_{M_1}/M_1:={\rm Pic}\,\mathcal{J}_{M_1}/\pi_{M_1}^*{\rm Pic}\,M_1$
  injects in
the Picard group of the  fiber $J(Y_t)$, where $t\in M$ is very general, because the fibers of
the map $\pi_M$ (which are also the fibers of the map $\pi_{U_1}:\mathcal{J}_{U_1}\rightarrow U_1$) are all reduced and
irreducible. In particular, it injects into ${\rm Pic}\,\mathcal{J}_{M_U}/M_U$, where
$M_U:={f'}^{-1}(U)\subset M_1$.
We   are thus reduced to proving that the $(-1)$-invariant part of ${\rm Pic}\,J(Y_t)$ is modulo  torsion
generated by $\Theta_{Y_t}$ for $t$ very general in $U_1$.
However, modulo torsion, the $(-1)$-invariant part of ${\rm Pic}\,J(Y_t)$ is isomorphic to the N\'eron-Severi group of
$J(Y_t)$. Finally, recall that we have a canonical isomorphism
$H^3(Y_t,\mathbb{Q})\cong H^1(J(Y_t),\mathbb{Q})$ which provides more generally an isomorphism of
local systems over $U$:
$$R^3u_*\mathbb{Q}\cong R^1\pi_{U*}\mathbb{Q}.$$
For the local system on the left, the corresponding monodromy group
$${\rm Im}\,(\rho:\pi_1(U,t)\rightarrow {\rm Aut}\,H^3(Y_t,\mathbb{Z})))$$
is the full symplectic group of the intersection pairing by Picard-Lefschetz theory
\cite{monodromy}. Hence the same is true
for the local system on the right. As $M_1$ is irreducible and
the morphism $f':M_1\rightarrow U_1$ is dominating,
the image of the morphism $$f'_{U*}:\pi_1(M_U,m_t)\rightarrow \pi_1(U,t),$$ where  $f'_U:M_U\rightarrow U$
is the restriction of $f'$ and $f'_U(m_t)=t$, is a subgroup of finite index in $\pi_1(U,t)$.
Thus the monodromy group of the family $\mathcal{J}_{M_U}\rightarrow M_U$ acts via a subgroup of finite index
of $Sp(H^1(\mathcal{J}_{m_t},\mathbb{Z}))$ on the cohomology of $\mathcal{J}_{m_t}$.
On the other hand, it is  a general fact that the  monodromy acts on
$H^2(J(Y_{m_t}),\mathbb{Q})=\bigwedge^2H^1(J(Y_{m_t}),\mathbb{Q})$ by preserving (for very general $t$) the N\'eron-Severi group of $J(Y_t)$ and
with finite orbits on ${\rm NS}\,(J(Y_t))_\mathbb{Q}$. The only elements of $\bigwedge^2H^1(J(Y_{m_t}),\mathbb{Z})$
which have finite orbit under $Sp(H^1(\mathcal{J}_{m_t},\mathbb{Z}))$ are the multiples of the class $\Theta_t$
and
we conclude that ${\rm NS}\,(J(Y_t))_\mathbb{Q}=\mathbb{Q}\Theta_{Y_t}$.
\end{proof}
\begin{proof}[Proof of Proposition \ref{prothetaample}] The morphism  $ \pi_{\calF^0}:  \mathcal{J}_{\calF^0} \to {\calF^0}$ is projective.
Next we recall from Lemma \ref{lemmataumor} that the involution $(-1)$ acting on $\mathcal{J}_{U_1}$, hence on
$\mathcal{J}_{\calF^0_{U_1}}$ by Lemma \ref{leisoenhaut},
 extends to an involution acting on $\mathcal{J}_{\calF^0}$ over $\calF^0$: it
can be defined  as the involution $\mathcal{F}\mapsto \mathcal{F}^\vee$ acting on reflexive sheaves on $\wt C_{l,Y}$,  where for any point $([l],[Y])\in \calF^0$, the  curve
$\wt C_{l,Y}$ is the incidence curve of the line $ l\in F(Y)$ (cf. Lemma \ref{lemmataumor}). Notice that since the relative compactified Prym is defined as a component of the fixed locus of $-\iota^*$, where $\iota: \wt C_{l,Y} \to \wt C_{l,Y}$ is the usual natural involution, we have $(-1)=\iota^*$ on $\mathcal{J}_{\calF^0_{U_1}}$.
It follows that starting from any $\pi_{\calF^0}$-relatively ample line bundle
$\mathcal{L}$ on $\mathcal{J}_{\calF^0}$, we can construct a $\pi_{\calF^0}$-relatively ample line bundle
$\mathcal{L}'$ which is both $\pi_{\calF^0}$-ample and $(-1)$-invariant, namely
$$\mathcal{L}'=\mathcal{L}\otimes (-1)^*\mathcal{L}.$$
 We now apply Lemma \ref{lepic} to the natural morphism
 $f'=p_{\mid \calF^0_{U_1}}: \calF^0_{U_1}\rightarrow U_1$
 and the $(-1)$-invariant line bundle
$\mathcal{L}'_{\mid \mathcal{J}_{\calF^0_{U_1}}}$. It says that up to replacing $\mathcal{L}'$ by a multiple, we
have \begin{eqnarray}
\label{eqpourL'}{\mathcal{L}'}_{\mid \mathcal{J}_{\calF^0_{U_1}}}=f^*\mathcal{O}(d\Theta)\otimes \pi_{\calF^0_{U_1}}^*\mathcal{N}_1,
\end{eqnarray} for some integer $d$, where $\mathcal{N}_1$ is  a line bundle on $\calF^0_{U_1}$.
We observe now that ${\rm Pic}\,\calF^0_{U_1}={\rm Pic}\,\calF^0$ because $\calF^0$ is smooth and ${\rm codim}\,(\calF^0\setminus \calF^0_{U_1}\subset \calF^0)\geq2$, and for the same reason,
${\rm Pic}\,\mathcal{J}_{\calF^0_{U_1}}={\rm Pic}\,\mathcal{J}_{\calF^0}$ because $\pi_{\calF^0}$ is flat, hence also  ${\rm codim}\,(\mathcal{J}_{\calF^0}\setminus \mathcal{J}_{\calF^0_{U_1}}\subset \mathcal{J}_{\calF^0})\geq2$, and $\mathcal{J}_{\calF^0}$ is smooth.
Thus the line bundle $\mathcal{N}_1$ extends to a unique line bundle
$\mathcal{N}$ on $\calF^0$ and (\ref{eqpourL'}) is true as well over $\mathcal{J}_{\calF^0}$, proving that
\begin{eqnarray}\label{eqpourlprime} {\mathcal{L}'}=
\mathcal{O}(d\overline{\widetilde{\Theta_1}})\otimes \pi_{\calF^0}^*\mathcal{N}.
\end{eqnarray}
As ${\mathcal{L}'}$ is $\pi_{\calF^0}$ ample, so is $\overline{\widetilde{\Theta_1}}$ by (\ref{eqpourlprime}).
\end{proof}
Using the above results and Proposition \ref{relative prym smooth}, we now prove the following result (see main Theorem  in the introduction) :
\begin{theorem}\label{theoconst} Let $j_1:U_1\rightarrow B$ be the inclusion. Then

(i) For any sufficiently large integer $d>0$, the sheaf of algebras
$$\mathcal{E}^*:=R^0j_{1*}(\oplus_{k\geq0} R^0\pi_{U_1*}\mathcal{O}_{\mathcal{J}_{U_1}}(kd\Theta_1))$$
is a sheaf of algebras of  finite type, and each summand $R^0j_{1*}( R^0\pi_{U_1*}\mathcal{O}_{\mathcal{J}_{U_1}}(kd\Theta_1))$ is a locally free coherent sheaf on $B$.

(ii) The variety $\overline{\mathcal{J}}:={\rm Proj}\,(\mathcal{E}^*)\rightarrow B$ is a smooth projective  compactification of $\mathcal{J}_{U_1}$.

(iii)  The variety $\overline{\mathcal{J}}$ is irreducible
hyper-K\"ahler.
\end{theorem}
\begin{proof} (i) and (ii)  It suffices to prove the existence of $d$ locally in the Zariski topology.
For any $b\in B$, there exists by Corollary \ref{coroverygood} a very good line $l_b\in Y_b$.
As the family $\calF^0\rightarrow B$ of very good lines in the fibers of $u$ is smooth
over $B$, we may assume up to an \'etale base change
$f':M\rightarrow B$, $m\rightarrow b$, that  there is a section
$f:M\rightarrow \mathcal{F}_M,\,m\mapsto l_m$. Furthermore, the conclusion of Corollary \ref{corotransdef} holds, so that the corresponding family of plane quintic curves $C_{l_m}$ induces a versal deformations of
${\rm Sing}(C_{l_m})$.
We can then apply the results of Section \ref{sectprym} and especially  Theorem
\ref{smoothness properties}
which  provides
a smooth projective flat compactification
$\pi_M:\mathcal{J}_M\rightarrow M$ of $\mathcal{J}_{M_1}:=\mathcal{J}_{U_1}\times_{U_1}M_{1}$, where $M_{1}={f'}^{-1}(U_1)\subset M$.
If $\mathcal{L}$ is a $\pi_M$-relatively ample line bundle on $\mathcal{J}_M$, for some
$l_0$ large enough, $R^i\pi_{M*}\mathcal{L}^{\otimes n}=0$ for $n\geq l_0$, $i>0$, and thus we conclude
by flatness of $\pi_M$ that:

(a)
$R^0\pi_{M*}\mathcal{L}^{\otimes n}$ is locally free for $n\geq d_0$.

(b) $\oplus_k R^0\pi_{M*}\mathcal{L}^{\otimes kd_0}$ is a sheaf of finitely generated
$\mathcal{O}_M$-algebras.

(c) The smooth variety $\mathcal{J}_M$ is isomorphic over $M$ to $ {\rm Proj}\,(\oplus_k R^0\pi_{M*}\mathcal{L}^{\otimes kd_0})$.

Next,  let
$j_{M_1}: M_1\rightarrow M$ be the inclusion map. As $M\setminus M_1$ has codimension $\geq 2$,
and $\pi_M$ is flat, $\mathcal{J}_{M}\setminus \mathcal{J}_{M_1} $ also has codimension $\geq2$. As
$M$ and $\mathcal{J}_{M}$ are  smooth, we conclude that
\begin{eqnarray}\label{eqpresquefin22sep} R^0j_{M_1*}(R^0\pi_{M_1*}({\mathcal{L}}_{\mid \mathcal{J}_{M_1}}^{\otimes k}))=R^0\pi_{M*}({\mathcal{L}}^{\otimes k}).
\end{eqnarray}
We now assume that $\mathcal{L}$ is $-1$-invariant, so $\mathcal{L}=\mathcal{L}'$ satisfies
(\ref{eqpourlprime}). Up to shrinking $M$,
the line bundle $\mathcal{N}$ appearing in (\ref{eqpourlprime}) is trivial on $M$,
so that ${\mathcal{L}}_{\mid \mathcal{J}_{M_1}}=f_1^*\mathcal{O}(d'\Theta_1)$ for some integer
$d'$, where $f_1:\mathcal{J}_{M_1}\rightarrow \mathcal{J}_{U_1}$ is the natural map over
$f'_{\mid M_1}/M_1\rightarrow U_1$.
  Thus (a), (b), (c) above and (\ref{eqpresquefin22sep}) prove (i) and (ii) after pull-back to $M$.
In other words, we proved that (i) and (ii) are true \'etale locally on $B$, that
 is, after  \'{e}tale base changes
$f':M\rightarrow B$ of small Zariski open sets of $B$ covering $B$. This clearly implies
(i) and (ii), for example because an \'etale base change is a local isomorphism in the analytic topology and (i) and (ii) are local statements in the analytic topology.

(iii)  We know by Proposition \ref{proextnondeg} that $\mathcal{J}_{U_1}$ has a nondegenerate holomorphic $2$-form which by
Theorem \ref{theogen2form} (iii) extends to a nondegenerate holomorphic $2$-form on $\overline{\mathcal{J}}$, since
${\rm codim} (\overline{\mathcal{J}}\setminus \mathcal{J}_{U_1}\subset \overline{\mathcal{J}})\geq2$.
What remains to be done is to prove that $\overline{\mathcal{J}}$ is irreducible hyper-K\"ahler. We have the following
lemma:
\begin{lemma}\label{leholo} The holomorphic $2$-forms on any finite \'etale cover $\widetilde{\overline{\mathcal{J}}}$ of $\overline{\mathcal{J}}$ are
multiples of the form $\sigma$ coming from $X$.
\end{lemma}
\begin{proof} The variety $\overline{\mathcal{J}}$ contains the Zariski closure of the  Theta divisor $\Theta \subset \mathcal{J}$
which is birational to a $\mathbb{P}^1$-bundle over
the Lehn-Lehn-Sorger-van Straten variety $F_3(X)$. Indeed, recall
from \cite{llss} that $F_3(X)$ parametrizes birationally nets $|D|$ of rational cubic curves on cubic surfaces
$S\subset X$. Consider the $\mathbb{P}^1$-bundle $\mathbb{P}\rightarrow F_3(X)$ with fiber over $(S,D)$ the
$\mathbb{P}^1$ of hyperplanes in $\mathbb{P}^5$ containing $S$.
Then $\mathbb{P}$ admits a morphism to $(\mathbb{P}^5)^*$ whose fiber over $H_Y$ parametrizes
the nets of cubic rational curves in $Y$. We already mentioned that via the
map
$$|D|\mapsto \Phi_X(D-h^2)\in J(Y),$$
the set of such nets dominates (in fact, is birational to) the Theta divisor of $J(Y)$. This construction
in family over $U$ provides the rational map
$\mathbb{P}\dashrightarrow\Theta\subset \mathcal{J}\subset Z$.

The proof of Lemma \ref{leholo} is now immediate:
The variety $F_3(X)$ is simply connected (it is a deformation of $S^{[4]}$ for some
K3 surface $S$, see \cite{addle}), hence $\mathbb{P}$ is simply connected, so the rational map
$\mathbb{P}\dashrightarrow \overline{\mathcal{J}}$ constructed above lifts to
a rational map
$\mathbb{P}\dashrightarrow \widetilde{\overline{\mathcal{J}}}$ for any finite \'etale cover $\widetilde{\overline{\mathcal{J}}}$ of $\overline{\mathcal{J}}$.
As $F_3(X)$ is  an irreducible hyper-K\"ahler manifold and $\mathbb{P}\rightarrow F_3(X)$ is
a $\mathbb{P}^1$-bundle, the holomorphic $2$-forms on $\mathbb{P}$ are all multiples of the restriction of
$\sigma_{\overline{\mathcal{J}}}$. It thus only suffices to show that if $\alpha$ is a holomorphic $2$-form
on $\widetilde{\overline{\mathcal{J}}}$ which vanishes on the image $\Theta$ of $\mathbb{P}$, $\alpha=0$. That follows however
immediately from the fact that,
fiberwise, as $\Theta_t\subset \widetilde{\mathcal{J}_t}$ is an ample divisor, the restriction map
$$H^0(\widetilde{\mathcal{J}_t},\Omega_{\widetilde{\mathcal{J}_t}}^i)\rightarrow H^0(\Theta_t^0,\Omega_{\Theta^0_t}^i)$$
is injective for $t\in U$, and $0\leq i\leq 2$, where $\Theta_t^0$ is the smooth locus of $\Theta_t$.
\end{proof}
By the Beauville-Bogomolov decomposition theorem
\cite{beauville}, Lemma \ref{leholo} implies that
$\overline{\mathcal{J}}$ is irreducible hyper-K\"ahler.
\end{proof}

\section{Construction of a birational map in the Pfaffian case}\label{sectpfaffian}
We have established in Section \ref{sectrans18janvier} that a general Pfaffian cubic fourfold satisfies the same vesality statements as a general cubic, and thus  the relative intermediate Jacobian associated to a general Pfaffian cubic has a smooth projective hyper--K\"ahler compactification $\ov {\mc{J}}$. The purpose of this section is to establish that this compactification $\ov{\mc{J}}$ is birational to an OG10 hyper--K\"ahler manifold, and thus by  Huybrechts' result \cite{huybrechts} ({\it two birationally equivalent HK manifolds are deformation equivalent}) is deformation equivalent to OG10. This completes the proof of the main theorem (stated in the introduction).

What is used about Pfaffian cubic fourfolds $X$ is that they have associated (cf. \cite{bedo}) degree $14$ $K3$ surfaces $\Sigma$. We will prove below that birationally the relative intermediate Jacobian fibration $\ov{\mc{J}}$ associated to a Pfaffian cubic $X$ can be interpreted as a moduli space of sheaves on $\Sigma$, linking in with the original construction of OG10 of O'Grady \cite{og10}.

\begin{remark}
Of course, the same approach can be applied to other classes of cubic fourfolds for which there is an associated $K3$ surface (in the sense of Hassett \cite{hassett}). In particular, as mentioned in the introduction, the relative intermediate Jacobian fibration can be seen to be related to the OG10 construction in the case of nodal cubic fourfolds (cf. Hwang--Nagai \cite{hwangnagai}) or in the case of degenerations to the chordal cubic fourfolds (cf. O'Grady--Rapagnetta). In fact, in those cases the geometry of relating the relative intermediate Jacobian fibration to sheaves on the associated $K3$ is easier than in the Pfaffian case. However, the problem is that our construction will not lead to a smooth compactification $\ov{\mc{J}}$ (and thus Huybrechts' result does not suffices). In other words, it is essential for our proof of deformation equivalence to OG10 that  the general Pfaffian cubic satisfies both: it has an associated $K3$ and  it behaves similarly to a general cubic from the perspective of good lines (see Section \ref{sectrans18janvier}, esp. \ref{corallow} and \ref{lepfgood}).
\end{remark}

\subsection{The Pfaffian case - Statement of the main result and Strategy}
Let $X$ be a Pfaffian cubic fourfold. This means that there is a
$6$-dimensional vector space $V_6$ and a $6$-dimensional vector
space $W_6\subset \bigwedge^2V_6^*$ such that $X\subset
\mathbb{P}(W_6)$ is the Pfaffian cubic hypersurface defined by the
Pfaffian equation $\omega^3=0$. Following \cite{bedo}, let
$\Sigma\subset G(2,V_6)$ be the  surface  defined as follows:
$$\Sigma=\{[l]\in G(2,V_6),\,\omega_{\mid V_l}=0\,\forall \omega\in W_6\},$$
 where we denote by $V_l\subset V_6$ the  rank $2$ vector subspace corresponding to $[l]$ ($l$ will denote the corresponding projective line in $\mathbb{P}(V_6)$).
Being defined (for general $X$ as above) as the complete
intersection of $6$ linear Pl\"ucker forms on $G(2,6)$, $\Sigma$ is
a smooth degree $14$ K3 surface.

Our goal in this section  is to prove the following result:
\begin{theorem} \label{theobirat}The intermediate Jacobian fibration $\mathcal{J}$ of
$X$ is birational to the O'Grady moduli space
$\mathcal{M}_{2,0,4}(\Sigma)$ parametrizing rank $2$ semistable
sheaves on $\Sigma$ with $c_1=0,\,c_2=4$.
\end{theorem}
\begin{cor} \label{cordefoeq} The compactified Jacobian fibration $\overline{\mathcal{J}}$ of Theorem
\ref{theoconst} is a deformation of O'Grady's $10$-dimensional variety.
\end{cor}
\begin{proof} Indeed, Theorem \ref{theoconst} is valid when $X$ is a general Pfaffian
cubic fourfold, thanks to the results of Section \ref{sectranpf} which guarantee that the assumptions
needed to make the local construction of Section \ref{sectprym} work are satisfied in the general Pfaffian case. Hence we can rephrase Theorem
\ref{theobirat} saying that  our family $(\overline{\mathcal{J}}_{X})_{[X]\in W\subset \mathcal{H}yp_{4,3}}$ of hyper-K\"ahler compactified Jacobian fibrations
$\overline{\mathcal{J}}_{X}$ parametrized  by an open set $W$ of the  space of all cubic fourfolds
 has a smooth member
which is birational to O'Grady's $10$-dimensional variety $OG10$, which is also a smooth hyper-K\"ahler
manifold. We then apply \cite{huybrechts} to conclude that the varieties
$\overline{\mathcal{J}}_X$ are deformation equivalent to $OG10$.
\end{proof}
We will  heavily use the results of \cite{ilievmarku} (based on
\cite{markutikho}) and their generalization in \cite{kuznetsov}. The
proof of Theorem \ref{theobirat} will be completed in subsection
\ref{secconstbirat}. As the intermediate Jacobian fibration is a
fibration over
 $\mathbb{P}(W_6^*)$ into intermediate Jacobians of cubic $3$-folds
 $Y_3=H\cap X$ which come equiped with a Pfaffian representation induced from the one of $X$,
 we will devote sections \ref{subseccubic28dec} and \ref{subsecV1428dec}
  to the Pfaffian $3$-fold geometry, which is the subject of the papers
 \cite{ilievmarku}, \cite{kuznetsov}.
The cubic threefold $Y_3$ is represented as the Pfaffian cubic in
$\mathbb{P}(W_5)$ for some $5$-dimensional vector space $W_5\subset
\bigwedge^2V_6^*$. Associated to this data, we get the index $1$,
degree $14$ Fano $3$-fold $\bV_{14}\subset G(2,V_6)$ defined as
\begin{eqnarray}\label{eqV14} \bV_{14}=\{[l]\in G(2,V_6),\,\omega_{\mid V_l}=0\,\forall \omega\in W_5\}.
\end{eqnarray}
Notice that by construction these Fano $3$--fold all contain the K3 surface $\Sigma$ defined above. The strategy of the proof is the following: First one notices that the relative intermediate Jacobian $\mc J_U$ is birational to a relative moduli space of vector bundles $\mc M$ on the hyperplane sections of $X$ (see beginning of Section \ref{subseccubic28dec}). Then, to a vector bundle on a cubic threefold $Y_3$, $E_{Y_3} \in \mc M$, one would like to associate a vector bundle $E_{\bV_{14}}$ on the corresponding degree $14$ Fano $3$-fold $\bV_{14}$ in order to then restrict it to the K3 surface $\Sigma$. However, to define this assignment $E_{Y_3} \mapsto E_{\bV_{14}}$ we need some more data than $E_{Y_3}$ alone, so we end up first getting a map from a variety dominating $\mc J_U$ (map defined in (\ref{rational map!})), and then showing that this map factors through $\mc J_U$ (Proposition \ref{propfactor}).

The bundles in question on $Y_3$ and on $\bV_{14}$ will be obtained via elementary transformations from  two natural rank two vector bundles on $Y_3$ and on $\bV_{14}$, which we now define.
The variety $\bV_{14}$ comes equipped with the tautological rank $2$
dual vector bundle, that we will denote by $\mathcal{E}_{14}$. The
Pfaffian cubic $Y_3$ comes equipped with the natural rank $2$ vector
bundle $\mathcal{E}_3$ with fiber $({\rm Ker}\,\omega)^*$ over a
$2$-form $\omega\in\mathbb{P}(\bigwedge^2V_6^*)$ of rank $4$. As
$V_6^*=H^0(Y_3,\mathcal{E}_3)=H^0(\bV_{14},\mathcal{E}_{14})$, we have
natural maps
$$\phi:\mathbb{P}(\mathcal{E}_{14})\rightarrow
\mathbb{P}(V_6),\,\psi:\mathbb{P}(\mathcal{E}_{3})\rightarrow
\mathbb{P}(V_6)$$ which are easily seen to have the same image $Q\subset \mathbb{P}(V_6)$.
According to \cite{kuznetsov}, $Q$ is a quartic hypersurface which
is singular along a  curve and $\phi$, $\psi$ are two small
resolutions of $Q$. In particular, $\mathbb{P}(\mathcal{E}_{14})$
and $\mathbb{P}(\mathcal{E}_{3})$ are birational, and even related
by a flop $\theta=\psi^{-1}\circ \phi$.

We will need the following lemma: Denote by $\Delta_{14}$, resp.
$\Delta_{3}$, the general fiber of the projective bundle
$\pi_{\bV_{14}}:\mathbb{P}(\mathcal{E}_{14})\rightarrow \bV_{14}$, resp.
$\pi_{Y_3}:\mathbb{P}(\mathcal{E}_{3})\rightarrow Y_{3}$. Via the birational
map $\theta$, these curves (which do not meet the indeterminacy
locus of $\theta$ or $\theta^{-1}$) can be seen  as curves either in
$\mathbb{P}(\mathcal{E}_{14})$ or in $\mathbb{P}(\mathcal{E}_{3})$,
and we will denote by ${\rm deg}_{\bV_{14}}(\cdot)$, resp. ${\rm
deg}_{Y_3}(\cdot)$ the degree of their projection in $\bV_{14}$, resp.
$Y_3$ with respect to the canonical polarizations. We will also denote by ${\rm deg}_Q(\cdot)$ the degree of
their projections in $Q\subset\mathbb{P}^5$ via $\phi$, resp. $\psi$.
\begin{lemma}\label{ledegVYQ} One has ${\rm deg}_Q\Delta_3={\rm
deg}_Q\Delta_{14}=1$ and \begin{eqnarray}\label{eqdeg28dec}{\rm
deg}_{\bV_{14}}\Delta_{14}=0,\,{\rm deg}_{Y_3}\Delta_{14}=4,\\
\nonumber {\rm deg}_{Y_{3}}\Delta_{3}=0,\,{\rm
deg}_{\bV_{14}}\Delta_{3}=4.
\end{eqnarray}
\end{lemma}
\begin{proof} Let $\omega\in Y_3$ be a general point. The fiber
$\Delta_3$ over $\omega$ is thus by definition the projective line
$L_\omega=\mathbb{P}({\rm Ker}\,\omega)$. Its image $C_\omega$ in $\bV_{14}$ is
the set of points $[l]$ in $\bV_{14}$ such that the corresponding projective line $l$ meets
$L_\omega$.
As $L_\omega=\mathbb{P}({\rm Ker}\,\omega)$, one has
$\omega_{\mid V_l}=0$ for any line $l$ meeting $L_\omega$. Thus
$C_\omega$ is the vanishing locus of the $4$-dimensional space
$W_5/\langle\omega\rangle$  of Pl\"ucker equations
 on the set of lines in $V_6$
meeting $L_\omega$. It is easily proved that this locus has degree
$4$ and this proves the last statement. Next let $[l]\in \bV_{14}$ be
a general point. Let $C_l$ be the image in $Y_3$ of the fiber
$\Delta_{14}$ over $[l]$. Then $C_l$ is the set of $\alpha\in W_5$
whose kernel intersects $V_l$ nontrivially, and the map $l\mapsto C_l$ associates to $x\in l$
the unique form $\omega\in \mathbb{P}(W_5)$ such that $x\in \mathbb{P}({\rm Ker}\,\omega)$. As all forms in $W_5$ vanish on $l$,
the natural morphism
$$V_6\otimes \mathcal{O}_l(-1)\rightarrow W_5^*\otimes
\mathcal{O}_l$$ of vector bundles over $l=\mathbb{P}^1$ factors
through $(V_6/V_l)\otimes \mathcal{O}_l(-1)$, and its cokernel has
thus degree $4$, which proves the second statement in
(\ref{ledegVYQ}). The other statements are immediate.
\end{proof}

\subsection{The cubic $3$-fold side.\label{subseccubic28dec}} We know by \cite{markutikho},
\cite{ilievmarku} that the moduli space of stable vector bundles of rank
$2$ on $Y_3$ with $c_1=0$ and ${\rm deg}_{Y_3}c_2=2$ is birationally isomorphic to
$J(Y_3)$ via the Abel-Jacobi map. Let $E$ be such a vector bundle.
\begin{lemma}\label{lecurveY3}  If $E$ is general,  then:

(i)   ${\rm dim}\,H^0(Y_3, E^*\otimes \mathcal{E}_3)=4$.

(ii) Furthermore, if $(Y_3,E)$ is general, for a general section $\sigma\in H^0(Y_3,
E^*\otimes \mathcal{E}_3)$, we have an exact sequence
\begin{eqnarray}\label{eaxctY3}
0\rightarrow E\stackrel{\sigma}{\rightarrow}
\mathcal{E}_3\rightarrow \mathcal{O}_S(C)\rightarrow 0,
\end{eqnarray}
where $S\in|\mathcal{O}_{Y_3}(2)|$ is a smooth K3 surface and
$C\subset S$ is a degree $9$, genus $5$ curve.
\end{lemma}
\begin{proof} (i) Note that $E\cong E^*$, so (i) is equivalent to  ${\rm dim}\,H^0(Y_3, E\otimes \mathcal{E}_3)=4$. We now
specialize to the case where $E$ is not locally free, namely
$E=\mathcal{I}_{l_1}\oplus \mathcal{I}_{l_2}$, where $l_1,\,l_2$ are
two general lines in $Y_3$. Then $h^0(Y_3, \mathcal{E}_3\otimes
\mathcal{I}_{l_i})=2$ for $i=1,\,2$ and
$H^p(Y_3,\mathcal{E}_3\otimes \mathcal{I}_{l_i})=0$ for $i=1,\,2$
and $p>0$, which implies the result for general $E$ by standard
deformation arguments.

(ii) As $E$ is stable, for any $0\not=\sigma\in H^0(Y_3, E^*\otimes
\mathcal{E}_3)$, the generic rank of $\sigma:E\rightarrow
\mathcal{E}_3$ must be $2$, so that we have the exact sequence
(\ref{eaxctY3}). We claim that the surface $S$ defined by the
vanishing of the determinant of $\sigma$ is smooth so that, in
particular, the rank of $\sigma$ is $1$ along $S$ and ${\rm
Coker}\,\sigma$ is a line bundle on $S$.  We clearly have
$S\in|\mathcal{O}_{Y_3}(2)|$ so that assuming the claim, $S$ is a
smooth K3 surface. Finally the exact sequence (\ref{eaxctY3}) and
the fact that ${\rm deg}_{Y_3}c_2(E)=2$, ${\rm
deg}_{Y_3}c_2(\mathcal{E}_3)=5$ immediately imply that the degree of
$C$ is $9$. To conclude, we observe that
$H^0(Y_3,E)=0,\,H^1(Y_3,E)=0$ and the exact sequence (\ref{eaxctY3})
give $H^0(S,\mathcal{O}_S(C))=6$, that is, $g(C)=5$.

 We now prove the claim. As the smoothness of the degeneracy surface $S$ is
an open property (on the moduli space of $4$-uples $(Y_3,E_3, E,\sigma)$ with
${\rm dim}\,{\rm Hom}\,(E,E_3)=4$, which is known to be irreducible by (i) and
 \cite{markutikho},
\cite{ilievmarku}), it suffices to show that the conclusion holds for at least one $4$-uple satisfying the property that ${\rm dim}\,{\rm Hom}\,(E,E_3)=4$.
It turns out that the result is true generically even in the degenerate case where
$E=\mathcal{I}_{l_1}\oplus \mathcal{I}_{l_2}$. This is proved by working more generally on the
universal Pfaffian cubic $Y_{Pf}$ in $\mathbb{P}^{14}$, of which $Y_3$ is a generic linear section.
One easily checks that given two general lines $l_1,\,l_2$ contained in $Y_{Pf}$,
and two sections $\sigma$, resp. $\tau$ of $\mathcal{E}_{Pf}\otimes \mathcal{I}_{l_1}$, resp.
$\mathcal{E}_{Pf}\otimes \mathcal{I}_{l_2}$, the quadratic equation $q=\sigma\wedge\tau \in H^0(Y_{Pf},{\rm det}\,\mathcal{E}_{Pf})=H^0(Y_{Pf},\mathcal{O}_{Y_{Pf}}(2))$ is of the form $\omega\mapsto \omega^2_{\mid W_4}$ for some
$4$-dimensional subspace $W_4\subset V_6$, hence defines a rank $6$ quadric $Q$ on $\mathbb{P}^{14}$. Comparing the differentials of  the defining equations for
$Y_{Pf}$ and $Q$, one then concludes that
$Q\cap Y_{Pf}$ is  smooth away from the set of $\omega$'s whose kernel is contained in the
 codimension $2$ linear subspace $W_4\subset W_6$, which has codimension $4$ in
 $Y_{Pf}$. The generic linear section $Y_3\subset Y_{Pf}$ thus intersects $Q\cap Y_{Pf}$ along a smooth surface.
\end{proof}
Note that if we restrict (\ref{eaxctY3}) to $S$, and then  to $C$,
we conclude, using  $\mathcal{O}_S(C)_{\mid C}=K_C$, that there is a
surjective morphism
\begin{eqnarray}\label{eqsurmorY3}
\mathcal{E}_{3\mid C}\rightarrow K_C,
\end{eqnarray}
hence a section
$$\phi_\sigma: C\rightarrow \mathbb{P}(\mathcal{E}_{3\mid C})$$
such that
$\phi_{\sigma}^*\mathcal{O}_{\mathbb{P}(\mathcal{E}_3)}(1)=K_C$.
This morphism induces a surjection
$H^0(\mathbb{P}(\mathcal{E}_3),\mathcal{O}_{\mathbb{P}(\mathcal{E}_3)}(1))\rightarrow
H^0(C,K_C)$ since both maps $H^0(Y_3,\mathcal{E}_3)\rightarrow
H^0(S,\mathcal{O}_S(C))$ and $H^0(S,\mathcal{O}_S(C))\rightarrow
H^0(C,K_C)$ are surjective. As $H^0(Q,\mathcal{O}_Q(1))=H^0(\mathbb{P}(\mathcal{E}_3),\mathcal{O}_{\mathbb{P}(\mathcal{E}_3)}(1))$, we get as well  a surjective
map $H^0(Q,\mathcal{O}_Q(1))\rightarrow H^0(C,K_C)$. Thus the image
of $C$ in $Q\subset \mathbb{P}(V_6)$ via $\psi\circ \phi_\sigma$ is
a linearly normal canonical genus $5$ curve.

The following lemma shows conversely how to recover the bundle from the curve $C\subset  \mathbb{P}(\mathcal{E}_3)$:
\begin{lemma}\label{lecurveY3reconsE} Let $C\subset\mathbb{P}(\mathcal{E}_3)$ be a general genus $5$
curve such that the image $C'$ of $C$ in $Y_3$  has degree $9$, the morphism
$C\rightarrow C'$ is an isomorphism, and the image of $C$ in $Q\subset \mathbb{P}^5$ via
$\psi\circ \phi_\sigma$ is a linearly normal canonical curve.  Then there exists a
unique stable vector bundle $E$ on $Y_3$ with $c_1=0,\,{\rm
deg}_{Y_3}c_2=2$ and a unique $\sigma \in H^0(Y_3,E^*\otimes
\mathcal{E}_3)$ determining the line bundle $\mathcal{O}_S(C)$ as in Lemma \ref{lecurveY3}.
\end{lemma}
Here ``general'' means that there is a component of the Hilbert scheme of such curves
on which the conclusion holds generically. In fact, this component is the one containing
the curves $\phi_\sigma(C)$ appearing in Lemma \ref{lecurveY3} and its proof.
\begin{proof} The curve $C'\subset Y_3$ has degree $9$ and genus
$5$. Assuming the Abel-Jacobi class of $C'$ is general in $J(Y_3)$, then (1)
 $C'$ is not contained in a hyperplane section of $Y_3$ and (2)
$C'$ is contained in a unique surface $S\subset Y_3$, where
$S$ is a member of $|\mathcal{O}_{Y_3}(2)|$. Indeed, curves contained in a
hyperplane section of $Y_3$ cannot have a general Abel-Jacobi class,
by Lemma \ref{leinfiJac} below. This proves the first statement. We
have $h^0(Y_3,\mathcal{O}_{Y_3}(2))=15$ and
$h^0(C',\mathcal{O}_{Y_3}(2)_{\mid C'})=14$, thus $C'$ is contained
in at least one quadric section of $X$. If $C'$ is contained in two
surfaces $S,\,S'$ as above, then as $S$ and $S'$ have no common
component by the first statement, $C'$ is a component of the
complete intersection $S\cap S'$ which has degree $12$. Thus $C'$ is
 residual to a degree $3$ curve, and its Abel-Jacobi point is, up
to a sign and a constant, the Abel-Jacobi point of a degree $3$
curve which again by Lemma \ref{leinfiJac} cannot be general in
$J(Y_3)$. This proves the second statement. Note that, according
to Lemma \ref{lecurveY3} (ii), the surface $S$ is smooth for general $C$
in the considered
 component of the Hilbert scheme of $\mathbb{P}(\mathcal{E}_3)$.  We now observe that the
restriction map $H^0(S, \mathcal{E}^*_{3\mid S}(C'))\rightarrow
H^0(C', \mathcal{E}^*_{3\mid C'}(K_{C'}))$ is surjective. This indeed
follows from the fact that $H^1(S,\mathcal{E}^*_{3\mid S})=0$ (see
\cite{kuznetsov}). Let $\sigma'\in H^0(S, \mathcal{E}^*_{3\mid
S}(C'))$ be  a lift of the natural section $\sigma\in H^0(C',
\mathcal{E}^*_{3\mid C'}(K_{C'}))$ giving the embedding of $C'$ in
$\mathbb{P}(\mathcal{E}_3)$ with image $C$. We have
${\rm dim}\,H^0(Y_3,\mathcal{E}_3)=6={\rm dim}\,H^0(S,
\mathcal{O}_S(C))$ and thus the property that the map
$$\sigma':H^0(Y_3,\mathcal{E}_3)\rightarrow H^0(S,
\mathcal{O}_S(C))$$
is an isomorphism is an open property. Furthermore, the line bundle $\mathcal{O}_S(C)$ is generically globally generated, and we thus conclude that
for generic $C$, we get a surjective morphism
$$\sigma':\mathcal{E}_{3\mid
S}\rightarrow \mathcal{O}_S(C'),$$
hence as well a surjective morphism
$$\mathcal{E}_3\rightarrow \mathcal{O}_S(C')$$
of sheaves on $Y_3$. Its kernel provides the desired bundle $E$. That $E$ is stable follows from
$H^0(Y_3,E)=0$, which is a consequence of the surjectivity, hence injectivity,
of the map
 $H^0(Y_3, \mathcal{E}_3)\rightarrow
H^0(S, \mathcal{O}_S(C))$ which implies that
$H^0(Y_3,E)=0$.
\end{proof}
We used above the following lemma:
\begin{lemma}\label{leinfiJac} Let $M$ be a smooth variety and
$Z\subset M\times Y_3$ be a codimension $2$ subvariety. Assume that
for general $m\in M$, the curve $Z_m\subset Y_3$ is contained in a
hyperplane section of $Y_3$. Then the Abel-Jacobi map
$\Phi_Z:M\rightarrow J(Y_3)$ is not dominating.
\end{lemma}
\begin{proof} This immediately follows from the fact that the
transpose of the differential of the Abel-Jacobi map
$$d\Phi_Z^*: \Omega_{J(Y_3),0}\rightarrow \Omega_{M,m}$$
factors through the restriction map
$$H^0(Y_3,\mathcal{O}_{Y_3}(1))\rightarrow
H^0(Z_m,\mathcal{O}_{Z_m}(1)),$$ where one uses the natural
identification (see section \ref{seccasecub})
$$\Omega_{J(Y_3),0}\cong H^0(Y_3,\mathcal{O}_{Y_3}(1)).$$
Our assumptions thus say that $d\Phi_Z^*$ is not injective so
$\Phi_Z$ is nowhere a submersion on $M$.
\end{proof}

 The two lemmas \ref{lecurveY3} and \ref{lecurveY3reconsE} together  show that a component of the family
of genus $5$ curves in $\mathbb{P}(\mathcal{E}_3)$, of $Y_3$-degree
$9$ and $Q$-degree $8$ is birationally  a $\mathbb{P}^5$-bundle over a $\mathbb{P}^3$-bundle over a Zariski open set of
$J(Y_3)$, which is itself birational to a moduli space of rank $2$ vector bundles on $Y_3$.
\begin{remark}{We believe that the $\mathbb{P}^3$-bundle is not Zariski locally trivial, that is, is not the projectivization of a
vector bundle,  over any Zariski open set of $J(Y_3)$.}
\end{remark}

 \subsection{The $\bV_{14}$ side. \label{subsecV1428dec}} Recall that  $\bV_{14}\subset G(2,6)$ denote
  a smooth  $3$-dimensional linear section of
  $G(2,6)$. We are going to study  degree $13$, genus $5$ curves  $C\subset \bV_{14}$
  such that $h^0(\mathcal{E}_{14\mid C})=6$. Then
 by Riemann--Roch,
   there is a nonzero morphism
 \begin{eqnarray}\label{eqsigmaC142}\sigma_C: \mathcal{E}_{14\mid C}\rightarrow K_C.
 \end{eqnarray}
 When the morphism is surjective, it provides
  a section
 $\phi_\sigma:C\rightarrow \mathbb{P}(\mathcal{E}_{14})$ such that
 $\phi_\sigma^*(\mathcal{O}_{\mathbb{P}(\mathcal{E}_{14})}(1))=K_C$.
 Counting dimensions from the viewpoint of
 genus $5$ curves equipped with  a semi-stable rank $2$ vector bundle
 $\mathcal{E}$ of degree $13$
 with  $h^0(C,\mathcal{E})=6$, we see that  the general such triple
 $(C,\mathcal{E},\sigma_C)$ corresponds to a morphism $\sigma_C$ which is surjective, which we will assume
 from now on.

  We  have:
 \begin{lemma} (i) Such curves $C\subset \bV_{14}$ exist for a general  smooth $\bV_{14}\subset G(2,6)$.

 (ii)  Let $\mathcal{L}$ be the Pl\"ucker line
 bundle on $G(2,6)$. Then ${\rm dim}\, H^0(C,\mathcal{L}_{\mid C})=9$ and
 the restriction map $H^0(\bV_{14},\mathcal{L})\rightarrow H^0(C,\mathcal{L}_{\mid
 C})$ is surjective. Hence $C$ is contained in exactly one K3 surface $S\in
 |\mathcal{L}|$.

 (iii)  The surface  $S$ is smooth.
 \end{lemma}
 \begin{proof}
   Note that $\bV_{14}$ contains a line $\Delta$. Let now $S_0\subset \bV_{14}$ be a $K3$ surface hyperplane section of $\bV_{14}$ containing $\Delta$ and having as only singularity a node $x_0$ which is not on $\Delta$.
   Let $\widetilde{S}_0$ be the desingularization of $S_0$ by blowing-up
   $x_0$. Then $\widetilde{S}_0$ contains in its   Picard lattice the subgroup generated by the
   classes $\mathcal{L}_S,\,\Delta,\,e$, where $e$ is the class of the exceptional curve.
   The intersection numbers are
   $$\mathcal{L}_S^2=14,\,\mathcal{L}_S\cdot\Delta=1,\,\mathcal{L}_S\cdot e=0,\,e^2=\Delta^2=-2,\,e\cdot\Delta=0.$$
   It thus follows that the curves $\widetilde{C}$ in  $|\mathcal{L}_S-\Delta-e|$ have genus $5$ and Pl\"ucker degree $13$.
   One easily checks that the general such curve $\widetilde{C}$ satisfies
  $ h^0(\widetilde{C},n^*\mathcal{E}_{14})=6$, where $n:\widetilde{C}\rightarrow S_0\subset G(2,6)$ is the natural map.
  We now deform  the surface $S_0$ to a smooth surface $
  S_t$ in $\bV_{14}$ on which the class
 $e+\Delta$ remains algebraic. Then the Picard lattice of the general such smoothification
 is generated by $\mathcal{L}_{S_t}$ and $e+\Delta$, and the class $e+\Delta$ is not effective anymore on
 $S_t$. Hence we also have $H^1(S_t,\mathcal{O}_{S_t}(e+\Delta))=0$.
   It thus follows that
 the curves $C$ in
 $|\mathcal{L}_{S_t}(-e-\Delta)|$ have the property that the restriction map
 $$H^0(S_t,\mathcal{L}_{S_t})\rightarrow H^0(C,\mathcal{L}_{\mid C})$$
 is an isomorphism.
 This proves the three statements except for smoothness of the general curves
 $C\subset S_t$ which follows from the fact that the line bundle $\mathcal{L}_{S_t}(-e-\Delta)$
 is  nef. (This was also true for the curves $\widetilde{C}$ on the surface $\widetilde{S}_0$ but as it was not embedded
 in $\bV_{14}$, the resulting curves $n(\widetilde{C})$ were nodal.)
 \end{proof}
 Assuming  the curve $C\subset \bV_{14}$ is general and thus satisfies the properties above, we now
 compute that $\chi(S,\mathcal{E}_{14\mid S}^*(C))=1$ and thus
 either $H^0(S,\mathcal{E}_{14\mid S}^*(C))\not=0$ or $H^0(S,\mathcal{E}_{14\mid
 S}(-C))\not=0$. As we have $H^0(S,\mathcal{L}(-C))=0$, and $\mathcal{L}={\rm det}\,\mathcal{E}_{14}$, the second
 case is excluded so that we have a nonzero morphism
 $\sigma:\mathcal{E}_{14}\rightarrow \mathcal{O}_S(C)$, extending the morphism $\sigma_C$
 of (\ref{eqsigmaC142}). We now compute
 $$c_2(\mathcal{E}_{14\mid S}^*(C))=c_2(\mathcal{E}_{14\mid S}^*)-{\rm deg}_{\bV_{14}}C+C^2=5-13+8=0.$$
 It thus follows that either $\sigma$ vanishes nowhere on $S$, or $\sigma$ vanishes along a curve
 in $S$ which does not meet $C$.  The second  case can only occur if $\rho(S)\geq3$
 while a dimension count shows that the family of surfaces
 $S$ appearing in this construction has dimension $8$, so that the  generically
 $S$ satisfies
  $\rho(S)=2$. Hence  $\sigma$ is everywhere surjective and  we thus
get a rank $2$ vector bundle $E$ on $\bV_{14}$ with trivial
determinant  fitting in the exact sequence
\begin{eqnarray}
\label{eqexa29dec}0\rightarrow E\rightarrow
\mathcal{E}_{14}\rightarrow \mathcal{O}_S(C)\rightarrow 0.
\end{eqnarray} One easily computes that  ${\rm deg}_{\bV_{14}}(c_2(E))=4$.
The following will be useful:
\begin{lemma} \label{lestable} The restriction of $E$ to a smooth hyperplane section
$\Sigma\subset \bV_{14}$ is a rank $2$ vector bundle on $\Sigma$ with
trivial determinant and $c_2=4$. If  $\rho(\Sigma)=1$,
$E_{\mid \Sigma}$ is stable.
\end{lemma}
\begin{proof} The first statement is obvious. The stability follows
 from the vanishing  $H^0(\Sigma, E_{\mid \Sigma})=0$
which is implied by  $H^0(\bV_{14},E)=0$ and  $H^1(\bV_{14},E\otimes \mathcal{L}^{-1})=0$, which are both implied by the exact sequence (\ref{eqexa29dec}) and the fact that $(C-c_1(\mathcal{L}))^2=-4$ on $\Sigma$, and
$H^0(\Sigma,\mathcal{L}^{-1}(C))=H^2(\Sigma,\mathcal{L}^{-1}(C))=0$, implying that  $H^2(\Sigma,\mathcal{L}^{-1}(C))=0$.
\end{proof}
 \begin{remark}\label{rema29dec} {\rm
Note that the vector bundle $E$ constructed above from the data of
the K3 surface $S$ and the line bundle $\mathcal{O}_S(C)$
satisfies ${\rm dim}\,{\rm Hom}(E,\mathcal{E}_{14})=4$ as easily
follows from (\ref{eqexa29dec}). It follows that the
$13$-dimensional family of genus $5$, degree $13$ curves on $\bV_{14}$
corresponds in fact to a $5$-dimensional family of vector bundles on
$\bV_{14}$.}
\end{remark}
\subsection{Construction of the rational map\label{secconstbirat}}
 We now make the following construction: Let $(Y_3,\mathcal{E}_3)$ be a general cubic threefold with Pfaffian structure, and let $E$ be a
general rank $2$ stable vector bundle  on $Y_3$ with $c_1(E)=0$ and ${\rm
deg}_{Y_3}c_2(E)=2$. By Lemma \ref{lecurveY3}, there is an associated
$8$-dimensional family of genus $5$ curves $C\subset
\mathbb{P}(\mathcal{E}_3)$ satisfying ${\rm deg}_{Y_3}(C)=9,\,{\rm
deg}_{Q}(C)=8$ such that the image $\psi(C)\subset Q$ is a linearly
normal canonical curve of genus $5$. We claim that the curve
$\psi(C)$ is the
complete intersection of $3$ quadrics in $\mathbb{P}^4$: For this,
we have to show that $C$ is not trigonal. However, from our
construction, we see that $C$ is contained in a general K3 surface
$S$ with Picard lattice generated by $h$ and the class $c$ of $C$,
with intersection lattice
$$h^2=6,\,h\cdot c=9,\,c^2=8.$$
the fact that $C$ is not trigonal then follows from \cite{greenlaz}.
As $\psi(C)\subset Q$ is the complete intersection of three quadrics
in a hyperplane section $H\cap  Q$ of $Q$, we can write
\begin{eqnarray}\label{eqqH}q_{\mid H}=s_1q_1+s_2q_2+s_3q_3,
\end{eqnarray} where $q$ is the defining equation for $Q$, and the $q_i$'s are
the defining equations for $\psi(C)$. Here the $s_i$'s are also
quadratic polynomials on $H$. It follows that $Q$ contains another
set of canonical curves of genus $5$, namely, viewing the expression
in the right hand side of  (\ref{eqqH}) as a quadric in the $6$
variables $q_i,\,s_i,\,i=1,\,2,\,3$, the plane defined by the
$q_i$'s determines one ruling of this quadric (these planes are
parametrized  by a $\mathbb{P}^3$) and the planes in the other
rulings will correspond to a second $\mathbb{P}^3$ of linearly
normal degree $8$ genus $5$ canonical curves in $Q\cap H$.
Concretely the curve $C_1\subset Q\cap H$ defined in $H$  by $q_1=q_2=s_3=0$ is such a
curve. The important point for us is that the original curve $C$ is
a general member of a linear system $|\mathcal{O}_S(C)|$ on a K3
surface $S\subset \mathbb{P}(\mathcal{E}_3)$, hence it does not meet
the surface $\Sigma_3\subset \mathbb{P}(\mathcal{E}_3)$  which is
contracted by $\psi$, which means that $\psi(C)$ does not meet the
singular curve of $Q$. The residual curve  $C_1$ constructed above
thus moves freely in $Q$ and also avoids the singular locus of $Q$
which is the indeterminacy locus of the rational map $\phi^{-1}$.
Thus it lifts to a curve $C'=\phi^{-1}(C_1)\subset
\mathbb{P}(\mathcal{E}_{14})$.
\begin{lemma}\label{ledegC1} The genus $5$ curve $C'$ satisfies ${\rm
deg}_{\bV_{14}}C'=13$.
\end{lemma}
\begin{proof} The rank $2$ vector space
$A^3(\mathbb{P}(\mathcal{E}_{14}))={\rm
Hdg}^6(\mathbb{P}(\mathcal{E}_{14}))=A^3(\mathbb{P}(\mathcal{E}_{3}))$
of curve classes in either of these two varieties is generated by the
classes $[\Delta_3]$ and $[\Delta_{14}]$. We can thus  write  in
this space
 $$[C]=\alpha[\Delta_3]+\beta[\Delta_{14}].
 $$
 Next, as the curve $C_1=\phi(C')\subset Q$ is residual to $\psi(C)$
 in the complete intersection in $Q$ of a hyperplane $H$ and two
 quadrics, we get that $[C']=4h_Q^3-[C]$ in $A^3(\mathbb{P}(\mathcal{E}_{14}))$, where $h_Q$ is the
 pull-back to $\mathbb{P}(\mathcal{E}_{14})$ of
 $c_1(\mathcal{O}_Q(1))$. Note that we also have
 $h_Q=c_1(\mathcal{O}_{\mathbb{P}(\mathcal{E}_{14})}(1))$.
 It thus follows that
 \begin{eqnarray}\label{eqinter28dec}{\rm deg}_{\bV_{14}}C'=4{\rm deg}_{\bV_{14}}h_Q^3-{\rm
 deg}_{\bV_{14}}C.
 \end{eqnarray}
 As $h_Q=c_1(\mathcal{O}_{\mathbb{P}(\mathcal{E}_{14})}(1))$, the
 standard theory of Chern classes (see \cite{fulton}) says that
 $\pi_{\bV_{14}*}(h_Q^3)=s_2(\mathcal{E}_{14})=c_1^2(\mathcal{E}_{14})-c_2(\mathcal{E}_{14})$
 and thus
 ${\rm deg}_{\bV_{14}}(h_Q^3)=14-5=9$.
 Next we have ${\rm deg}_{Y_3}(C)=9,\,{\rm deg}_Q(C)=8$, which by
 Lemma \ref{ledegVYQ} gives
 $$ \alpha+\beta=8,\,\,4\beta=9.$$
 We thus deduce that $4\alpha=23$. This finally gives using
 (\ref{eqinter28dec}):
 $${\rm
deg}_{\bV_{14}}C'=36-4\alpha=36-23=13.$$
\end{proof}
By a dimension count (or by the reversibility of the construction), we observe that for generic $C$,  the genus $5$, $\bV_{14}$-degree $13$, $Q$-degree $8$
curve $C'$ is generic in
$\mathbb{P}(\mathcal{E}_{14}))$ and we can thus apply the construction of Section \ref{subsecV1428dec}
to get from $C'$ a stable rank $2$ vector bundle $E'$ on $\bV_{14}$ with
trivial determinant and ${\rm deg}_{\bV_{14}}c_2(E')=4$.

\begin{proof}[Proof of Theorem \ref{theobirat}] Let $X$ be a general
Pfaffian cubic fourfold and $u:\mathcal{Y}_U\rightarrow U$ be the
universal family of smooth hyperplane sections of $X$. The general fiber of $u$ is thus a
general cubic threefold with Pfaffian structure. For each
point $t\in U$, there is a canonical morphism from the moduli space
$\mathcal{M}_{2,0,2,t}$ of rank $2$ vector bundles on
${Y_t}$ with trivial determinant and ${\rm
deg}_{{Y_t}}c_2=2$ to the intermediate Jacobian
$J({Y_t})$ which maps $E$ to
$\Phi_{{Y_t}}(c_2(E)-c_2(\mathcal{E}_{3,t}(-1)))$. Here
$\mathcal{E}_{3,t}$ denotes the restriction to ${Y_t}$ of
the Pfaffian vector bundle $\mathcal{E}$ on $X$. This morphism is
birational by \cite{ilievmarku}. This way we conclude that the
moduli space $\mathcal{M}$ of sheaves on $X$ supported on a
hyperplane section and with the same numerical data as $E$ (seen as
a sheaf on $X$) is birational to $\mathcal{J}_U$, where
 $\mathcal{J}_U\rightarrow U$ is the family of intermediate
 Jacobians.
 \begin{remark}{\rm  For general $X$, this birational isomorphism
 does not exist, or rather takes values in a torsor under
 $\mathcal{J}$.}
 \end{remark}
On the other hand, we also have the universal family
$$v:\mathcal{V}_{14,U}\rightarrow U$$
of corresponding linear sections of the Grassmannian. For each $t\in
U$ corresponding to a $W_{5,t}\subset W_6\subset \bigwedge^2V_6^*$,
the fiber $\mathcal{V}_{14,t}$ is the complete intersection of
$G(2,V_6)$ with $5$ Pl\"ucker hypersurfaces defined by $W_{5,t}$. We
thus have a natural inclusion $\Sigma\subset \mathcal{V}_{14,t}$ as
a Pl\"ucker hypersurface since by definition $\Sigma\subset
G(2,V_6)$ is   the vanishing locus in  $G(2,V_6)$ of the $6$
Pl\"ucker equations defined by $W_6$.

The construction described above done in
family over $U$ now gives us the following: There exists a smooth
projective variety $W$ which admits a morphism $g:W\rightarrow
\mathcal{M}\cong_{birat}\mathcal{J}_U $ with rationally connected
fibers and a rational map
\be \label{rational map!}
f:W\dashrightarrow \mathcal{M}_{2,0,4}(\Sigma).
\ee
The general point of  $W$  parametrizes  a general rank $2$
vector bundle $E$  with $c_1=0$ and ${\rm deg}_{{Y_t}}c_2=2$
on a fiber $Y_t$ of $\mathcal{Y}_U$,  the choice of a general nonzero
morphism $\sigma:E\rightarrow \mathcal{E}_{3,t}$ defined up to a
coefficient,  a general member $C$ of the linear system
$|\mathcal{O}_S(C)|$, where $\mathcal{O}_S(C)$ is defined by the
exact sequence (\ref{eaxctY3}), and a general   $(2,2,2,1)$ complete
intersection curve $C'$ contained in $Q_t$, residual in $Q_t$ to the $(2,2,2,1)$
complete intersection curve $\psi_t(C)\subset Q_t$ (we will see in
fact $C'$ as living in $\mathbb{P}(\mathcal{E}_{14,t})$). Thus the
general fiber of the map $g$ has dimension $10$ and $W$ has
dimension $20=10+10$. The rational map $f$ associates then to these
data the vector bundle $E'_{\mid \Sigma}$, which is stable by Lemma \ref{lestable}, where the vector bundle
$E'=E_{S',C'}$ on $\mathcal{V}_{14,t}$ is associated as in Section \ref{subsecV1428dec} to the
 curve $C'$. Here $S'$ is the generically unique
Pl\"ucker hyperplane section of $\mathcal{V}_{14,t}$ containing
$C'$.
\begin{remark}{\rm One
easily checks that $E'$ does not depend on the choice of $C$ or  the
residual curve $C'$. The only reason to introduce these curves was
the fact that they do not meet the  singular locus of $Q_t$, which
is not true for the associated K3 surfaces where they lie. It will
also appear below that $E'$ neither depends on the choice of
$\sigma$.}
\end{remark}
The proof of Theorem \ref{theobirat} will be completed using  the
following:
\begin{prop} \label{propfactor} The rational map $f$ factors through $\mathcal{J}_U$ and induces
 a birational isomorphism
$$g:\mathcal{J}_U\dashrightarrow\mathcal{M}_{2,0,4}(\Sigma).$$
\end{prop}
\begin{proof} We know by  Theorem \ref{theoconst}
 that  $\mathcal{J}_U$ is not uniruled (this is indeed implied by
the fact  that a smooth projective completion of $\mathcal{J}_U$ admits a generically
nondegenerate holomorphic $2$-form). It follows that $\mathcal{J}_U$
is birational to the basis of the maximal rationally connected fibration of
$W$. We now have the following lemma:
\begin{lemma}\label{lesurj28dec} (i) The rational map $f$ is
dominating.

(ii) The general fiber of $f$ is rationally connected.
\end{lemma}

\begin{proof} (i) The variety $W$ has two holomorphic (in fact
algebraic) $2$-forms, namely the pull-back
$\widetilde{\sigma_{\mathcal{J}}}$ to $W$ of the holomorphic $2$-form
${\sigma}_{\mathcal{J}}$ on $\mathcal{J}_U$ constructed in Theorem
\ref{theoconst}, and the form $f^*\sigma_{\mathcal{M}_{2,0,4}}$. We
claim that for some $\lambda\not=0$
\begin{eqnarray}\label{eqtwoform}
\widetilde{\sigma_{\mathcal{J}}}=\lambda
f^*\sigma_{\mathcal{M}_{2,0,4}}.
\end{eqnarray}
This equation  immediately implies the surjectivity of $f$, since
the generic rank of $\widetilde{\sigma_{\mathcal{J}}}$ is equal to ${\rm
dim}\,\mathcal{J}_U=10$ and the rank of
$f^*\sigma_{\mathcal{M}_{2,0,4}}$ is not greater than the rank of
$f$, so the equality (\ref{eqtwoform}) implies that the generic rank
of $f$ is $10={\rm dim}\,\mathcal{M}_{2,0,4}(\Sigma)$, implying that $f$
is dominant. We prove now the claim: Note that $W$ is a fibration
over $\mathcal{J}_U$ (or rather a smooth projective
compactification $\overline{\mathcal{J}}$ of $\mathcal{J}_U$)  with
rationally connected general fiber, hence
$H^{2,0}(W)=H^{2,0}(\overline{\mathcal{J}})$ is of dimension $1$ by
Theorem \ref{theoconst}(iii). As
$\widetilde{\sigma_{\mathcal{J}}}\not=0$, it thus suffices to prove that
$f^*\sigma_{\mathcal{M}_{2,0,4}}\not=0$. This can be proved by a
Chow-theoretic argument using Mumford's theorem \cite{mumford0}.
Indeed, there is a natural inclusion
$$\Sigma\times U\subset \mathcal{V}_{14,U}$$
which is the restriction over $U$ of the natural inclusion
$$j:\Sigma\times \mathbb{P}^5\subset \mathcal{V}_{14},$$
where $\mathcal{V}_{14}$ is the universal family of
$\bV_{14}$-threefolds containing $\Sigma$. It is immediate to check
that $pr_{1*}\circ j^*:{\rm CH}_1(\mathcal{V}_{14})_{hom}\rightarrow
{\rm CH}_0(\Sigma)_{hom}$ (which is just the restriction map $j_t^*$
on each ${\rm CH}_1(\mathcal{V}_{14,t})_{hom}$, where $j_t$ is the
inclusion of $\Sigma$ in $\mathcal{V}_{14,t}$), is an isomorphism.
On the other hand, $\mathcal{V}_{14}$ is birationally equivalent to
the universal family $\mathcal{Y}$ of hyperplane sections of $X$,
which satisfies ${\rm CH}_1(\mathcal{Y})_{hom}\cong {\rm
CH}_1(X)_{hom}$ since ${\rm CH}_0(X)_{hom}=0$ and $\mathcal{Y}$ is a
projective bundle over $X$. This fibered birational isomorphism
 induces an isomorphism between the intermediate Jacobian
fibrations over $U$ (see \cite{ilievmarku}) hence a fiberwise
isomorphism
\begin{eqnarray}
\label{eqisojac}{\rm CH}_1(\mathcal{V}_{14,t})_{hom}\cong {\rm
CH}_1({Y_t})_{hom}, \end{eqnarray} since for rationally
connected threefolds $Y$, the Abel-Jacobi map ${\rm
CH}_1(Y)_{hom}\rightarrow J(Y)$ is an isomorphism (see
\cite{blochsrinivas}). This easily implies that
$${\rm CH}_1(\mathcal{V}_{14})_{\mathbb{Q},hom}\cong {\rm
CH}_1(\mathcal{Y})_{\mathbb{Q},hom}$$ since ${\rm
CH}_0(\mathcal{V}_{14,t})_{hom}=0,\,{\rm
CH}_0({Y_t})_{hom}=0$.
 We now observe that
each point $w$ of fiber $W_t$ of the variety $W$ over $U$
parametrizes vector bundles $E'_w$, resp. $E_w$, on fibers
$\mathcal{V}_{14,t}$, resp. ${Y_t}$, and that for each $t\in
U$, the two maps
$$c_V:W_t\rightarrow {\rm CH}_1(\mathcal{V}_{14,t}),\,\,c_Y: W\rightarrow {\rm
CH}_1({Y_t}),$$
$$c_V(w)=c_2(E'_w),\,\,c_Y(w)=c_2(E_w)$$
coincide up to sign and a constant via (\ref{eqisojac}). With the notation above,
this follows from the construction for a given $t$ of the curve $C'$ as residual to
the curve $C$ in a $(2,2,1)$ complete intersection in $Q_t$. Combining
these observations, we conclude that the map
$${\rm CH}_0(W)_{hom}\rightarrow {\rm CH}_0(\Sigma)_{hom},$$
$$w\mapsto c_2(E'_{w\mid \Sigma})$$
is surjective, hence by Mumford's theorem \cite{mumford0} that the
corresponding pull-back of the holomorphic $2$-form on $\Sigma$ is
nonzero. However, by construction of the holomorphic $2$-form on
$\mathcal{M}_{2,0,4}(\Sigma)$, this pull-back is nothing but
$f^*\sigma_{\mathcal{M}_{2,0,4}}$.

 (ii) Let $E$ be a general stable rank $2$ vector bundle on $\Sigma$ with trivial determinant
 and $c_2=4$.  The fiber of $f$ over $E$ essentially consists of vector bundles
 $E_{S',C',t}$ on threefolds $\mathcal{V}_{14,t}$ containing $\Sigma$ such that
 $$E_{S',C',t\mid \Sigma}\cong E.$$
More precisely, for each such vector bundle, one can apply the results of
Section \ref{subsecV1428dec}: choosing a general section
$\tilde{\sigma}$ of ${\rm Hom}(E_{S',C',t},\mathcal{E}_{14,t})$ one gets a degeneracy K3 surface
$S_{\tilde{\sigma}}\subset \mathcal{V}_{14,t}$ and a line bundle $\mathcal{O}_{S_{\tilde{\sigma}}}(D)$
on
 $S_{\tilde{\sigma}}$ which is a quotient of $\mathcal{E}_{14,t\mid S_{\tilde{\sigma}}}$,  providing  a section $\phi_{\tilde{\sigma}}:S_{\tilde{\sigma}}\rightarrow \mathbb{P}(\mathcal{E}_{14,t})$.
 For a general curve $D_0\in |\mathcal{O}_{S_{\tilde{\sigma}}}(D)|$, the lifted curve $\phi_{\tilde{\sigma}}(D_0)$ is of genus
 $5$, $Q$-degree $8$ and $\bV_{14}$-degree $13$ and
 using the map $\phi:\mathbb{P}(\mathcal{E}_{14,t})\rightarrow Q_t$, it provides a complete
 intersection curves of type $(1,2,2,2)$ contained in $Q_t$ and the residual curve $D'_0$ in a
 $(2,2,1)$ complete intersection of $Q_t$ provides a genus $5$ curves in $\mathbb{P}(\mathcal{E}_{3,t})$
 of $Q$-degree $8$ and $Y_3$-degree $9$. Applying Lemma \ref{lecurveY3reconsE}, we then reconstructed an element of $W$ with image $E$ under $f$.  Given the vector bundle $E_{S',C',t}$ on
 a threefold $\mathcal{V}_{14,t}$, the extra  data described above, namely the choices
 of $\sigma$ and of the curves $D_0,\,D'_0$, are parametrized  by a rationally connected variety
 so the proof will be finished once we  know that  $E_{S',C',t}$ is
 determined by $E$.
\begin{lemma}\label{sublemma} Let $E_{S',C',t}$ be a general rank $2$ vector bundle on
$\mathcal{V}_{14,t}$ constructed as in Section \ref{subsecV1428dec}
and let $E$ be its restriction to $\Sigma$. Then the restriction map
$${\rm Hom}({E}_{S',C',t},\mathcal{E}_{14,t})\rightarrow {\rm
Hom}(E,\mathcal{E}_{14\mid \Sigma})$$ is an isomorphism, where
$\mathcal{E}_{14\mid \Sigma}$ denotes the Pl\"ucker rank $2$ vector
bundle restricted to  $\Sigma$. \end{lemma} \begin{proof} Indeed, the injectivity
is obvious and on the other hand, both sides have dimension $4$.
This was already proved in Remark \ref{rema29dec} for the left hand
side. For the right hand side, we can specialize the general vector
bundle $E$   to the case where $E^*=\mathcal{I}_{z_1}\oplus
\mathcal{I}_{z_2}$ where $z_1,\,z_2$ are two length two subschemes
of $\Sigma$; then $H^0(\Sigma, E^*\otimes
\mathcal{E}_{14\mid \Sigma})=H^0(\Sigma,\mathcal{E}_{14\mid \Sigma}\otimes
\mathcal{I}_{z_1})\oplus H^0(\Sigma,\mathcal{E}_{14\mid \Sigma}\otimes
\mathcal{I}_{z_2})$ has dimension $4$ while $H^i(\Sigma, E^*\otimes
\mathcal{E}_{14\mid \Sigma})=H^i(\Sigma,\mathcal{E}_{14\mid \Sigma}\otimes
\mathcal{I}_{z_1})\oplus H^i(\Sigma,\mathcal{E}_{14\mid \Sigma}\otimes
\mathcal{I}_{z_2})=0$ for $i>0$. The conclusion then follows from a
deformation argument.\end{proof}

Let now $\sigma\in{\rm Hom}(E,\mathcal{E}_{14\mid \Sigma})$ be a general section. Then
we get a degeneracy curve $D_\sigma\in|\mathcal{L}_{\mid \Sigma}|$, where
$\mathcal{L}$ is the Pl\"ucker line bundle, and an exact sequence
$$0\rightarrow E\rightarrow \mathcal{E}_{14\mid \Sigma}\rightarrow \mathcal{O}_{D_\sigma}(Z)\rightarrow 0$$
where $Z$ is a divisor of degree $13$ on $D_\sigma$. This gives a section
$\phi_\sigma:D_\sigma\rightarrow \mathbb{P}(\mathcal{E}_{14})$ with image $\widetilde{D_\sigma}$.
For each vector bundle $E_{S',C',t}$ on some $\mathcal{V}_{14,t}\supset \Sigma$ restricting to $E$ on
$\Sigma$, the section $\sigma$ extends to a section
$\tilde{\sigma}$ by Lemma \ref{sublemma}, and thus there is
a K3 surface
$$\widetilde{S_{\tilde{\sigma}}}:=\phi_{\tilde{\sigma}}(S_{\tilde{\sigma}})\subset \mathbb{P}(\mathcal{E}_{14}),$$  which
is a lift of the degeneracy surface $S_{\tilde{\sigma}}\subset \mathcal{V}_{14,t}$. The surface
$\widetilde{S_{\tilde{\sigma}}}$
  intersects $\mathbb{P}(\mathcal{E}_{14\mid \Sigma})$ along the curve $\widetilde{D_\sigma}=\phi_\sigma(D_\sigma)$. The surface $S_{\tilde{\sigma}}$ carries a line bundle
  $\mathcal{O}_{S_{\tilde{\sigma}}}(C_\sigma)$ which restricts to $\mathcal{O}_{D\sigma}(Z)$ on
  $D_\sigma$. Note that the curve ${D_\sigma}\subset S_{\tilde{\sigma}}$ is a member
  of $|\mathcal{L}_{\mid S_{\tilde{\sigma}}}|$.
The uniqueness of $E_{S',C',t}$ then follows from the results of Section \ref{subsecV1428dec} and
from the following:
\begin{lemma}  For a general curve
$\widetilde{D_\sigma}\subset \mathbb{P}(\mathcal{E}_{14})$ as above, there exists a unique surface $\widetilde{S_{\tilde{\sigma}}}\subset \mathbb{P}(\mathcal{E}_{14})$ satisfying the conditions above, that is, lifting a $K3$ surface
in some $V_{14,t}$ containing $\Sigma$ and intersecting $\mathbb{P}(\mathcal{E}_{14\mid \Sigma})$ along the curve $\widetilde{D_\sigma}$.
\end{lemma}
\begin{proof} Let $N_{\widetilde{D_\sigma}/\mathbb{P}(\mathcal{E}_{14})}$ be the normal bundle
of $ \widetilde{D_\sigma}$ in $ \mathbb{P}(\mathcal{E}_{14})$. There is an exact sequence
\begin{eqnarray}\label{eqnormalbundle1jan} 0\rightarrow {T_{\mathbb{P}(\mathcal{E}_{14})/G(2,6)}}_{\mid \widetilde{D_\sigma}}\rightarrow N_{\widetilde{D_\sigma}/\mathbb{P}(\mathcal{E}_{14})}\rightarrow
N_{D_\sigma/G(2,6)}=\mathcal{L}_{\mid D_\sigma}^7\rightarrow 0,
\end{eqnarray}
and each surface $\widetilde{S_{\tilde{\sigma}}}$ extending $\widetilde{D_\sigma}$ as above
provides an inclusion
$$N_{D_\sigma/S_{\tilde{\sigma}}}=\mathcal{L}_{\mid D_\sigma}\subset N_{\widetilde{D_\sigma}/\mathbb{P}(\mathcal{E}_{14})},$$
or equivalently a nonzero section of $N_{\widetilde{D_\sigma}/\mathbb{P}(\mathcal{E}_{14})}\otimes \mathcal{L}^{-1}$. It is not hard to see that this section determines the surface $\widetilde{S_{\tilde{\sigma}}}$, so we only have to prove that, for general
$\widetilde{D_\sigma}$ as above:
\begin{eqnarray}\label{eqsecnorm1jan} h^0(\widetilde{D_\sigma},N_{\widetilde{D_\sigma}/\mathbb{P}(\mathcal{E}_{14})}\otimes \mathcal{L}^{-1})=1.
\end{eqnarray}
In order to prove (\ref{eqsecnorm1jan}), we write the normal bundle sequence twisted by $\mathcal{L}^{-1}$ for
$\widetilde{D_\sigma}\subset \widetilde{S_{\tilde{\sigma}}}\subset \mathbb{P}(\mathcal{E}_{14})$. This gives
$$0\rightarrow \mathcal{O}_{D_\sigma}\rightarrow N_{\widetilde{D_\sigma}/\mathbb{P}(\mathcal{E}_{14})}\otimes \mathcal{L}^{-1}\rightarrow
({N_{\widetilde{S_{\tilde{\sigma}}}/ \mathbb{P}(\mathcal{E}_{14})}})_{\mid \widetilde{D_\sigma}}\otimes \mathcal{L}^{-1}\rightarrow 0$$
and (\ref{eqsecnorm1jan}) will follow from $h^0(\widetilde{D_\sigma},({N_{\widetilde{S_{\tilde{\sigma}}}/ \mathbb{P}(\mathcal{E}_{14})}})_{\mid \widetilde{D_\sigma}}\otimes \mathcal{L}^{-1})=0$ which itself will be a consequence of
\begin{eqnarray} \label{eqsurSigmavan} h^0(\widetilde{S_{\tilde{\sigma}}},N_{\widetilde{S_{\tilde{\sigma}}}/ \mathbb{P}(\mathcal{E}_{14})}\otimes \mathcal{L}^{-1})=0,\\ \nonumber
h^1(\widetilde{S_{\tilde{\sigma}}},N_{\widetilde{S_{\tilde{\sigma}}}/ \mathbb{P}(\mathcal{E}_{14})}\otimes \mathcal{L}^{-2})=0.
\end{eqnarray}
The second vanishing statement is obtained by writing the normal bundle sequence (\ref{eqnormalbundle1jan}) for
$\widetilde{S_{\tilde{\sigma}}}$:
\begin{eqnarray}0\rightarrow ({T_{\mathbb{P}(\mathcal{E}_{14})/G(2,6)}})_{\mid \widetilde{S_{\tilde{\sigma}}}}\rightarrow N_{\widetilde{S_{\tilde{\sigma}}}/\mathbb{P}(\mathcal{E}_{14})}\rightarrow
N_{S_{\tilde{\sigma}}/G(2,6)}=\mathcal{L}_{\mid S_{\tilde{\sigma}}}^6\rightarrow 0, \end{eqnarray}
where
the line bundle $({T_{\mathbb{P}(\mathcal{E}_{14})/G(2,6)}})_{\mid \widetilde{S_{\tilde{\sigma}}}}$
is isomorphic to $\mathcal{L}^{-1}_{\mid S_{\tilde{\sigma}}}(2C_\sigma)$. One then concludes using
$$H^1(S_{\tilde{\sigma}},\mathcal{L}_{\mid S_{\tilde{\sigma}}}^{-1})=0,\,H^1(S_{\tilde{\sigma}},\mathcal{L}^{-3}_{\mid S_{\tilde{\sigma}}}(2C_\sigma))=0,$$
which both follow from standard vanishing theorems on the K3 surface $S_{\tilde{\sigma}}$.
It remains to prove the first vanishing statement. However, according to
 Section \ref{subsecV1428dec}, the deformation space of $\widetilde{S_{\tilde{\sigma}}}$
in $\mathbb{P}(\mathcal{E}_{14})$ is smooth and isomorphic to a $\mathbb{P}^3$-bundle over
the $10$-dimensional moduli space of sheaves on $\mathcal{V}_{14}$ supported on fibers $\mathcal{V}_{14,t}$ and
are locally free on $\mathcal{V}_{14,t}$ of rank $2$, with trivial determinant and ${\rm deg}_{\bV_{14}}c_2=4$. By (i) and Lemma \ref{sublemma}, the
restriction map $(E_t,\tilde{\sigma})\rightarrow (E_{t\mid\Sigma},\tilde{\sigma}_{\mid \Sigma})$ has generically surjective differential, hence also injective differential. It is clear however that sections of
$N_{\widetilde{S_{\tilde{\sigma}}}/ \mathbb{P}(\mathcal{E}_{14})}\otimes \mathcal{L}^{-1}$, seen as sections of
$N_{\widetilde{S_{\tilde{\sigma}}}/ \mathbb{P}(\mathcal{E}_{14})}$ vanishing on $\widetilde{D_\sigma}$, belong to the kernel of
this differential. Hence they must be trivial.
\end{proof}
This concludes the proof of Lemma \ref{lesurj28dec}.
\end{proof}
Lemma \ref{lesurj28dec} implies Proposition \ref{propfactor} as
follows: since $f$ is dominating by (i) and
$\mathcal{M}_{2,0,4}(\Sigma)$ is not uniruled, $f$ must factor
through the MRC fibration of $W$, that is, through $\mathcal{J}_U$.
The general fiber of the induced rational map
$g:\mathcal{J}_U\dashrightarrow \mathcal{M}_{2,0,4}(\Sigma)$ are
then rationally connected by (ii). But as $\mathcal{J}_U$ is not
uniruled, the general fiber of $g$ is a point, so $g$ is birational.
\end{proof}
The proof of Theorem \ref{theobirat} is now finished.
\end{proof}

 \bibliography{hkref}
\end{document}